\documentclass[11pt,twoside,reqno]{amsart}
\usepackage[colorlinks=true,citecolor=blue]{hyperref}

\numberwithin{equation}{section}


\usepackage{amsmath, amssymb, amsthm, amsfonts, cite, dsfont, enumerate, epsfig, float, geometry, doi,
	infwarerr, mathrsfs, mathtools, mathrsfs, mathtools, relsize, stmaryrd, tabularx, txfonts, nicefrac, subfig}
\usepackage[normalem]{ulem}

\newtheorem{theorem}{Theorem}[section]
\newtheorem{lemma}[theorem]{Lemma}
\newtheorem{proposition}[theorem]{Proposition}

\newtheorem{conjecture}[theorem]{Conjecture}
\newtheorem{corollary}[theorem]{Corollary}
\newtheorem{question}[theorem]{Question}

\theoremstyle{definition}
\newtheorem{definition}[theorem]{Definition}
\newtheorem{notation}[theorem]{Notation}

\newtheorem{example}[theorem]{Example}
\newtheorem{remark}[theorem]{Remark}
\numberwithin{equation}{section}

\newcommand{\doilink}[1]{\href{https://doi.org/#1}{#1}}

\newcommand{\naturals}{\ensuremath{\mathbb{N}}}

\newcommand{\Reals}{\ensuremath{\mathbb{R}}}

\newcommand{\set}{\ensuremath{\mathcal}}
\newcommand{\cset}[1]{\mathcal{#1}^{\textnormal{c}}} % complement of a set

\newcommand{\OneTo}[1]{[#1]}
\newcommand{\card}[1]{|#1|}
\newcommand{\bigcard}[1]{\bigl|#1\bigr|}

\DeclareMathOperator{\Vertex}{\mathsf{V}}
\DeclareMathOperator{\Edge}{\mathsf{E}}
\DeclareMathOperator{\Adjacency}{\mathbf{A}}
\DeclareMathOperator{\Laplacian}{\mathbf{L}}
\DeclareMathOperator{\SignlessLaplacian}{\mathbf{Q}}
\DeclareMathOperator{\AllOne}{\mathbf{J}}
\DeclareMathOperator{\Identity}{\mathbf{I}}

\DeclareMathOperator{\Complete}{\mathsf{K}}
\DeclareMathOperator{\Friendship}{\mathsf{F}}

\DeclareMathOperator{\Path}{\mathsf{P}}
\DeclareMathOperator{\Cycle}{\mathsf{C}}
\DeclareMathOperator{\SRG}{\mathsf{srg}}

\DeclareMathOperator{\Star}{\mathsf{S}}
\DeclareMathOperator{\Clique}{\omega}
\DeclareMathOperator{\Chromatic}{\chi}

\newcommand{\Gr}[1]{\mathsf{#1}}                          % graph
\newcommand{\CGr}[1]{\overline{\mathsf{#1}}}              % complement graph
\newcommand{\V}[1]{\Vertex(#1)}                           % vertex set of a graph
\newcommand{\E}[1]{\Edge(#1)}                             % edge set of a graph
\newcommand{\A}{\Adjacency}                               % adjacency matrix of a graph
\newcommand{\LM}{\Laplacian}                              % Laplacian matrix of a graph
\newcommand{\Q}{\SignlessLaplacian}                       % signless Laplacian matrix of a graph
\newcommand{\D}{\mathbf{D}}                               % diagonal matrix whose entries on the diagonal are the degrees of the vertices
\newcommand{\J}[1]{\AllOne_{#1}}                          % square all-ones matrix
\newcommand{\I}[1]{\Identity_{#1}}                        % identity matrix
\newcommand{\indnum}[1]{\alpha(#1)}                       % independence number

\newcommand{\clnum}[1]{\Clique(#1)}                       % clique number

\newcommand{\chrnum}[1]{\Chromatic(#1)}                   % chromatic number

\newcommand{\Eigval}[2]{\lambda_{#1}(#2)}  % k-th eigenvalue of the adjacency matrix of a graph
\newcommand{\CoG}[1]{\Complete_{#1}}       % complete graph
\newcommand{\CoBG}[2]{\Complete_{#1,#2}}   % Complete bipartite graph
\newcommand{\FG}[1]{\Friendship_{#1}}      % Friendship graph
\newcommand{\GFG}[2]{\Friendship_{#1,#2}}  % Generalized frienship graph
\newcommand{\PathG}[1]{\Path_{#1}}         % path of given length
\newcommand{\CG}[1]{\Cycle_{#1}}           % cycle graph of a given length
\newcommand{\SG}[1]{\Star_{#1}}            % star graph with given number of leaves
\newcommand{\srg}[4]{\SRG(#1,#2,#3,#4)}    % strongly regular graph

\newcommand{\DU}{\hspace{0.1em} \dot{\cup} \hspace{0.1em}}  % disjoint union
\newcommand{\NS}{\, \underline{\vee} \,}   % operation of neighbor splitting of join of graphs.
\newcommand{\NNS}{\, \uuline{\vee} \,}     % operation of non-neighbor splitting of join of graphs.
\newcommand{\DuplicationGraph}[1]{\mathrm{Du}(#1)}                                                                        % Duplication graph
\newcommand{\Corona}[2]{{#1} \circ {#2}}                                                                                  % Corona of two graphs
\newcommand{\EdgeCorona}[2]{{#1} \diamondsuit \hspace*{0.03cm} {#2}}                                                      % Edge corona of two graphs
\newcommand{\DuplicationCorona}[2]{{#1} \boxminus {#2}}                                                                   % Duplication corona
\newcommand{\DuplicationEdgeCorona}[2]{{#1} \boxplus {#2}}                                                                % Duplication edge corona
\newcommand{\ClosedNeighborhoodCorona}[2]{{#1} \, \underline{\boxtimes} \, {#2}}                                          % Closed neighborhood corona
\newcommand{\SubdivisionGraph}[1]{\mathrm{S}(#1)}                                                                         % Subdivision graph
\newcommand{\BipartiteIncidenceGraph}[1]{\mathrm{B}(#1)}                                                                  % Bipartite incidence graph
\newcommand{\SVBVJ}[2]{\SubdivisionGraph{#1} \, \ddot{\vee} \, \BipartiteIncidenceGraph{#2}}                              % Subdivision-vertex-bipartite-vertex join graph
\newcommand{\SEBEJ}[2]{\SubdivisionGraph{#1} \, \overset{\raisebox{-0.1cm}{$=$}}{\vee} \, \BipartiteIncidenceGraph{#2}}   % Subdivision-edge-bipartite-edge join graph

\newcommand{\SEBVJ}[2]{\SubdivisionGraph{#1} \, {\overset{\raisebox{-0.25cm}{$\stackrel{\rule{0.25cm}{0.05mm}}{\cdot}$}}{\vee}} \hspace*{0.1cm} \BipartiteIncidenceGraph{#2}}
\newcommand{\SVBEJ}[2]{\SubdivisionGraph{#1} \, {\overset{\raisebox{0.0cm}{$\stackrel{\cdot}{\rule{0.25cm}{0.05mm}}$}}{\vee}} \hspace*{0.1cm} \BipartiteIncidenceGraph{#2}}
\newcommand{\trace}[1]{\text{Tr}{(#1)}}
\newcommand{\Gmats}{\{\A,\LM,\Q, {\bf{\mathcal{L}}}, \overline{\A}, \overline{\LM}, \overline{\Q}, \overline{{\bf{\mathcal{L}}}}\} }
\newcommand{\AM}[1]{\text{AM}{(#1)}}
\newcommand{\GM}[1]{\text{GM}{(#1)}}
\newcommand{\diag}[1]{\operatorname{diag}\bigl(#1\bigr)}

\MHInternalSyntaxOn
\renewcommand{\dcases}
{
	\MT_start_cases:nnnn
	{\quad}
	{$\m@th\displaystyle##$\hfil}
	{$\m@th\displaystyle##$\hfil}
	{\lbrace}
}
\MHInternalSyntaxOff

\makeatletter\c@MaxMatrixCols=15\makeatother

\begin{document}
\setlength{\baselineskip}{1.15\baselineskip}

\title{On Spectral Graph Determination}

\author{Igal Sason  \and Noam Krupnik \and Suleiman Hamud \and Abraham Berman}

\maketitle
\thispagestyle{empty}

\vspace*{-0.8cm}
\begin{center}
{\em Technion - Israel Institute of Technology, Technion City, Haifa 3200003, Israel}
\end{center}

\vskip 4mm {\noindent {\bf Abstract.}
The study of spectral graph determination is a fascinating area of research in spectral graph
theory and algebraic combinatorics. This field focuses on examining the spectral characterization
of various classes of graphs, developing methods to construct or distinguish cospectral nonisomorphic
graphs, and analyzing the conditions under which a graph's spectrum uniquely determines its structure.
This paper presents an overview of both classical and recent advancements in these topics, along
with newly obtained proofs of some existing results, which offer additional insights.

\vspace*{0.2cm}

\vspace*{0.2cm}

\vspace*{0.2cm}
\noindent {\bf Correspondence}: Igal Sason, Technion - Israel Institute of Technology, Technion City,
Haifa 3200003, Israel. Email: eeigal@technion.ac.il; Tel: +97248294699.

\tableofcontents{}

\section{Introduction}
\label{section: Introduction}

Spectral graph theory lies at the intersection of combinatorics and matrix theory, exploring the structural
and combinatorial properties of graphs through the analysis of the eigenvalues and eigenvectors of matrices
associated with these graphs \cite{BrouwerH2011,Chung1997,CvetkovicDS1995,CvetkovicRS2010,GodsilR2001}.
Spectral properties of graphs offer powerful insights into a variety of useful graph characteristics, enabling
the determination or estimation of features such as the independence number, clique number, chromatic number,
and the Shannon capacity of graphs, which are notoriously NP-hard to compute.

A particularly intriguing topic in spectral graph theory is the study of cospectral graphs, i.e., graphs that
share identical multisets of eigenvalues with respect to one or more matrix representations. While isomorphic
graphs are always cospectral, non-isomorphic graphs may also share spectra, leading to the study of non-isomorphic
cospectral (NICS) graphs. This phenomenon raises profound questions about the extent to which a graph’s spectrum
encodes its structural properties. Conversely, graphs determined by their spectrum (DS graphs) are uniquely
identifiable, up to isomorphism, by their eigenvalues. In other words, a graph is DS if and only if no other
non-isomorphic graph shares the same spectrum.

The problem of spectral graph determination and the characterization of DS graphs dates back to the pioneering
1956 paper by G\"{u}nthard and Primas \cite{GunthardP56}, which explored the interplay between graph theory and
chemistry. This paper posed the question of whether graphs can be uniquely determined by their spectra with respect
to their adjacency matrix $\A$.

While every graph can be determined by its adjacency matrix, which enables the determination
of every graph by its eigenvalues and a basis of corresponding eigenvectors, the characterization of graphs
for which eigenvalues alone suffice for identification forms a fertile area of research in spectral graph theory.
This research holds both theoretical interest and practical implications.

Subsequent studies have broadened the scope of this question to include determination by the spectra of other
significant matrices, such as the Laplacian matrix ($\LM$), signless Laplacian matrix ($\Q$), and normalized
Laplacian matrix (${\bf{\mathcal{L}}}$), among many other matrices associated with graphs.
The study of cospectral and DS graphs with respect to these matrices has become a cornerstone of spectral graph
theory. This line of research has far-reaching applications in diverse fields, including chemistry and molecular
structure analysis, physics and quantum computing, network communication theory, machine learning, and data science.

One of the most prominent conjectures in this area is Haemers' conjecture \cite{Haemers2016,Haemers2024}, which
posits that most graphs are determined by the spectrum of their adjacency matrices ($\A$-DS). Despite many efforts
in proving this open conjecture, some theoretical and experimental progress on the theme of this conjecture has been
recently presented in \cite{KovalK2024,WangW2024}, while also graphs or graph families that are not DS continue to be
discovered. Haemers’ conjecture has spurred significant interest in classifying DS graphs and understanding the factors
that influence spectral determination, particularly among special families of graphs such as regular graphs, strongly
regular graphs, trees, graphs of pyramids, as well as the construction of NICS graphs by a variety of graph operations.
Studies in these directions of research have been covered in the seminal works by Schwenk \cite{Schwenk1973}, and by van
Dam and Haemers \cite{vanDamH03,vanDamH09}, as well as in more recent studies (in part by the authors) such as
\cite{AbdianBTKO21,AbiadH2012,AbiadBBCGV2022,Butler2010,ButlerJ2011,BermanCCLZ2018,Butler2016,ButlerH2016,BuZ2012,BuZ2012b,
CamaraH14,DasP2013,DuttaA20,GodsilM1982,HamidzadeK2010,HamudB24,JinZ2014,KannanPW22,KoolenHI2016,KoolenHI2016b,KovalK2024,KrupnikB2024,
LinLX2019,LiuZG2008,MaRen2010,OboudiAAB2021,OmidiT2007,OmidiV2010,Sason2024,YeLS2025,ZhangLY09,ZhangLZY09,ZhouBu2012}, and references therein.
Specific contributions of these papers to the problem of the spectral determination of graphs are addressed in the continuation of this article.

This paper surveys both classical and recent results on spectral graph determination, also presenting newly obtained
proofs of some existing results, which offer additional insights.

The paper emphasizes the significance of adjacency spectra ($\A$-spectra), and it provides conditions for $\A$-cospectrality, $\A$-NICS, and
$\A$-DS graphs, offering examples that support or refute Haemers’ conjecture.
We furthermore address the cospectrality of graphs with respect to the Laplacian, signless Laplacian, and normalized Laplacian matrices. For regular
graphs, cospectrality with respect to any one of these matrices (or the adjacency matrix) implies cospectrality with respect to all the others,
enabling a unified framework for studying DS and NICS graphs across different matrix representations. However, for irregular graphs, cospectrality
with respect to one matrix does not necessarily imply cospectrality with respect to another. This distinction underscores the complexity of
analyzing spectral properties in irregular graphs, where the interplay among different matrix representations becomes more intricate and often
necessitates distinct techniques for characterization and comparison.

The structure of the paper is as follows: Section~\ref{section: preliminaries} provides preliminary material in matrix theory, graph theory,
and graph-associated matrices. Section~\ref{section: DS graphs} focuses on graphs determined by their spectra (with respect to one or
multiple matrices). Section~\ref{section: special families of graphs} examines special families of graphs and their determination by
adjacency spectra. Section~\ref{section: graph operations} analyzes unitary and binary graph operations, emphasizing their impact on
spectral determination and construction of NICS graphs. Finally, Section~\ref{section: summary and outlook} concludes the paper with
open questions and an outlook on spectral graph determination, highlighting areas for further research.

\section{Preliminaries}
\label{section: preliminaries}
The present section provides preliminary material and notation in matrix theory, graph theory,
and graph-associated matrices, which serves for the presentation of this paper.

\subsection{Matrix Theory Preliminaries}
\label{subsection: Matrix Theory Preliminaries}

The following standard notation in matrix theory is used in this paper:
\begin{itemize}
	\item $\Reals^{n\times m}$ denotes the set of all $n \times m$ matrices with real entries,
	\item $\Reals^{n} \triangleq \Reals^{n\times 1}$ denotes the set of all $n$-dimensional column vectors with real entries,
	\item $\I{n}\in\Reals^{n\times n}$ denotes the $n \times n$ identity matrix,
    \item $\mathbf{0}_{k,m} \in\Reals^{k\times m}$ denotes the $k \times m$ all-zero matrix,
    \item $\J{k,m}\in\Reals^{k\times m}$ denotes the $k \times m$ all-ones matrix,
	\item $\mathbf{1}_n \triangleq \J{n,1} \in \Reals^n$ denotes the $n$-dimensional column vector of ones.
\end{itemize}

Throughout this paper, we deal with real matrices.

The concepts of \emph{Schur complement} and \emph{interlacing of eigenvalues} are useful in papers on
spectral graph determination and cospectral graphs, and are also addressed in this paper.

\begin{definition}
\label{definition: Schur complement}
Let $\mathbf{M}$ be a block matrix
\begin{align}
\mathbf{M}=
\begin{pmatrix}
\mathbf{A} & \mathbf{B}\\
\mathbf{C} & \mathbf{D}
\end{pmatrix},
\end{align}
where the block $\mathbf{D}$ is invertible. The \emph{Schur complement of $D$ in $M$} is
\begin{align}
\label{definition: eq - Schur complement}
\mathbf{M/D}= \mathbf{A}-\mathbf{BD}^{-1}\mathbf{C}.
\end{align}
\end{definition}

Schur proved the following remarkable theorem:
\begin{theorem}[Theorem on the Schur complement \cite{Schur1917}]
\label{theorem: Schur complement}
If $D$ is invertible, then
\begin{align}
\label{eq: Schur's formula}
\det{\mathbf{M}} & =\det(\mathbf{M/D}) \, \det{\mathbf{D}}.
\end{align}
\end{theorem}

\begin{theorem}[Cauchy Interlacing Theorem \cite{ParlettB1998}]
\label{thm:interlacing}
Let $\lambda_{1} \ge \ldots \ge \lambda_{n}$ be the eigenvalues of a symmetric matrix $\mathbf{M}$
and let $\mu_{1}\ge\ldots\ge\mu_{m}$ be the eigenvalues of a \emph{principal $m \times m$ submatrix
of $\mathbf{M}$} (i.e., a submatrix that is obtained by deleting the same set of rows and columns
from $M$) then $\lambda_{i}\ge\mu_{i}\ge\lambda_{n-m+i}$ for $i=1,\ldots,m$.
\end{theorem}

\begin{definition}[Completely Positive Matrices]
\label{definition: completely positive matrix}
A matrix $\A \in \Reals^{n \times n}$ is called {\em completely positive} if there exists a matrix
${\mathbf{B}} \in \Reals^{n \times m}$ whose all entries are nonnegative such that $\A = {\mathbf{B}} {\mathbf{B}}^\mathrm{T}$.
\end{definition}
A completely positive matrix is therefore symmetric and all its entries are nonnegative. The interested reader
is referred to the textbook \cite{ShakedBbook19} on completely positive matrices, also addressing their connections
to graph theory.

\begin{definition}[Positive Semidefinite Matrices]
\label{definition: positive semidefinite matrix}
A matrix $\A \in \Reals^{n \times n}$ is called {\em positive semidefinite} if $\A$ is symmetric, and
the inequality $\underline{x}^{\mathrm{T}} \A \underline{x} \geq 0$ holds for every column vector $\underline{x} \in \Reals^n$.
\end{definition}
\begin{proposition}
\label{proposition: positive semidefinite matrix}
A symmetric matrix is positive semidefinite if and only if one of the following conditions hold:
\begin{enumerate}
\item All its eigenvalues are nonnegative (real) numbers.
\item There exists a matrix ${\mathbf{B}} \in \Reals^{n \times m}$ such that $\A = {\mathbf{B}} {\mathbf{B}}^\mathrm{T}$.
\end{enumerate}
\end{proposition}

The next result readily follows.
\begin{corollary}
\label{corollary: c.p. yields p.s.}
A completely positive matrix is positive semidefinite.
\end{corollary}
\begin{remark}
\label{remark: matrix of order 5}
Regarding Corollary~\ref{corollary: c.p. yields p.s.}, it is natural to ask whether, under certain conditions, a positive
semidefinite matrix whose all entries are nonnegative is also completely positive. By \cite[Theorem~3.35]{ShakedBbook19}, this
holds for all square matrices of order $n \leq 4$. Moreover, \cite[Example~3.45]{ShakedBbook19} also presents an explicit example
of a matrix of order~5 that is positive semidefinite with all nonnegative entries but is not completely positive.
\end{remark}

\subsection{Graph Theory Preliminaries}
\label{subsection: Graph Theory Preliminaries}

A graph $\Gr{G} = (\V{\Gr{G}}, \E{\Gr{G}})$ forms a pair
where $\V{\Gr{G}}$ is a set of vertices and $\E{\Gr{G}}\subseteq \V{\Gr{G}} \times \V{\Gr{G}}$
is a set of edges.

In this paper all the graphs are assumed to be
\begin{itemize}
	\item {\em finite} - $\bigcard{\V{\Gr{G}}}<\infty$,
	\item {\em simple} - $\Gr{G}$ has no parallel edges and no self loops,
	\item {\em undirected} - the edges in $\Gr{G}$ are undirected.
\end{itemize}

We use the following terminology:
\begin{itemize}
\item The {\em degree}, $d(v)$, of a vertex $v\in \V{\Gr{G}}$ is the number of vertices in $\Gr{G}$ that are adjacent to $v$.
\item A {\em walk} in a graph $\Gr{G}$ is a sequence of vertices in $\Gr{G}$, where every two consecutive vertices in the sequence
are adjacent in $\Gr{G}$.
\item A {\em path} in a graph is a walk with no repeated vertices.
\item A {\em cycle} $\Cycle$ is a closed walk, obtained by adding an edge to a path in $\Gr{G}$.
\item The {\em length of a path or a cycle} is equal to its number of edges. A {\em triangle} is a cycle of length~3.
\item A {\em connected graph} is a graph in which every pair of distinct vertices is connected by a path.
\item The {\em distance} between two vertices in a connected graph is the length of a shortest path that connects them.
\item The {\em diameter} of a connected graph is the maximum distance between any two vertices in the graph, and the diameter
of a disconnected graph is set to be infinity.
\item The {\em connected component} of a vertex $v \in \V{\Gr{G}}$ is the subgraph whose vertex set $\set{U} \subseteq \V{\Gr{G}}$
consists of all the vertices that are connected to $v$ by any path (including the vertex $v$ itself), and its edge set consists of
all the edges in $\E{\Gr{G}}$ whose two endpoints are contained in the vertex set $\set{U}$.
\item A {\em tree} is a connected graph that has no cycles (i.e., it is a connected and {\em acyclic} graph).
\item A {\em spanning tree} of a connected graph $\Gr{G}$ is a tree with the vertex set $\V{\Gr{G}}$ and some of the edges of~$\Gr{G}$.
\item A graph is {\em regular} if all its vertices have the same degree.
\item A {\em $d$-regular} graph is a regular graph whose all vertices have degree $d$.
\item A {\em bipartite graph} is a graph $\Gr{G}$ whose vertex set is a disjoint union of two subsets such that no two vertices in
the same subset are adjacent.
\item A {\em complete bipartite graph} is a bipartite graph where every vertex in each of the two partite sets is adjacent to all
the vertices in the other partite set.
\end{itemize}

\begin{definition}[Complement of a graph]
The \emph{complement} of a graph $\Gr{G}$, denoted by $\CGr{G}$, is a graph whose vertex set is $\V{\Gr{G}}$, and its edge set
is the complement set $\CGr{\E{\Gr{G}}}$.
Every vertex in $\V{\Gr{G}}$ is nonadjacent to itself in $\Gr{G}$ and $\CGr{G}$, so $\{i,j\} \in \E{\CGr{G}}$ if and only if
$\{i, j\} \notin \E{\Gr{G}}$ with $i \neq j$.
\end{definition}

\begin{definition}[Disjoint union of graphs]
\label{def:disjoint_union_graphs}
Let $\Gr{G}_1, \ldots, \Gr{G}_k$ be graphs.
If the vertex sets in these graphs are not pairwise disjoint,
let $\Gr{G}'_2, \ldots, \Gr{G}'_k$ be isomorphic copies of
$\Gr{G}_2, \ldots, \Gr{G}_k$, respectively, such that none
of the graphs $\Gr{G}_1, \Gr{G}'_2, \ldots \Gr{G}'_k$ have
a vertex in common. The disjoint union of these graphs,
denoted by $\Gr{G} = \Gr{G}_1 \DU \ldots \DU \Gr{G}_k$,
is a graph whose vertex and edge sets are equal to the
disjoint unions of the vertex and edge sets of
$\Gr{G}_1, \Gr{G}'_2, \ldots, \Gr{G}'_k$
($\Gr{G}$ is defined up to an isomorphism).
\end{definition}

\begin{definition}
Let $k\in \naturals$ and let $\Gr{G}$ be a graph. Define
$k \Gr{G} = \Gr{G} \DU \Gr{G} \DU \ldots \DU \Gr{G}$
to be the disjoint union of $k$ copies of $\Gr{G}$.
\end{definition}

\begin{definition}[Join of graphs]
\label{definition: join of graphs}
Let $\Gr{G}$ and $\Gr{H}$ be two graphs with disjoint vertex sets.
The join of $\Gr{G}$ and $\Gr{H}$ is defined to be their disjoint union,
together with all the edges that connect the vertices in $\Gr{G}$ with
the vertices in $\Gr{H}$. It is denoted by $\Gr{G} \vee \Gr{H}$.
\end{definition}

\begin{definition}[Induced subgraphs]
\label{definition: Induced subgraphs}
Let $\Gr{G}=(\Vertex,\Edge)$ be a graph, and let $\set{U} \subseteq \Vertex$. The \emph{subgraph of $\Gr{G}$ induced
by $\set{U}$} is the graph obtained by the vertices in $\set{U}$ and the edges in $\Gr{G}$ that has both ends on $\set{U}$.
We say that $\Gr{H}$ is an \emph{induced subgraph of $\Gr{G}$}, if it is induced by some $\set{U} \subseteq \Vertex$.
\end{definition}

\begin{definition}[Strongly regular graphs]
\label{definition: strongly regular graphs}
A regular graph $\Gr{G}$ that is neither complete nor empty is called a
{\em strongly regular} graph with parameters $(n,d,\lambda,\mu)$,
where $\lambda$ and $\mu$ are nonnegative integers, if the following conditions hold:
\begin{enumerate}[(1)]
\item \label{Item 1 - definition of SRG}
$\Gr{G}$ is a $d$-regular graph on $n$ vertices.
\item \label{Item 2 - definition of SRG}
Every two adjacent vertices in $\Gr{G}$ have exactly $\lambda$ common neighbors.
\item \label{Item 3 - definition of SRG}
Every two distinct and nonadjacent vertices in $\Gr{G}$ have exactly $\mu$ common neighbors.
\end{enumerate}
The family of strongly regular graphs with these four specified parameters is denoted by $\srg{n}{d}{\lambda}{\mu}$.
It is important to note that a family of the form $\srg{n}{d}{\lambda}{\mu}$ may contain multiple nonisomorphic strongly
regular graphs. Throughout this work, we refer to a strongly regular graph as $\srg{n}{d}{\lambda}{\mu}$ if it belongs to
this family.
\end{definition}

\begin{proposition}[Feasible parameter vectors of strongly regular graphs]
\label{proposition: necessary condition for the parameter vector of SRGs}
The four parameters of a strongly regular graph $\srg{n}{d}{\lambda}{\mu}$ satisfy the equality
\begin{align}
\label{eq: necessary condition for the parameter vector of SRGs}
(n-d-1)\mu = d(d-\lambda-1).
\end{align}
\end{proposition}
\begin{remark}
\label{remark: necessary condition for the parameter vector of SRGs}
Equality~\eqref{eq: necessary condition for the parameter vector of SRGs} provides a necessary, but
not sufficient, condition for the existence of a strongly regular graph $\srg{n}{d}{\lambda}{\mu}$.
For example, as shown in \cite{Haemers93}, no $(76,21,2,7)$ strongly regular graph exists, even though
the condition $(n-d-1)\mu = 378 = d(d-\lambda-1)$ is satisfied in this case.
\end{remark}

\begin{notation}[Classes of graphs]
\noindent

\begin{itemize}
\item $\CoG{n}$ is the complete graph on $n$ vertices.
\item $\PathG{n}$ is the path graph on $n$ vertices.
\item $\CoBG{\ell}{r}$ is the complete bipartite graph whose degrees of partite sets are $\ell$ and $r$ (with possible equality between $\ell$ and $r$).
\item $\SG{n}$ is the star graph on $n$ vertices $\SG{n} = \CoBG{1}{n-1}$.
\end{itemize}
\end{notation}

\begin{definition}[Integer-valued functions of a graph]
\noindent
\begin{itemize}
\item Let $k \in \naturals $. A \emph{proper} $k$-\emph{coloring} of a graph $\Gr{G}$
is a function $c \colon \V{\Gr{G}} \to \{1,2,...,k\}$, where $c(v) \ne c(u)$ for every
$\{u,v\}\in \E{\Gr{G}}$. The \emph{chromatic number} of $\Gr{G}$, denoted by $\chrnum{\Gr{G}}$,
is the smallest $k$ for which there exists a proper $k$-coloring of $\Gr{G}$.
\item A \emph{clique} in a graph $\Gr{G}$ is a subset of vertices $U\subseteq \V{\Gr{G}}$
where the subgraph induced by $U$ is a complete graph. The \emph{clique number} of $\Gr{G}$,
denoted by $\omega(\Gr{G})$, is the largest size of a clique in $\Gr{G}$; i.e., it is the
largest order of an induced complete subgraph in $\Gr{G}$.
\item An \emph{independent set} in a graph $\Gr{G}$ is a subset of vertices $U\subseteq \V{\Gr{G}}$,
where $\{u,v\} \notin \E{\Gr{G}}$ for every $u,v \in U$. The \emph{independence number} of $\Gr{G}$,
denoted by $\indnum{\Gr{G}}$, is the largest size of an independent set in $\Gr{G}$.
\end{itemize}
\end{definition}

\begin{definition}[Orthogonal and orthonormal representations of a graph]
\label{def: orthogonal representation}
Let $\Gr{G}$ be a finite, simple, and undirected graph, and let $d \in \naturals$.
\begin{itemize}
\item
An {\em orthogonal representation} of the graph $\Gr{G}$ in the $d$-dimensional
Euclidean space $\Reals^d$ assigns to each vertex $i \in \V{\Gr{G}}$
a nonzero vector ${\bf{u}}_i \in \Reals^d$ such that
${\bf{u}}_i^{\mathrm{T}} {\bf{u}}_j = 0$ for every $\{i, j\} \notin \E{\Gr{G}}$
with $i \neq j$. In other words, for every two distinct and nonadjacent vertices in the graph,
their assigned nonzero vectors should be orthogonal in $\Reals^d$.
\item
An {\em orthonormal representation} of $\Gr{G}$ is additionally represented by unit vectors,
i.e., $\| {\bf{u}}_i \| = 1$ for all $i \in \V{\Gr{G}}$.
\item
In an orthogonal (orthonormal) representation of $\Gr{G}$, every two nonadjacent vertices in $\Gr{G}$
are mapped (by definition) into orthogonal (orthonormal) vectors, but adjacent vertices may not necessarily
be mapped into nonorthogonal vectors. If ${\bf{u}}_i^{\mathrm{T}} {\bf{u}}_j \neq 0$ for all
$\{i, j\} \in \E{\Gr{G}}$, then such a representation of $\Gr{G}$ is called {\em faithful}.
\end{itemize}
\end{definition}

\begin{definition}[Lov\'{a}sz $\vartheta$-function \cite{Lovasz79_IT}]
\label{definition: Lovasz theta function}
Let $\Gr{G}$ be a finite, simple, and undirected graph. Then, the {\em Lov\'{a}sz
$\vartheta$-function of $\Gr{G}$} is defined as
\begin{eqnarray}
\label{eq: Lovasz theta function}
\vartheta(\Gr{G}) \triangleq \min_{{\bf{c}}, \{{\bf{u}}_i\}} \, \max_{i \in \V{\Gr{G}}} \,
\frac1{\bigl( {\bf{c}}^{\mathrm{T}} {\bf{u}}_i \bigr)^2} \, ,
\end{eqnarray}
where the minimum on the right-hand side of \eqref{eq: Lovasz theta function} is taken over all
unit vectors ${\bf{c}}$ and all orthonormal representations $\{{\bf{u}}_i: i \in \V{\Gr{G}} \}$
of $\Gr{G}$. In \eqref{eq: Lovasz theta function}, it suffices to consider orthonormal representations
in a space of dimension at most $n = \card{\V{\Gr{G}}}$.
\end{definition}

The Lov\'{a}sz $\vartheta$-function of a graph $\Gr{G}$ can be calculated by solving (numerically) a convex optimization problem.
Let ${\bf{A}} = (A_{i,j})$ be the $n \times n$ adjacency matrix of $\Gr{G}$ with
$n \triangleq \card{\V{\Gr{G}}}$. The Lov\'{a}sz
$\vartheta$-function $\vartheta(\Gr{G})$ can be expressed as the solution of the following semidefinite programming (SDP) problem:
\vspace*{0.2cm}
\begin{eqnarray}
\label{eq: SDP problem - Lovasz theta-function}
\mbox{\fbox{$
\begin{array}{l}
\text{maximize} \; \; \mathrm{Tr}({\bf{B}} \J{n})  \\
\text{subject to} \\
\begin{cases}
{\bf{B}} \succeq 0, \\
\mathrm{Tr}({\bf{B}}) = 1, \\
A_{i,j} = 1  \; \Rightarrow \;  B_{i,j} = 0, \quad i,j \in \OneTo{n}.
\end{cases}
\end{array}$}}
\end{eqnarray}

\vspace*{0.1cm}
There exist efficient convex optimization algorithms (e.g., interior-point methods) to compute
$\vartheta(\Gr{G})$, for every graph $\Gr{G}$, with a precision of $r$ decimal digits, and a
computational complexity that is polynomial in $n$ and $r$. The reader is referred to
Section~2.5 of \cite{Sason2024} for an account of the various interesting properties of
the Lov\'{a}sz $\vartheta$-function. Among these properties, the sandwich theorem states that
for every graph $\Gr{G}$, the following inequalities hold:
\begin{align}
\label{eq1: sandwich}
\indnum{\Gr{G}} \leq \vartheta(\Gr{G}) \leq \chrnum{\CGr{G}}, \\
\label{eq2: sandwich}
\clnum{\Gr{G}} \leq \vartheta(\CGr{G}) \leq \chrnum{\Gr{G}}.
\end{align}
The usefulness of \eqref{eq1: sandwich} and \eqref{eq2: sandwich} lies in the fact that while the
independence, clique, and chromatic numbers of a graph are NP-hard to compute, the Lov\'{a}sz
$\vartheta$-function can be efficiently computed as a bound in these inequalities by solving the
convex optimization problem in \eqref{eq: SDP problem - Lovasz theta-function}.

\bigskip
\subsection{Matrices associated with a graph}
\label{subsection: Matrices associated with a graph}

\subsubsection{Four matrices associated with a graph}

\noindent

\vspace*{0.1cm}
Let $\Gr{G}=(\Vertex,\Edge)$ be a graph with vertices $\left\{ v_{1},...,v_{n}\right\} $.
There are several matrices associated with $\Gr{G}$. In this survey, we consider four of them,
all are symmetric matrices in $\mathbb{R}^{n\times n}$: the \emph{adjacency matrix} ($\A$),
\emph{Laplacian matrix} ($LM$), \emph{signless Laplacian matrix} ($\Q$), and the
\emph{normialized Laplacian matrix} (${\bf{\mathcal{L}}}$).

\begin{enumerate}
\item The adjacency matrix of a graph $\Gr{G}$, denoted by $\A = \A(\Gr{G})$, has the binary-valued entries
\begin{align}
\label{eq: adjacency matrix}
(\A(\Gr{G}))_{i,j}=
\begin{cases}
1 & \mbox{if} \, \{v_i,v_j\} \in \E{\Gr{G}}, \\
0 & \mbox{if} \, \{v_i,v_j\} \notin \E{\Gr{G}}.
\end{cases}
\end{align}
\item The Laplacian matrix of a graph $\Gr{G}$, denoted by $\LM = \LM(\Gr{G})$, is given by
\begin{align}
\LM(\Gr{G}) = \D(\Gr{G})-\A(\Gr{G}),
\end{align}
where
\begin{align}
\D(\Gr{G}) = \diag{d(v_1), d(v_2), \ldots ,d(v_n)}
\end{align}
is the {\em diagonal matrix} whose entries in the principal diagonal are the degrees of the $n$ vertices of $\Gr{G}$.
\item The signless Laplacian martix of a graph $\Gr{G}$, denoted by $\Q = \Q(\Gr{G})$, is given by
\begin{align}
\label{eq: signless Laplacian martix}
\Q(\Gr{G}) = \D(\Gr{G})+\A(\Gr{G}).
\end{align}
\item The normalized Laplacian matrix of a graph $\Gr{G}$, denoted by $\mathcal{L}(\Gr{G})$, is given by
\begin{align}
\label{eq: normalized Laplacian matrix}
\mathcal{L}(\Gr{G}) = \D^{-\frac12}(\Gr{G}) \, \LM(\Gr{G}) \,  \D^{-\frac12}(\Gr{G}),
\end{align}
where
\begin{align}
\D^{-\frac12}(\Gr{G}) = \diag{d^{-\frac12}(v_1), d^{-\frac12}(v_2), \ldots, d^{-\frac12}(v_n)},
\end{align}
with the convention that if $v \in \V{\Gr{G}}$ is an isolated vertex in $\Gr{G}$ (i.e., $d(v)=0$), then
$d^{-\frac12}(v) = 0$.	
The entries of ${\bf{\mathcal{L}}} = (\mathcal{L}_{i,j})$ are given by
\begin{align}
\mathcal{L}_{i,j} =
\begin{dcases}
\begin{array}{cl}
1, \quad & \mbox{if $i=j$ and $d(v_i) \neq 0$,} \\[0.2cm]
-\dfrac{1}{\sqrt{d(v_i) \, d(v_j)}}, \quad & \mbox{if $i \neq j$ and $\{v_i,v_j\} \in \E{\Gr{G}}$}, \\[0.5cm]
0, \quad & \mbox{otherwise}.
\end{array}
\end{dcases}
\end{align}	
\end{enumerate}

In the continuation of this section, we also occasionally refer to two other matrices that are associated with undirected graphs.
\begin{definition}
\label{definition: incidence matrix}
Let $\Gr{G}$ be a graph with $n$ vertices and $m$ edges. The {\em incidence matrix} of $\Gr{G}$, denoted by ${\mathbf{B}} = {\mathbf{B}}(\Gr{G})$
is an $n \times m$  matrix with binary entries, defined as follows:
\begin{align}
B_{i,j} =
\begin{cases}
1 & \text{if vertex \(v_i \in \V{\Gr{G}}\) is incident to edge \(e_j \in \E{\Gr{G}}\)}, \\
0 & \text{if vertex \(v_i \in \V{\Gr{G}}\) is not incident to edge \(e_j \in \E{\Gr{G}}\)}.
\end{cases}
\end{align}
For an undirected graph, each edge $e_j$ connects two vertices $v_i$ and $v_k$, and the corresponding column in
$\mathbf{B}$ has exactly two $1$'s, one for each vertex.
\end{definition}

\begin{definition}
\label{definition: oriented incidence matrix}
Let $\Gr{G}$ be a graph with $n$ vertices and $m$ edges. An {\em oriented incidence matrix} of $\Gr{G}$, denoted by ${\mathbf{N}} = {\mathbf{N}}(\Gr{G})$
is an $n \times m$  matrix with ternary entries from $\{-1, 0, 1\}$, defined as follows. One first selects an arbitrary orientation to each edge in $\Gr{G}$,
and then define
\begin{align}
N_{i,j} =
\begin{cases}
-1 & \text{if vertex \(v_i \in \V{\Gr{G}}\) is the tail (starting vertex) of edge \(e_j \in \E{\Gr{G}}\)}, \\
+1 & \text{if vertex \(v_i \in \V{\Gr{G}}\) is the head (ending vertex) of edge \(e_j \in \E{\Gr{G}}\)}, \\
\hspace*{0.2cm} 0 & \text{if vertex \(v_i \in \V{\Gr{G}}\) is not incident to edge \(e_j \in \E{\Gr{G}}\)}.
\end{cases}
\end{align}
Consequently, each column of $\mathbf{N}$ contains exactly one entry equal to 1 and one entry equal to $-1$, representing the head and tail of the corresponding
oriented edge in the graph, respectively, with all other entries in the column being zeros.
\end{definition}

For $X\in \left\{ A,L,Q,\mathcal{L} \right\}$, the \emph{$X$-spectrum} of a graph $\Gr{G}$, $\sigma_X(G)$, is the multiset of the eigenvalues
of $X(G)$. We denote the elements of the multiset of eigenvalues of $\{\A, \LM, \Q, \mathcal{L}\}$, respectively, by
\begin{align}
\label{eq2:26.09.23}
& \Eigval{1}{\Gr{G}} \geq \Eigval{2}{\Gr{G}} \geq \ldots \geq \Eigval{n}{\Gr{G}}, \\
\label{eq3:26.09.23}
& \mu_1(\Gr{G}) \leq \mu_2(\Gr{G}) \leq \ldots \leq \mu_n(\Gr{G}), \\
\label{eq4:26.09.23}
& \nu_1(\Gr{G}) \geq \nu_2(\Gr{G}) \geq \ldots \geq \nu_n(\Gr{G}), \\
\label{eq5:26.09.23}
& \delta_1(\Gr{G}) \leq \delta_2(\Gr{G}) \leq \ldots \leq \delta_n(\Gr{G}).
\end{align}

\begin{example}
Consider the complete bipartite graph $\Gr{G} = \CoBG{2}{3}$ with the adjacency matrix
$$\A(\Gr{G})=
\begin{pmatrix}
       {\bf{0}}_{2,2} & \J{2,3} \\
       \J{3,2} & {\bf{0}}_{3,3}
\end{pmatrix}.$$
The spectra of $\Gr{G}$ can be verified to be given as follows:
\begin{enumerate}
\item
The $\A$-spectrum of $\Gr{G}$ is
\begin{align}
\sigma_{\A}(\Gr{G})=\Bigl\{ -\sqrt{6}, [0]^{3}, \sqrt{6}\Bigr\},
\end{align}
with the notation that $[\lambda]^m$ means that $\lambda$ is an eigenvalue with multiplicity $m$.

\item The $\LM$-spectrum of $\Gr{G}$ is
\begin{align}
\sigma_{\LM}(\Gr{G})=\Bigl\{ 0, [2]^{2}, 3, 5\Bigr\} .
\end{align}
		
\item The $\Q$-spectrum of $\Gr{G}$ is
\begin{align}
\sigma_{\Q}(\Gr{G})=\Bigl\{ 0, [2]^{2}, 3, 5\Bigr\} .
\end{align}

\item The ${\bf{\mathcal{L}}}$-spectrum of $\Gr{G}$ is
\begin{align}
\sigma_{{\bf{\mathcal{L}}}}(\Gr{G})=\Bigl\{ 0, [1]^{3}, 2 \Bigr\} .
\end{align}
\end{enumerate}
\end{example}

\begin{remark}
If $\Gr{H}$ is an induced subgraph of a graph $\Gr{G}$, then $\A(\Gr{H})$ is a principal submatrix of $A(\Gr{G})$.
However, since the degrees of the remaining vertices are affected by the removal of vertices when forming the induced
subgraph $\Gr{H}$ from the graph $\Gr{G}$, this property does not hold for the other three associated matrices discussed
in this paper (namely, the Laplacian, signless Laplacian, and normalized Laplacian matrices).
\end{remark}

\begin{definition}
Let $\Gr{G}$ be a graph, and let $\CGr{G}$ be the complement graph of $\Gr{G}$. Define the following matrices:
\begin{enumerate}
\item $\overline{\A}(\Gr{G}) = \A(\overline{\Gr{G}})$.
\item $\overline{\LM}(\Gr{G}) = \LM(\overline{\Gr{G}})$.
\item $\overline{\Q}(\Gr{G}) = \Q(\overline{\Gr{G}})$.
\item $\overline{{\bf{\mathcal{L}}}}(\Gr{G}) = {\bf{\mathcal{L}}}(\overline{\Gr{G}})$.
\end{enumerate}
\end{definition}

\begin{definition}
Let $\mathcal{X} \subseteq \Gmats$. The $\mathcal{X}$-spectrum of a graph $\Gr{G}$ is a list with $\sigma_X(\Gr{G})$
for every $X\in \mathcal{X}$.
\end{definition}

Observe that if $\mathcal{X} = \{ X \}$ is a singleton, then the $\mathcal{X}$ spectrum is equal to the $X$-spectrum.

We now describe some important applications of the four matrices.

\subsubsection{Properties of the adjacency matrix}
\begin{theorem}[Number of walks of a given length between two fixed vertices]
\label{thm: number of walks of a given length}
Let $\Gr{G} = (\Vertex, \Edge)$ be a graph with a vertex set $\Vertex = \V{\Gr{G}} = \{ v_1, \ldots, v_n\}$, and
let $\A = \A(\Gr{G})$ be the adjacency matrix of $\Gr{G}$.
Then, the number of walks of length $\ell$, with the fixed endpoints $v_i$ and $v_j$, is equal to $(\A^\ell)_{i,j}$.
\end{theorem}

\begin{corollary}[Number of closed walks of a given length]
\label{corollary: Number of Closed Walks of a Given Length}
Let $\Gr{G} = (\Vertex, \Edge)$ be a simple undirected graph on $n$ vertices with an adjacency matrix
$\A = \A(\Gr{G})$, and let its spectrum (with respect to $\A$) be given by $\{\lambda_j\}_{j=1}^n$.
Then, for all $\ell \in \naturals$, the number of closed walks of length $\ell$ in $\Gr{G}$ is equal
to $\sum_{j=1}^n \lambda_j^{\ell}$.
\end{corollary}

\begin{corollary}[Number of edges and triangles in a graph]
\label{corollary: number of edges and triangles in a graph}
Let $\Gr{G}$ be a simple undirected graph with $n = \card{\V{\Gr{G}}}$ vertices,
$e = \card{\E{\Gr{G}}}$ edges, and $t$ triangles. Let $\A = \A(\Gr{G})$ be the
adjacency matrix of $\Gr{G}$, and let $\{\lambda_j\}_{j=1}^n$ be its adjacency
spectrum. Then,
\begin{align}
& \sum_{j=1}^n \lambda_j = \mathrm{tr}(\A) = 0, \label{eq: trace of A is zero} \\
& \sum_{j=1}^n \lambda_j^2 = \mathrm{tr}(\A^2) = 2 e, \label{eq: number of edges from A} \\
& \sum_{j=1}^n \lambda_j^3 = \mathrm{tr}(\A^3) = 6 t. \label{eq: number of triangles from A}
\end{align}
\end{corollary}

For a $d$-regular graph, the largest eigenvalue of its adjacency matrix is equal to~$d$.
Consequently, by Eq.~\eqref{eq: number of edges from A}, for $d$-regular graphs,
$\sum_j \lambda_j^2 = 2e = nd = n \lambda_1$. Interestingly, this turns to be
a necessary and sufficient condition for the regularity of a graph, which means
that the adjacency spectrum enables to identify whether a graph is regular.
\begin{theorem}  \cite[Corollary~3.2.2]{CvetkovicRS2010}
\label{theorem: graph regularity from A-spectrum}
A graph $\Gr{G}$ on $n$ vertices is regular if and only if
\begin{align}
\sum_{i=1}^n \lambda_i^2 = n \lambda_1,
\end{align}
where $\lambda_1$ is the largest eigenvalue of the adjacency matrix of $\Gr{G}$.
\end{theorem}

\begin{theorem}[The eigenvalues of strongly regular graphs]
\label{theorem: eigenvalues of srg}
The following spectral properties are satisfied by the family of strongly regular graphs:
\begin{enumerate}[(1)]
\item \label{Item 1: eigenvalues of srg}
A strongly regular graph has at most three distinct eigenvalues.
\item \label{Item 2: eigenvalues of srg}
Let $\Gr{G}$ be a connected strongly regular graph, and let its parameters be $\SRG(n,d,\lambda,\mu)$.
Then, the largest eigenvalue of its adjacency matrix is $\Eigval{1}{\Gr{G}} = d$ with multiplicity~1,
and the other two distinct eigenvalues of its adjacency matrix are given by
\begin{align}
\label{eigs-SRG}
p_{1,2} = \tfrac12 \, \Biggl( \lambda - \mu \pm \sqrt{ (\lambda-\mu)^2 + 4(d-\mu) } \, \Biggr),
\end{align}
with the respective multiplicities
\begin{align}
\label{eig-multiplicities-SRG}
m_{1,2} = \tfrac12 \, \Biggl( n-1 \mp \frac{2d+(n-1)(\lambda-\mu)}{\sqrt{(\lambda-\mu)^2+4(d-\mu)}} \, \Biggr).
\end{align}
\item \label{Item 3: eigenvalues of srg}
A connected regular graph with exactly three distinct eigenvalues is strongly regular.
\item \label{Item 4: eigenvalues of srg}
Strongly regular graphs for which $2d+(n-1)(\lambda-\mu) \neq 0$ have integral eigenvalues and the multiplicities
of $p_{1,2}$ are distinct.
\item \label{Item 5: eigenvalues of srg}
A connected regular graph is strongly regular if and only if it has three distinct eigenvalues, where the largest
eigenvalue is of multiplicity~1.
\item \label{Item 6: eigenvalues of srg}
A disconnected strongly regular graph is a disjoint union of $m$ identical complete graphs $\CoG{r}$, where $m \geq 2$ and $r \in \naturals$.
It belongs to the family $\srg{mr}{r-1}{r-2}{0}$, and its adjacency spectrum is $\{ (r-1)^{[m]}, (-1)^{[m(r-1)]} \}$, where superscripts
indicate the multiplicities of the eigenvalues, thus having two distinct eigenvalues.
\end{enumerate}
\end{theorem}

The following result follows readily from Theorem~\ref{theorem: eigenvalues of srg}.
\begin{corollary}
\label{corollary: cospectral SRGs}
Strongly regular graphs with identical parameters $(n,d,\lambda,\mu)$ are cospectral.
\end{corollary}

\begin{remark}
\label{remark: NICS SRGs}
Strongly regular graphs having identical parameters $(n, d, \lambda, \mu)$ are
cospectral but may not be isomorphic. For instance, Chang graphs form a set
of three nonisomorphic strongly regular graphs with identical parameters
$\srg{28}{12}{6}{4}$ \cite[Section~10.11]{BrouwerM22}. Consequently, the three
Chang graphs are strongly regular NICS graphs.
\end{remark}

An important class of strongly regular graphs, for which $2d+(n-1)(\lambda-\mu)=0$, is given by the family of conference graphs.
\begin{definition}[Conference graphs]
\label{definition: conference graphs}
A conference graph on $n$ vertices is a strongly regular graph with the parameters
$\srg{n}{\tfrac12(n-1)}{\tfrac14(n-5)}{\tfrac14(n-1)}$, where $n$ must satisfy $n=4k+1$
with $k \in \naturals$.
\end{definition}
If $\Gr{G}$ is a conference graph on $n$ vertices, then so is its complement $\CGr{G}$; it is, however,
not necessarily self-complementary. By Theorem~\ref{theorem: eigenvalues of srg}, the distinct eigenvalues of
the adjacency matrix of $\Gr{G}$ are given by $\tfrac12 (n-1)$, $\tfrac12 (\hspace*{-0.1cm}
\sqrt{n}-1)$, and $-\tfrac12 (\hspace*{-0.1cm} \sqrt{n}+1)$ with multiplicities
$1, \tfrac12 (n-1)$, and $\tfrac12 (n-1)$, respectively. In contrast to Item~\ref{Item 4: eigenvalues of srg}
of Theorem~\ref{theorem: eigenvalues of srg}, the eigenvalues $\pm \tfrac12 (\hspace*{-0.1cm} \sqrt{n}+1)$ are
not necessarily integers. For instance, the cycle graph $\CG{5}$, which is a conference graph, has an adjacency
spectrum $\bigl\{2, \bigl[\tfrac12 (\hspace*{-0.1cm} \sqrt{5}-1) \bigr]^{(2)},
\bigl[-\tfrac12 (\hspace*{-0.1cm} \sqrt{5}+1) \bigr]^{(2)} \}$. Thus, apart from the largest eigenvalue, the
other eigenvalues are irrational numbers.

\subsubsection{Properties of the Laplacian matrix}
\begin{theorem}
\label{theorem: On the Laplacian matrix of a graph}
Let $\Gr{G}$ be a finite, simple, and undirected graph, and let $\LM$ be the Laplacian matrix of $\Gr{G}$. Then,
\begin{enumerate}
\item \label{Item 1: Laplacian matrix of a graph}
The Laplacian matrix $\LM = {\mathbf{N}} {\mathbf{N}}^{\mathrm{T}}$ is positive semidefinite,
where ${\mathbf{N}}$ is the oriented incidence matrix of $\Gr{G}$ (see Definition~\ref{definition: oriented incidence matrix} and \cite[p.~185]{CvetkovicRS2010}).
\item \label{Item 2: Laplacian matrix of a graph}
The smallest eigenvalue of $\, \LM$ is zero, with a multiplicity equal to the number of components in $\Gr{G}$ (see \cite[Theorem~7.1.2]{CvetkovicRS2010}).
\item \label{Item 3: Laplacian matrix of a graph}
The size of the graph, $\bigcard{\E{\Gr{G}}}$, equals one-half of the sum of the eigenvalues of $\, \LM$, counted with multiplicities (see
\cite[Eq.~(7.4)]{CvetkovicRS2010}).
\end{enumerate}
\end{theorem}

The following celebrated theorem provides an operational meaning of the $\LM$-spectrum of graphs in counting their number of
spanning subgraphs.
\begin{theorem}[Kirchhoff's Matrix-Tree Theorem \cite{Kirchhoff1958}]
\label{theorem: number of spanning trees}
The number of spanning trees in a connected and simple graph $\Gr{G}$ on $n$ vertices is determined by the $n-1$ nonzero
eigenvalues of the Laplacian matrix, and it is equal to $\frac{1}{n} \overset{n}{\underset{\ell=2}{\prod}} \, \mu_\ell(\Gr{G})$.
\end{theorem}

\begin{corollary}[Cayley's Formula \cite{Cayley1889}]
\label{corollary: number of spanning trees}
The number of spanning trees of $\CoG{n}$ is $n^{n-2}$.
\end{corollary}
\begin{proof}
The $\LM$-spectrum of $\CoG{n}$ is given by $\{0, [n]^{n-1}\}$, and the result readily follows from Theorem~\ref{theorem: number of spanning trees}.
\end{proof}
	
\subsubsection{Properties of the signless Laplacian matrix}
\begin{theorem}
\label{theorem: On the signless Laplacian matrix of a graph}
Let $\Gr{G}$ be a finite, simple, and undirected graph, and let $\Q$ be
the signless Laplacian matrix of $\Gr{G}$. Then,
\begin{enumerate}
\item \label{Item 1: signless Laplacian matrix of a graph}
The matrix $\Q$ is positive semidefinite. Moreover, it is a completely positive matrix, expressed as $\Q = {\mathbf{B}} {\mathbf{B}}^{\mathrm{T}}$, where
${\mathbf{B}}$ is the incidence matrix of $\Gr{G}$ (see Definition~\ref{definition: incidence matrix} and \cite[Section~2.4]{CvetkovicRS2010}).
\item \label{Item 2: signless Laplacian matrix of a graph}
If $\Gr{G}$ is a connected graph, then it is bipartite if and only if the least eigenvalue of $\Q$ is equal to zero. In this case,
$0$ is a simple $\Q$-eigenvalue (see \cite[Theorem~7.8.1]{CvetkovicRS2010}).
\item \label{Item 3: signless Laplacian matrix of a graph}
The multiplicity of 0 as an eigenvalue of $\Q$ is equal to the number of bipartite components in $\Gr{G}$ (see \cite[Corollary~7.8.2]{CvetkovicRS2010}).
\item \label{Item 4: signless Laplacian matrix of a graph}
The size of the graph $\bigl| E(\Gr{G}) \bigr| $ is equal to one-half the sum of the eigenvalues of~$\Q$, counted with multiplicities
(see \cite[Corollary~7.8.9]{CvetkovicRS2010}).
\end{enumerate}
\end{theorem}
The interested reader is referred to \cite{OliveiraLAK2010} for bounds on the $\Q$-spread (i.e., the difference between the largest and smallest
eigenvalues of the signless Laplacian matrix), expressed as a function of the number of vertices in the graph. In regard to
Item~\ref{Item 2: signless Laplacian matrix of a graph} of Theorem~\ref{theorem: On the signless Laplacian matrix of a graph}, the interested
reader is referred to \cite{Cardoso2008} for a lower bound on the least eigenvalue of signless Laplacian matrix for connected non-bipartite graphs,
and to \cite{ChenH2018} for a lower bound on the least eigenvalue of signless Laplacian matrix for a general simple graph with a fixed number of
vertices and edges.

\subsubsection{Properties of the normalized Laplacian matrix}
The normalized Laplacian matrix of a graph, defined in \eqref{eq: normalized Laplacian matrix}, exhibits several
interesting spectral properties, which are introduced below.
\begin{theorem} \cite{CvetkovicRS2010,CvetkovicRS2007}
\label{theorem: On the normalized Laplacian matrix of a graph}
Let $\Gr{G}$ be a finite, simple, and undirected graph, and let ${\bf{\mathcal{L}}}$ be
the normalized Laplacian matrix of $\Gr{G}$. Then,
\begin{enumerate}
\item \label{Item 1: normalized Laplacian matrix of a graph}
The eigenvalues of ${\bf{\mathcal{L}}}$ lie in the interval $[0,2]$ (see \cite[Section~7.7]{CvetkovicRS2010}).
\item \label{Item 2: normalized Laplacian matrix of a graph}
The number of components in $\Gr{G}$ is equal to the multiplicity of~0 as an eigenvalue of ${\bf{\mathcal{L}}}$ (see \cite[Theorem~7.7.3]{CvetkovicRS2010}).
\item \label{Item 3: normalized Laplacian matrix of a graph}
The largest eigenvalue of ${\bf{\mathcal{L}}}$ is equal to~2 if and only if the graph has a bipartite component (see \cite[Theorem~7.7.2(v)]{CvetkovicRS2010}).
Furthermore, the number of the bipartite components of $\Gr{G}$ is equal to the multiplicity of~2 as an eigenvalue of~${\bf{\mathcal{L}}}$.
\item \label{Item 4: normalized Laplacian matrix of a graph}
The sum of its eigenvalues (including multiplicities) is less than or equal to the graph order $(n)$, with equality if and only
if the graph has no isolated vertices (see \cite[Theorem~7.7.2(i)]{CvetkovicRS2010}).
\end{enumerate}
\end{theorem}

\subsubsection{More on the spectral properties of the four associated matrices}

\noindent

The following theorem considers equivalent spectral properties of bipartite graphs.
\begin{theorem}
\label{theorem: equivalences for bipartite graphs}
Let $\Gr{G}$ be a graph. The following are equivalent:
\begin{enumerate}
\item \label{Item 1: TFAE bipartite graphs}
$\Gr{G}$ is a bipartite graph.
\item \label{Item 2: TFAE bipartite graphs}
$\Gr{G}$ does not have cycles of odd length.
\item \label{Item 3: TFAE bipartite graphs}
The $\A$-spectrum of $\Gr{G}$ is symmetric around zero, and for every eigenvalue $\lambda$ of $\A(G)$,
the eigenvalue $-\lambda$ is of the same multiplicity \cite[Theorem~3.2.3]{CvetkovicRS2010}.
\item \label{Item 4: TFAE bipartite graphs}
The $\LM$-spectrum and $\Q$-spectrum are identical (see \cite[Proposition~7.8.4]{CvetkovicRS2010}).
\item \label{Item 5: TFAE bipartite graphs}
The ${\bf{\mathcal{L}}}$-spectrum has the same multiplicity of $0$'s and $2$'s as eigenvalues (see
\cite[Corollary~7.7.4]{CvetkovicRS2010}).
\end{enumerate}
\end{theorem}

\begin{remark}
\label{remark: on connected bipartite graphs}
Item~\ref{Item 3: TFAE bipartite graphs} of Theorem~\ref{theorem: equivalences for bipartite graphs} can be
strengthened if $\Gr{G}$ is a connected graph. In that case, $\Gr{G}$ is bipartite if and only if $\lambda_1 = -\lambda_n$
(see \cite[Theorem~3.2.4]{CvetkovicRS2010}).
\end{remark}

\begin{table}[hbt]
\centering
\begin{tabular}{|c|c|c|c|c|c|}
\hline
\textbf{Matrix} & \textbf{\# edges} & \textbf{bipartite} & \textbf{\# components} & \textbf{\# bipartite components} & \textbf{\# of closed walks} \\
\hline
$\A$ & Yes & Yes & No & No & Yes \\
\hline
$\LM$ & Yes & No & Yes & No & No \\
\hline
$\Q$ & Yes & No & No & Yes & No \\
\hline
${\bf{\mathcal{L}}}$ & No & Yes & Yes & Yes & No \\
\hline
\end{tabular}
\caption{Some properties of a finite, simple, and undirected graph that one can or cannot determine
by the $X$-spectrum for $X\in \{\A,\LM,\Q, {\bf{\mathcal{L}}} \}$} \label{table:properties_determined by the spectrum}
\end{table}
Table~\ref{table:properties_determined by the spectrum}, borrowed from \cite{Butler2014}, lists properties of a graph that
can or cannot be determined by the $X$-spectrum for $X\in \{\A, \LM, \Q, \bf{\mathcal{L}}\}$. From the $\A$-spectrum of a
graph $\Gr{G}$, one can determine the number of edges and the number of triangles in $\Gr{G}$
(by Eqs.~\eqref{eq: number of edges from A} and \eqref{eq: number of triangles from A}, respectively), and whether the graph
is bipartite or not (by Item~\ref{Item 3: TFAE bipartite graphs} of Theorem~\ref{theorem: equivalences for bipartite graphs}).
However, the $\A$ spectrum does not indicate the number of components (see Example~\ref{example: ANICS graphs with 5 vertices}).
From the $\LM$-spectrum of a graph $\Gr{G}$, one can determine the number of edges
(by Item~\ref{Item 3: Laplacian matrix of a graph} of Theorem~\ref{theorem: On the Laplacian matrix of a graph}),
the number of spanning trees (by Theorem~\ref{theorem: number of spanning trees}),
the number of components of $\Gr{G}$ (by Item~\ref{Item 2: Laplacian matrix of a graph} of
Theorem~\ref{theorem: On the Laplacian matrix of a graph}), but not the number of its triangles, and
whether the graph $\Gr{G}$ is bipartite. From the $\Q$-spectrum, one can determine whether the graph is bipartite, the number
of bipartite components, and the number of edges (respectively, by Items~\ref{Item 3: signless Laplacian matrix of a graph}
and~\ref{Item 4: signless Laplacian matrix of a graph} of Theorem~\ref{theorem: On the signless Laplacian matrix of a graph}),
but not the number of components of the graph, and whether the graph is bipartite (see Remark~\ref{remark: bipartiteness}).
From the ${\bf{\mathcal{L}}}$-spectrum, one can determine the number of components and the
number of bipartite components in $\Gr{G}$ (by Theorem~\ref{theorem: On the normalized Laplacian matrix of a graph}),
and whether the graph is bipartite (by Items~\ref{Item 1: TFAE bipartite graphs} and~\ref{Item 5: TFAE bipartite graphs} of
Theorem~\ref{theorem: equivalences for bipartite graphs}). The number of closed walks in $\Gr{G}$ is determined by
the $\A$-spectrum (by Corollary~\ref{corollary: Number of Closed Walks of a Given Length}), but not by the spectra with
respect to the other three matrices.

\begin{remark}
\label{remark: bipartiteness}
By Item~\ref{Item 2: signless Laplacian matrix of a graph} of Theorem~\ref{theorem: On the signless Laplacian matrix of a graph},
a connected graph is bipartite if and only if the least eigenvalue of its signless Laplacian matrix is equal to zero.
If the graph is disconnected and it has a bipartite component and a non-bipartite component, then the least eigenvalue of its
signless Laplacian matrix is equal to zero, although the graph is not bipartite.
According to Table~\ref{table:properties_determined by the spectrum}, the $\Q$-spectrum alone does not determine whether a
graph is bipartite. This is due to the fact that the $\Q$-spectrum does not provide information about the number of components in the
graph or whether the graph is connected.
It is worth noting that while neither the $\LM$-spectrum nor the $\Q$-spectrum independently determines whether a graph is bipartite,
the combination of these spectra does. Specifically, by Item~\ref{Item 4: TFAE bipartite graphs} of
Theorem~\ref{theorem: equivalences for bipartite graphs}, the combined knowledge of both spectra enables to establish this property.
\end{remark}

\section{Graphs determined by their spectra}
\label{section: DS graphs}

The spectral determination of graphs has long been a central topic in spectral graph theory. A major open question in this area is:
"Which graphs are determined by their spectrum (DS)?" This section begins our survey of both classical and recent results on spectral
graph determination. We explore the spectral characterization of various graph classes, methods for constructing or distinguishing
cospectral nonisomorphic graphs, and conditions under which a graph’s spectrum uniquely determines its structure. Additionally, we
present newly obtained proofs of existing results, offering further insights into this field.

\begin{definition}
Let $\Gr{G},\Gr{H}$ be two graphs. A mapping $\phi \colon \V{\Gr{G}} \rightarrow \V{\Gr{H}}$ is a
\emph{graph isomorphism} if
\begin{align}
\{u,v\} \in \E{\Gr{G}} \iff \bigl\{ \phi(u),\phi(v) \bigr\} \in \E{\Gr{H}}.
\end{align}
If there is an isomorphism between $\Gr{G}$ and $\Gr{H}$, we say that these graphs are \emph{isomorphic}.
\end{definition}

\begin{definition}
A \emph{permutation matrix} is a $\{0,1\}$--matrix in which each row and each column contains exactly one entry equal to $1$.
\end{definition}

\begin{remark}
In terms of the adjacency matrix of a graph, $\Gr{G}$ and $\Gr{H}$ are cospectral graphs if $\A(\Gr{G})$ and $\A(\Gr{H})$
are similar matrices, and $\Gr{G}$ and $\Gr{H}$ are isomorphic if the similarity of their adjacency matrices is through
a permutation matrix ${\bf{P}}$, i.e.
\begin{align}
A(\Gr{G}) = {\bf{P}} \, \A(\Gr{H}) \, {\bf{P}}^{-1}.
\end{align}
\end{remark}

\subsection{Graphs determined by their adjacency spectrum (DS graphs)}
\label{subsection: Graphs determined by their adjacency spectrum}

\begin{theorem} \cite{vanDamH03}
\label{theorem: van Dam and Haemers, 2003 - thm1}
All of the graphs with less than five vertices are DS.
\end{theorem}

\begin{example}
\label{example: ANICS graphs with 5 vertices}
The star graph $\SG{5}$ and a graph formed by the disjoint union of a length-4 cycle and an isolated vertex,
$\CG{4} \DU \CoG{1}$, have the same $\A$-spectrum $\{-2 , [0]^3 , 2\}$. They are, however, not isomorphic
since $\SG{5}$ is connected and $\CG{4} \DU \CoG{1}$ is disconnected (see Figure~\ref{fig:graphs with 5 vertices}).
\vspace*{-0.1cm}
\begin{figure}[hbt]
\centering
\includegraphics[width=8cm]{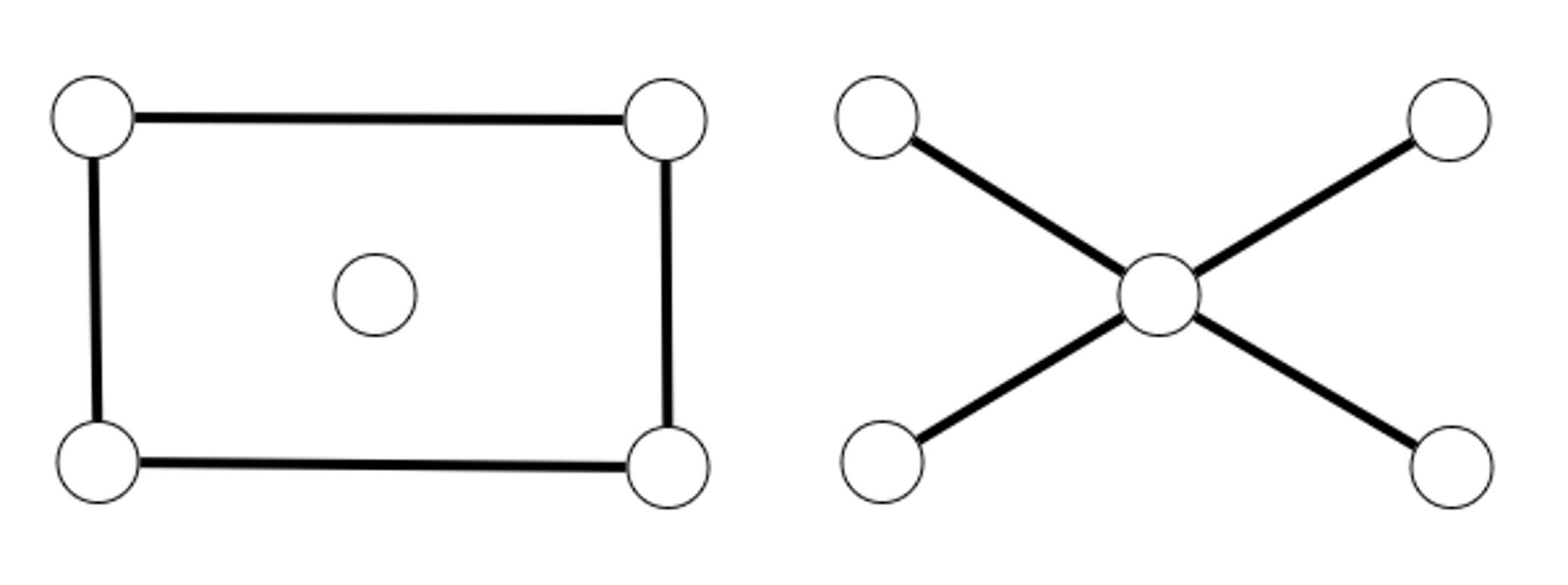}
\caption{The graphs $\SG{4} = \CoBG{1}{4}$ and $\CG{4} \DU \CoG{1}$ (i.e., a union of a 4-length cycle
and an isolated vertex) are cospectral and nonisomorphic graphs ($\A$-NICS graphs) on five vertices.
These two graphs therefore cannot be determined by their adjacency matrix.}
\label{fig:graphs with 5 vertices}
\end{figure}
It can be verified computationally that all the connected nonisomorphic graphs on five vertices can be
distinguished by their $\A$-spectrum (see \cite[Appendix~A1]{CvetkovicRS2010}).
\end{example}

\begin{theorem} \cite{vanDamH03}
\label{theorem: van Dam and Haemers, 2003 - thm2}
All the regular graphs with less than ten vertices are DS (and, as will be clarified later, also
$\mathcal{X}$-DS for every $\mathcal{X} \subseteq \{\A, \LM, \Q\}$).
\end{theorem}

\begin{example}
\label{example: NICS regular graphs on 10 vertices}
\cite{vanDamH03} The following two regular graphs in Figure \ref{fig:graphs with 10 vertices} are
$\{\A, \LM, \Q, \bf{\mathcal{L}}\}$-NICS.
\begin{figure}[hbt]
\centering
\includegraphics[width=12cm]{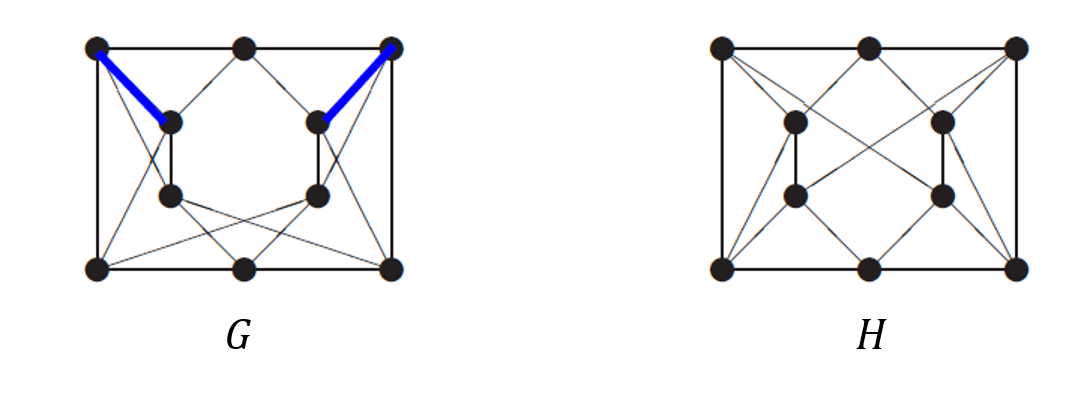}
\caption{$\{\A, \LM, \Q, \bf{\mathcal{L}}\}$-NICS regular graphs with $10$ vertices.
These cospectral graphs are nonisomorphic because each of the two blue edges in $\Gr{G}$ belongs
to three triangles, whereas no such an edge exists in $\Gr{H}$.}\label{fig:graphs with 10 vertices}
\end{figure}
The regular graphs $\Gr{G}$ and $\Gr{H}$ in Figure~\ref{fig:graphs with 10 vertices} can be verified to be cospectral
with the common characteristic polynomial $$P(x)= x^{10} - 20x^8 - 16x^7 + 110x^6 + 136x^5 - 180x^4 - 320x^3 + 9x^2 + 200x + 80.$$
These graphs are also nonisomorphic because each of the two blue edges in $\Gr{G}$ belongs
to three triangles, whereas no such an edge exists in $\Gr{H}$. Furthermore, it is shown in Example~4.18 of \cite{Sason2024}
that each pair of the regular NICS graphs on 10~vertices, denoted by $\{\Gr{G}, \Gr{H}\}$ and $\{\CGr{G}, \CGr{H}\}$,
exhibits distinct values of the Lov\'{a}sz $\vartheta$-functions, whereas the graphs $\Gr{G}$, $\CGr{G}$, $\Gr{H}$,
and $\CGr{H}$ share identical independence numbers~(3), clique numbers~(3), and chromatic numbers~(4).
Furthermore, based on these two pairs of graphs, it is constructively shown in Theorem~4.19 of \cite{Sason2024} that for
every even integer $n \geq 14$, there exist connected, irregular, cospectral, and nonisomorphic graphs on $n$ vertices,
being jointly cospectral with respect to their adjacency, Laplacian, signless Laplacian, and normalized Laplacian matrices,
while also sharing identical independence, clique, and chromatic numbers, but being distinguished by their Lov\'{a}sz
$\vartheta$-functions.
\end{example}

\begin{remark}
\label{remark: relations to Igal's paper 2023}
In continuation to Example~\ref{example: NICS regular graphs on 10 vertices}, it is worth noting that closed-form
expressions for the Lov\'{a}sz $\vartheta$-functions of regular graphs, which are edge-transitive or strongly regular,
were derived in \cite[Theorem~9]{Lovasz79_IT} and \cite[Proposition~1]{Sason23}, respectively. In particular,
it follows from \cite[Proposition~1]{Sason23} that strongly regular graphs with identical four parameters
$(n,d,\lambda,\mu)$ are cospectral and they have identical Lov\'{a}sz $\vartheta$-numbers, although they need not be
necessarily isomorphic. For such an explicit counterexample, the reader is referred to \cite[Remark~3]{Sason23}.
\end{remark}

We next introduce friendship graphs to address their possible determination by their spectra with respect to
several associated matrices.
\begin{definition}
\label{definition: friendship graph}
Let $p\in \naturals$. \emph{The friendship graph} $\FG{p}$, also known as the \emph{windmill graph}, is a graph
with $2p+1$ vertices, consisting of a single vertex (the central vertex) that is adjacent to all the other $2p$ vertices.
Furthermore, every pair of these $2p$ vertices shares exactly one common neighbor, namely the central vertex (see
Figure~\ref{fig:friendship graph F4}). This graph has $3p$ edges and $p$ triangles.
\end{definition}
\begin{figure}[H]
\centering
\includegraphics[width=3cm]{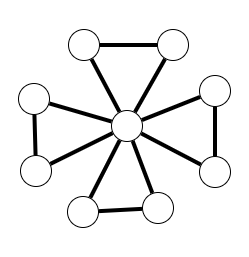}
\caption{The friendship (windmill) graph $\FG{4}$ has 9~vertices, 12 edges, and~4 triangles.}\label{fig:friendship graph F4}
\end{figure}
The term friendship graph in Definition~\ref{definition: friendship graph} originates from the
\emph{Friendship Theorem} \cite{Erdos1963}. This theorem states that if $\Gr{G}$ is a finite graph where
any two vertices share exactly one common neighbor, then there exists a vertex that is adjacent to all other
vertices. In this context, the adjacency of vertices in the graph can be interpreted socially as a representation
of friendship between the individuals represented by the vertices (assuming friendship is a mutual relationship).
For a nice exposition of the proof of the Friendship Theorem, the interested reader is referred to Chapter~44 of
\cite{AignerZ18}.

\begin{theorem}
\label{theorem: special classes of DS graphs}
The following graphs are DS:
\begin{enumerate}[1.]
\item \label{item 1: DS graphs}
All graphs with less than five vertices, and also all regular graphs with less than 10 vertices \cite{vanDamH03} (recall
Theorems~\ref{theorem: van Dam and Haemers, 2003 - thm1} and~\ref{theorem: van Dam and Haemers, 2003 - thm2}).
\item \label{item 2: DS graphs}
The graphs $\CoG{n}$, $\CG{n}$, $\PathG{n}$, $\CoBG{m}{m}$ and $\CGr{\CoG{n}}$ \cite{vanDamH03}.
\item \label{item 3: DS graphs}
The complement of the path graph $\CGr{\PathG{n}}$ \cite{DoobH02}.
\item \label{item 4: DS graphs}
The disjoint union of $k$ path graph with no isolated vertices, the disjoint union of $k$ complete graphs with no isolated
vertices, and the disjoint union of $k$ cycles (i.e., every 2-regular graph) \cite{vanDamH03}.
\item \label{item 5: DS graphs}
The complement graph of a DS regular graph \cite{CvetkovicRS2010}.
\item \label{item 6: DS graphs}
Every $(n-3)$-regular graph on $n$ vertices \cite{CvetkovicRS2010}.
\item \label{item 7: DS graphs}
The friendship graph $\FG{p}$ for $p \ne 16$ \cite{CioabaHVW2015}.
\item \label{item 8: DS graphs}
Sandglass graphs, which are obtained by appending a triangle to each of the pendant (i.e., degree-1) vertices of a path \cite{LuLYY09}.
\item \label{item 9: DS graphs}
If $\Gr{H}$ is a subgraph of a graph $\Gr{G}$, and $\Gr{G} \setminus \Gr{H}$ denotes the graph obtained from $\Gr{G}$
by deleting the edges of $\Gr{H}$, then also the following graphs are DS \cite{CamaraH14}:
\begin{itemize}
\item $\CoG{n} \setminus (\ell \CoG{2})$ and $\CoG{n} \setminus \CoG{m}$, where $m \leq n-2$,
\item $\CoG{n} \setminus \CoBG{\ell}{m}$,
\item $\CoG{n} \setminus \Gr{H}$, where $\Gr{H}$ has at most four edges.
\end{itemize}
\end{enumerate}
\end{theorem}

\subsection{Graphs determined by their spectra with respect to various matrices (X-DS graphs)}
\label{subsection: Graphs determined by their X-DS spectrum}

\noindent

In this section, we consider graphs that are determined by the spectra of various associated matrices beyond the adjacency matrix spectrum.

\begin{definition}
Let $\Gr{G} , \Gr{H}$ be two graphs and let $\mathcal{X} \subseteq \Gmats$.
\begin{enumerate}
\item $\Gr{G}$ and $\Gr{H}$ are said to be \emph{$\mathcal{X}$-cospectral} if they have the same $X$-spectrum, i.e.
$\sigma_X(\Gr{G}) = \sigma_X(\Gr{H})$.
\item
Nonisomorphic graphs $\Gr{G}$ and $\Gr{H}$ that are $\mathcal{X}$-cospectral are said to be \emph{$\mathcal{X}$-NICS},
where {\em NICS} is an abbreviation of {\em non-isomorphic and cospectral}.
\item A graph $\Gr{G}$ is said to be \emph{determined by its $\mathcal{X}$-spectrum
($\mathcal{X}$-DS)} if every graph that is $\mathcal{X}$-cospectral to $\Gr{G}$ is also isomorphic to $\Gr{G}$.
\end{enumerate}
\end{definition}

\begin{notation}
For a singleton $\mathcal{X} = \{ X \}$, we abbreviate $\{ X \} $-cospectral, $\{X\}$-DS and $\{X\}$-NICS by $X$-cospectral,
$X$-DS and $X$-NICS, respectively. For the adjacency matrix, we will abbreviate $\A$-DS by DS.
\end{notation}

\begin{remark}
\label{remark: X,Y cospectrality}
  Let $\mathcal{X} \subseteq \mathcal{Y} \subseteq \Gmats$. The following holds by definition:
\begin{itemize}
  \item If two graph $\Gr{G}, \Gr{H}$ are $\mathcal{Y}$-cospectral, then they are $\mathcal{X}$-cospectral.
  \item If a graph $\Gr{G}$ is $\mathcal{X}$-DS, then it is $\mathcal{Y}$-DS.
\end{itemize}
\end{remark}

\begin{definition}
\label{definition: generalized spectrum}
Let $\Gr{G}$ be a graph. The \emph{generalized spectrum} of $\Gr{G}$ is the $\{\A, \overline{\A}\}$-spectrum of $\Gr{G}$.
\end{definition}

The following result on the cospectrality of regular graphs can be readily verified.
\begin{proposition}
\label{proposition: regular graphs cospectrality}
Let $\Gr{G}$ and $\Gr{H}$ be regular graphs that are $\mathcal{X}$-cospectral for {\em some}
$\mathcal{X} \subseteq \{\A, \LM, \Q, \bf{\mathcal{L}}\}$. Then, $\Gr{G}$ and $\Gr{H}$ are
$\mathcal{Y}$-cospectral for {\em every}
$\mathcal{Y} \subseteq \{\A, \overline{\A}, \LM, \overline{\LM}, \Q, \overline{\Q}, {\bf{\mathcal{L}}}, \overline{{\bf{\mathcal{L}}}} \}$.
In particular, the cospectrality of regular graphs (and their complements) stays unaffected
by the chosen matrix among $\{\A, \LM, \Q, \bf{\mathcal{L}}\}$.
\end{proposition}

\begin{definition}
\label{definition: DGS}
A graph $\Gr{G}$ is said to be \emph{determined by its generalized spectrum (DGS)} if it is uniquely determined by its
generalized spectrum. In other words, a graph $\Gr{G}$ is DGS if and only if every graph $\Gr{H}$ with the same
$\{\A, \overline{\A}\}$-spectrum as $\Gr{G}$ is necessarily isomorphic to $\Gr{G}$.
\end{definition}
If a graph is not DS, it may still be DGS, as additional spectral information is available. Conversely, every DS graph
is also DGS. For further insights into DGS graphs, including various characterizations, conjectures, and studies, we refer
the reader to \cite{WangXu06,Wang13,Wang17}.

\vspace*{0.2cm}
The continuation of this section characterizes graphs that are $X$-DS, where $X \in \{\LM, \Q, \mathcal{L}\}$, with pointers
to various studies. We first consider regular DS graphs.

\begin{theorem} \cite[Proposition~3]{vanDamH03}
\label{theorem: regular DS graphs}
For regular graphs, the properties of being DS, $\LM$-DS, and $\Q$-DS are equivalent.
\end{theorem}
\begin{remark}
\label{remark: recurring approach}
To avoid any potential confusion, it is important to emphasize that in statements such as Theorem~\ref{theorem: regular DS graphs},
the only available information is the spectrum of the graph. There is no indication or prior knowledge that the spectrum corresponds
to a regular graph. In such cases, the regularity of the graph is not part of the revealed information and, therefore, cannot be used
to determine the graph. This recurring approach --- stating that $\Gr{G}$ is stated to be a graph satisfying certain properties (e.g., regularity,
strong regularity, etc.) and then examining whether the graph can be determined from its spectrum --- appears throughout this paper. It
should be understood that the only available information is the spectrum of the graph, and no additional properties of the graph beyond
its spectrum are disclosed.
\end{remark}

\begin{remark}
\label{remark: DS regular graphs are not necessarily DS w.r.t. normalized Laplacian}
The crux of the proof of Theorem~\ref{theorem: regular DS graphs} is that there are no two NICS
graphs, with respect to either $\A$, $\LM$, or $\Q$, where one graph is regular and the other is
irregular (see \cite[Proposition~2.2]{vanDamH03}).
This, however, does not extend to NICS graphs with respect to the normalized Laplacian matrix $\mathcal{L}$,
and regular DS graphs are not necessarily $\mathcal{L}$-DS. For instance, the cycle $\CG{4}$ and the bipartite complete
graph $\CoBG{1}{3}$ (i.e., $\SG{3}$) share the same $\mathcal{L}$-spectrum, which is given by $\{0, 1^2, 2\}$,
but these graphs are nonisomorphic (as $\CG{4}$ is regular, in contrast to $\CoBG{1}{3}$). It therefore follows
that the 2-regular graph $\CG{4}$ is {\em not} $\mathcal{L}$-DS, although it is DS (see Item~\ref{item 2: DS graphs}
of Theorem~\ref{theorem: special classes of DS graphs}).
More generally, it is conjectured in \cite{Butler2016} that $\CG{n}$ is $\mathcal{L}$-DS if
and only if $n>4$ and $4 \hspace*{-0.1cm} \not| \, n$.
\end{remark}

\begin{theorem}
\label{theorem: L-DS graphs}
The following graphs are $\LM$-DS:
\begin{enumerate}[1.]
\item $\PathG{n},\CG{n},\CoG{n},\CoBG{m}{m}$ and their complements \cite{vanDamH03}.
\item The disjoint union of $k$ paths, $\PathG{n_1} \DU \PathG{n_2} \DU \ldots \DU \PathG{n_k}$ each
having at least one edge \cite{vanDamH03}.
\item The complete bipartite graph $\CoBG{m}{n}$ with $m,n\geq2$ and $\frac{5}{3}n<m$ \cite{Boulet2009}.
\item \label{stars: L-DS}
The star graphs $\SG{n}$ with $n \neq 3$ \cite{OmidiT2007,LiuZG2008}.
\item Trees with a single vertex having a degree greater than~2 (referred to as starlike trees) \cite{OmidiT2007,LiuZG2008}.
\item The friendship graph $\FG{p}$ \cite{LiuZG2008}.
\item The path-friendship graphs, where a friendship graph and a starlike tree are joined by merging their vertices of degree greater than~2 \cite{OboudiAAB2021}.
\item The wheel graph $\Gr{W}_{n+1} \triangleq \CoG{1} \vee \CG{n}$ for $n \neq 7$ (otherwise, if $n=7$, then it is not $\LM$-DS) \cite{ZhangLY09}.
\item The join of a clique and an independent set on $n$ vertices, $\CoG{n-m} \vee \, \CGr{\CoG{m}}$, where $m \in \OneTo{n-1}$ \cite{DasL2016}.
\item Sandglass graphs (see also Item~\ref{item 8: DS graphs} in Theorem~\ref{theorem: special classes of DS graphs}) \cite{LuLYY09}.
\item The join graph $\Gr{G} \vee \CoG{m}$, for every $m \in \naturals$, where $\Gr{G}$ is a disconnected graph \cite{ZhouBu2012}.
\item The join graph $\Gr{G} \vee \CoG{m}$, for every $m \in \naturals$, where $\Gr{G}$ is an $\LM$-DS connected graph on $n$ vertices and $m$ edges
with $m \leq 2n-6$, $\CGr{G}$ is a connected graph, and either one of the following conditions holds \cite{ZhouBu2012}:
\begin{itemize}
\item $\Gr{G} \vee \CoG{1}$ is $\LM$-DS;
\item the maximum degree of $\Gr{G}$ is smaller than $\tfrac12 (n-2)$.
\end{itemize}
\item Specifically, the join graph $\Gr{G} \vee \CoG{m}$, for every $m \in \naturals$, where $\Gr{G}$ is an $\LM$-DS tree on $n \geq 5$ vertices
(since, the equality $m=n-1$ holds for a tree on $n$ vertices and $m$ edges) \cite{ZhouBu2012}.
\end{enumerate}
\end{theorem}

\begin{remark}
In general, a disjoint union of complete graphs is not determined by its Laplacian spectrum.
\end{remark}

\begin{theorem}
\label{theorem: Q-DS graphs}
The following graphs are $\Q$-DS:
\begin{enumerate}[1.]
\item The disjoint union of $k$ paths, $\PathG{n_1} \DU \PathG{n_2} \DU \ldots \DU \PathG{n_k}$ each
having at least one edge \cite{vanDamH03}.
\item The star graphs $\SG{n}$ with $n \geq 3$ \cite{BuZ2012b,OmidiV2010}.
\item Trees with a single vertex having a degree greater than~2 \cite{BuZ2012b,OmidiV2010}.
\item The friendship graph $\FG{k}$ \cite{WangBHB2010}.
\item The lollipop graphs, where a lollipop graph, denoted by $\mathrm{H}_{n,p}$ where $n,p \in \naturals$ and $p<n$,
is obtained by appending a cycle $\CG{p}$ to a pendant vertex of a path $\PathG{n-p}$ \cite{HamidzadeK2010,ZhangLZY09}.
\item $\Gr{G} \vee \CoG{1}$ where $\Gr{G}$ is a either a $1$-regular graph, an $(n-2)$-regular graph of order $n$
or a $2$-regular graph with at least $11$ vertices \cite{BuZ2012}.
\item If $n \geq 21$ and $0 \leq q \leq n-1$, then $\CoG{1} \vee (\PathG{q} \DU \, (n-q-1) \CoG{1})$ \cite{YeLS2025}.
\item If $n \geq 21$ and $3 \leq q \leq n-1$, then $\CoG{1} \vee (\CG{q} \DU \, (n-q-1) \CoG{1})$ is $\Q$-DS if and
only if $q \neq 3$ \cite{YeLS2025}.
\item The join of a clique and an independent set on $n$ vertices, $\CoG{n-m} \vee \, \CGr{\CoG{m}}$, where
$m \in \OneTo{n-1}$ and $m \neq 3$ \cite{DasL2016}.
\end{enumerate}
\end{theorem}

Since the regular graphs $\CoG{n}$, $\CGr{\CoG{n}}$, $\CoBG{m}{m}$ and $\CG{n}$ are DS, they are also $\mathcal{X}$-DS
for every $\mathcal{X} \subseteq \{\A, \LM, \Q \}$ (see Theorem~\ref{theorem: regular DS graphs}). This, however, does
not apply to regular ${\bf{\mathcal{L}}}$-DS graphs (see Remark~\ref{remark: DS regular graphs are not necessarily DS w.r.t. normalized Laplacian}),
which are next addressed.

\begin{theorem}
\label{theorem: X-DS friendship graphs}
The following graphs are ${\bf{\mathcal{L}}}$-DS:
\begin{itemize}
\item $\CoG{n}$, for every $n \in \naturals$ \cite{ButlerH2016}.
\item The friendship graph $\FG{k}$, for $k \geq 2$ \cite[Corollary~1]{BermanCCLZ2018}.
\item More generally, $\mathrm{F}_{p,q} = \CoG{1} \vee (p \CoG{q})$ if $q \geq 2$, or $q=1$ and $p \geq 2$ \cite[Theorem~1]{BermanCCLZ2018}.
\end{itemize}
\end{theorem}

\noindent

\section{Special families of graphs}
\label{section: special families of graphs}
This section introduces special families of structured graphs and it states conditions for their unique
determination by their spectra.

\subsection{Stars and graphs of pyramids}
\label{subsection: Stars and graphs of pyramids}

\noindent

\begin{definition}
\label{definition: graphs of pyramids}
For every $k,n \in \naturals$ with $k<n$, define the graph  $T_{n,k}=\CoG{k} \vee \, \overline{\CoG{n-k}}$.
For $k=1$, the graph $T_{n,k}$ represents the \emph{star graph} $\SG{n}$. For $k=2$, it represents
a graph comprising $n-2$ triangles sharing a common edge, referred to as a \emph{crown}. For $n,k$ satisfying
$1<k<n$, the graphs $T_{n,k}$ are referred to as \emph{graphs of pyramids} \cite{KrupnikB2024}.
\end{definition}

\begin{theorem} \cite{KrupnikB2024}
\label{thm: KrupnikB2024 - pyramids are DS}
The graphs of pyramids are DS for every $1<k<n$.
\end{theorem}

\begin{theorem} \cite{KrupnikB2024}
\label{thm: KrupnikB2024 - DS star graphs}
The star graph $\SG{n}$ is DS if and only if $n-1$ is prime.
\end{theorem}

To prove these theorems, a closed-form expression for the spectrum of $T_{n,k}$ is derived in \cite{KrupnikB2024}, which
also presents a generalized result. Subsequently,
using Theorem~\ref{thm: number of walks of a given length}, the number of edges and triangles in any graph cospectral with
$T_{n,k}$ are calculated. Finally, Schur's theorem (Theorem~\ref{theorem: Schur complement}) and Cauchy's interlacing theorem
(Theorem~\ref{thm:interlacing}) are applied in \cite{KrupnikB2024} to prove Theorems~\ref{thm: KrupnikB2024 - pyramids are DS}
and~\ref{thm: KrupnikB2024 - DS star graphs}.

\subsection{Complete bipartite graphs}
\label{subsection: Complete bipartite graphs}
By Theorem~\ref{thm: KrupnikB2024 - DS star graphs}, the star graph $\SG{n}=\CoBG{1}{n-1}$ is DS if and only if $n-1$ is prime.
By Theorem~\ref{theorem: special classes of DS graphs}, the regular complete bipartite graph $\CoBG{m}{m}$ is DS for every $m \in \naturals$.
Here, we generalize these results and provide a characterization for the DS property of $\CoBG{p}{q}$ for every $p,q\in \naturals$.

\begin{theorem} \cite{vanDamH03}
\label{thm:spectrum of CoBG}
The spectrum of the complete bipartite graph $\CoBG{p}{q}$ is $\bigl\{-\sqrt{pq}, [0]^{p+q-2}, \sqrt{pq} \bigr\}$.
\end{theorem}

This theorem can be proved by Theorem~\ref{theorem: Schur complement}. An alternative simple proof is next presented.
\begin{proof}
The adjacency matrix of $\CoBG{p}{q}$ is given by
\begin{align}
\A(\CoBG{p}{q}) =
\begin{pmatrix}
\mathbf{0}_{p,p} & \J{p,q}\\
\J{q,p} & \mathbf{0}_{q,q}
\end{pmatrix} \in \Reals^{(p+q) \times (p+q)}
\end{align}
The rank of $\A(\CoBG{p}{q})$ is equal to 2, so the multiplicity of $0$ as an eigenvalue is $p+q-2$.
By Corollary~\ref{corollary: number of edges and triangles in a graph}, the two remaining eigenvalues
are given by $\pm \lambda$ for some $\lambda \in \Reals$, since the eigenvalues sum to zero. Furthermore,
\begin{align}
2\lambda^2 = \sum_{i=1}^{p+q} \lambda_i^2 = 2 \, \card{\E{\CoBG{p}{q}}} = 2pq,
\end{align}
so $\lambda = \sqrt{pq}$.
\end{proof}

For $p,q \in \mathbb{N}$, the arithmetic and geometric means of $p,q$ are, respectively, given by
$\AM{p,q}=\tfrac12 (p+q)$ and $\GM{p,q}= \sqrt{ pq}$. The AM-GM inequality states that for
every $p,q \in \naturals$, we have $\GM{p,q} \le \AM{p,q}$ with equality if and only if $p=q$.

\begin{definition}
\label{definition: AM minimizer}
Let $p,q \in \naturals$. The two-elements multiset $\{p,q\} $ is said to be an \emph{AM-minimizer} if
it attains the minimum arithmetic mean for their given geometric mean, i.e.,
\begin{align}
\label{eq: AM minimizer}
\AM{p,q} &= \min \Bigl\{\AM{a,b}: \; a,b \in \mathbb{N}, \, \GM{a,b}=\GM{p,q} \Bigr\} \\
\label{eq2: AM minimizer}
&= \min \Bigl\{\tfrac12 (a+b): \; a,b \in \mathbb{N}, \, ab=pq \Bigr\}.
\end{align}
\end{definition}

\begin{example}
\label{example: AM minimizer}
The following are AM-minimizers:
\begin{itemize}
\item $\{k,k\}$ for every $k\in \naturals $. By the AM-GM inequality, it is the only case where
$\GM{p,q} = \AM{p,q}$.
\item $\{p,q\}$ where $p,q$ are prime numbers. In this case, the following family of multisets
\begin{align}
\Bigl\{ \{a,b\} : \,  a,b \in \mathbb{N}, \; \GM{a,b}=\GM{p,q} \Bigr\}
\end{align}
only contains the two multisets $\{p,q\},\{pq,1\}$, and $p+q \leq pq < pq+1$ since $p,q \geq 2$.
\item $\{1,q\}$ where $q$ is a prime number.
\end{itemize}
\end{example}

\begin{theorem}
\label{thm:when CoBG is DS?}
The following holds for every $p,q \in \naturals$:
\begin{enumerate}
\item \label{thm:when CoBG is DS? - part1}
Let $\Gr{G}$ be a graph that is cospectral with $\CoBG{p}{q}$. Then, up to isomorphism, $G = \CoBG{a}{b} \cup \Gr{H}$
(i.e., $\Gr{G}$ is a disjoint union of the two graphs $\CoBG{a}{b}$ and $\Gr{H}$), where $\Gr{H}$ is an empty graph and
$a,b \in \naturals$ satisfy $\GM{a,b} = \GM{p,q}$.
\item \label{thm:when CoBG is DS? - part2}
The complete bipartite graph $\CoBG{p}{q}$ is DS if and only if $\{p,q\}$ is an AM-minimizer.
\end{enumerate}
\end{theorem}
\begin{remark}
\label{remark: complete bipartite graphs}
Item~\ref{thm:when CoBG is DS? - part2} of Theorem~\ref{thm:when CoBG is DS?} is equivalent to Corollary~3.1 of
\cite{MaRen2010}, for which an alternative proof is presented here.
\end{remark}

\begin{proof}
(Proof of Theorem~\ref{thm:when CoBG is DS?}):
\begin{enumerate}
\item Let $\Gr{G}$ be a graph cospectral with $\CoBG{p}{q}$. The number of edges in $\Gr{G}$ equals the number of edges
in $\CoBG{p}{q}$, which is $pq$. As $\CoBG{p}{q}$ is bipartite, so is $\Gr{G}$. Since $\A(\Gr{G})$ is of rank $2$,
and $\A(\PathG{3})$ has rank $3$, it follows from the Cauchy's Interlacing Theorem (Theorem~\ref{thm:interlacing})
that $\PathG{3}$ is not an induced subgraph of $\Gr{G}$. \newline
It is claimed that $\Gr{G}$ has a single nonempty connected component. Suppose to the contrary that $\Gr{G}$ has (at least)
two nonempty connected components $\Gr{H}_1,\Gr{H}_2$. For $i\in \{1,2\}$, since $\Gr{H}_i$ is a non-empty graph, $\A(\Gr{H}_i)$
has at least one eigenvalue $\lambda \ne 0$. Since $\Gr{G}$ is a simple graph, the sum of the eigenvalues of $\A(\Gr{H}_i)$
is $\trace{\A(\Gr{H}_i)}=0$, so $\Gr{H}_i$ has at least one positive eigenvalue. Thus, the induced subgraph $\Gr{H}_1 \cup \Gr{H}_2$
has at least two positive eigenvalues, while $\Gr{G}$ has only one positive eigenvalue, contradicting Cauchy's Interlacing Theorem. \\
Hence, $\Gr{G}$ can be decomposed as $\Gr{G} = \CoBG{a}{b} \cup \, \Gr{H}$ where $\Gr{H}$ is an empty graph. Since $\Gr{G}$
and $\CoBG{p}{q}$ have the same number of edges, $pq=ab$, so $\GM{p,q}=\GM{a,b}$.
\item First, we will show that if $\{p,q\}$ is not an AM-minimizer, then the graph $\CoBG{p}{q}$ is not $\A$-DS. This is done
by finding a nonisomorphic graph to $\CoBG{p}{q}$ that is $\A$-cospectral with it. By assumption, since $\{p,q\}$ is not an
AM-minimizer, there exist $a, b \in \naturals$ satisfying $\GM{a,b} = \GM{p,q}$ and $a + b < p+q$.
Define the graph $\Gr{G}=\CoBG{a}{b} \vee \, \overline{\CoG{r}}$ where $r=p+q-a-b$. Observe that $r \in \naturals$.
The $\A$-spectrum of both of these graphs is given by
\begin{align}
\sigma_{\A}(\Gr{G}) = \sigma_{\A}(\CoBG{p}{q}) = \bigl\{-\sqrt{pq},[0]^{pq-2},\sqrt{pq} \bigr\},
\end{align}
so these two graphs are nonisomorphic and cospectral, which means that $\Gr{G}$ is not $\A$-DS. \newline
We next prove that if $\{p,q\}$ is an AM-minimizer, then $\CoBG{p}{q}$ is $\A$-DS. Let $\Gr{G}$ be a
graph that is cospectral with $\CoBG{p}{q}$. From the first part of this theorem,
$\Gr{G}=\CoBG{a}{b} \cup \, \Gr{H}$ where $\GM{a,b} = \GM{p,q}$ and $\Gr{H}$ is an empty graph. Consequently, it follows that
$\AM{a,b}=\tfrac12(a+b) \leq \tfrac12(p+q) = \AM{p,q}$.
Since $\{p,q\}$ is assumed to be an AM-minimizer, it follows that $\AM{a,b} \ge \AM{p,q}$, and thus equality holds.
Both equalities $\GM{a,b} = \GM{p,q}$ and $\AM{a,b} = \AM{p,q}$ can be satisfied simultaneously if and only if
$\{ a , b \} = \{ p , q \}$, so $r=p+q-a-b=0$ and $\Gr{G}=\CoBG{p}{q}$.
\end{enumerate}
\end{proof}

\begin{corollary}
\label{cor: bipartite not DS}
Almost all of the complete bipartite graphs are not DS. More specifically, for every $n \in \naturals$, there exists
a single complete bipartite graph on $n$ vertices that is DS.
\end{corollary}

\begin{proof}
Let $n \in \naturals$. By the \emph{fundamental theorem of arithmetic}, there is a unique decomposition
$n = \prod_{i=1}^{k} p_i$ where $k\in \naturals$ and $\{p_i\}$ are prime numbers for every $1 \le i \le k$.
Consider the family of multisets
\begin{align}
\set{D} = \Bigl\{ \{a,b\} :  a,b \in \mathbb{N} , \GM{a,b}=\sqrt{n} \Bigr\}.
\end{align}
This family has $2^k$ members, since every prime factor $p_i$ of $n$ should be in the prime decomposition of $a$ or $b$.
Since the minimization of $\AM{a,b}$ under the equality constraint $\GM{a,b}=\sqrt{n}$ forms a convex optimization problem,
only one of the multisets in the family $\set{D}$ is an AM-minimizer.
Thus, if $n = \prod_{i=1}^{k} p_i$, then the number of complete bipartite graphs of $n$ vertices is $O(2^k)$, and
(by Item~\ref{thm:when CoBG is DS? - part2} of Theorem~\ref{thm:when CoBG is DS?}) only one of them is DS.
\end{proof}

\subsection{Tur\'{a}n graphs}
\label{subsection: Turan graphs}
The Tur\'{a}n graphs are a significant and well-studied class of graphs in extremal graph theory, forming an important family
of multipartite complete graphs. Tur\'{a}n graphs are particularly known for their role in Tur\'{a}n's theorem, which provides a
solution to the problem of finding the maximum number of edges in a graph that does not contain a complete subgraph of a given
order \cite{Turan1941}. Before delving into formal definitions, it is noted that the distinction of the Tur\'{a}n graphs as
multipartite complete graphs is that they are as balanced as possible, ensuring their vertex sets are divided into parts of
nearly equal size.

\begin{definition}
Let $n_1, \ldots, n_k$ be natural numbers. Define the \emph{complete $k$-partite graph}
\begin{align}
\CoG{n_1, \ldots, n_k}= \bigvee_{i=1}^{k}\overline{\CoG{n_i}}.
\end{align}
A graph is multipartite if it is $k$-partite for some $k \geq 2$.
\end{definition}

\begin{definition}
\label{definition: Turan graph}
Let $2 \le k \le n$. The \emph{Tur\'{a}n graph} $T(n,k)$
(not to be confused with the graph of pyramids $T_{n,k}$) is
formed by partitioning a set of $n$ vertices into $k$ subsets,
with sizes as equal as possible, and then every two vertices
are adjacent in that graph if and only if they belong to different subsets.
It is therefore expressed as the complete $k$-partite graph
$K_{n_1,\dots,n_k}$, where $|n_i-n_j| \leq 1$ for all $i,j \in \OneTo{k}$
with $i \neq j$. Let $q$ and $s$ be the quotient and remainder, respectively,
of dividing $n$ by $k$ (i.e., $n = qk+s$, $s \in \{0,1, \ldots, k-1\}$),
and let $n_1 \leq \ldots \leq n_k$. Then,
\begin{align}
\label{eq: n_i in Turan's graph}
n_i=
\begin{cases}
q, & 1\leq i \leq k-s,\\
q+1, & k-s+1 \leq i \leq k.
\end{cases}
\end{align}
By construction, the graph $T(n,k)$ has a clique of order  $k$ (any subset of vertices with
a single representative from each of the $k$ subsets is a clique of order  $k$), but it cannot
have a clique of order  $k+1$ (since vertices from the same subset are nonadjacent).
Note also that, by \eqref{eq: n_i in Turan's graph}, the Tur\'{a}n graph $T(n,k)$ is an
$(n-q)$-regular graph if and only if $n$ is divisible by $k$, and then $q = \frac{n}{k}$.
\end{definition}

\begin{definition}
Let $q,k \in \naturals$. Define the \emph{regular complete multipartite graph},
$\mathrm{K}_{q}^{k}: = \overset{k}{\underset{i=1}{\bigvee}} \overline{\CoG{q}}$, to be the $k$-partite
graph with $q$ vertices in each part. Observe that $\mathrm{K}_{q}^{k} = T(kq,k)$.
\end{definition}

Let $\Gr{G}$ be a simple graph on $n$ vertices that does not contain a clique of order greater than
a fixed number $k \in \naturals$. Tur\'{a}n investigated a fundamental problem in extremal graph theory of
determining the maximum number of edges that $\Gr{G}$ can have \cite{Turan1941}.
\begin{theorem}[Tur\'{a}n's Graph Theorem]
\label{theorem: Turan's theorem}
Let $\Gr{G}$ be a graph on $n$ vertices with a clique of order at most $k$ for some $k \in \naturals$.
Then,
\begin{align}
\card{\E{\Gr{G}}} &\leq \card{\E{T(n,k)}} \\
&= \biggl(1-\frac1k\biggr) \, \frac{n^2-s^2}{2} + \binom{s}{2}, \quad s \triangleq n - k \bigg\lfloor \frac{n}{k} \bigg\rfloor.
\end{align}
\end{theorem}
For a nice exposition of five different proofs of Tur\'{a}n's Graph Theorem, the interested reader is referred
to Chapter~41 of \cite{AignerZ18}.

\begin{corollary}
\label{corollary:turan} Let $k \in \naturals$, and let $\Gr{G}$ be a graph
on $n$ vertices where $\omega(\Gr{G})\le k$ and $\card{\E{\Gr{G}}}=\card{\E{T(n,k)}}$.
Let $\Gr{G}_{1}$ be a graph obtained by adding an arbitrary edge to
$\Gr{G}$. Then $\omega(\Gr{G}_1)>k$.
\end{corollary}

\subsubsection{The spectrum of the Tur\'{a}n graph}

\begin{theorem} \cite{EsserH1980}
\label{theorem: spectrum of multipartite graphs}
Let $k\in\naturals$, and let $n_1 \leq n_2 \leq \ldots \leq n_k$ be natural numbers.
Let $\Gr{G} = \CoG{n_1,n_2, \dots, n_k}$ be a complete multipartite graph on $n = n_1 + \ldots n_k$
vertices. Then,
\begin{itemize}
\item $\Gr{G}$ has one positive eigenvalue, i.e., $\lambda_1(\Gr{G}) > 0$ and $\lambda_2(\Gr{G})\le 0$.
\item $\Gr{G}$ has $0$ as an eigenvalue with multiplicity $n-k$.
\item $\Gr{G}$ has $k-1$ negative eigenvalues, and
\begin{align}
n_1 \leq -\lambda_{n-k+2}(\Gr{G}) \leq n_2 \leq -\lambda_{n-k+3}(\Gr{G}) \le n_3 \leq \ldots \leq n_{k-1} \leq -\lambda_{n}(\Gr{G}) \le n_{k}.
\end{align}
\end{itemize}
\end{theorem}

\begin{corollary}
\label{corollary:Kqk-spectrum}
The spectrum of the regular complete $k$-partite graph $\CoG{q, \ldots, q} \triangleq \CoG{q}^k$ is given by
\begin{align}
\sigma_{\A}(\CoG{q}^{k})=\Bigl\{ [-q]^{k-1}, [0]^{(q-1)k}, q(k-1) \Bigr\}.
\end{align}
\end{corollary}

\begin{proof}
This readily follows from Theorem~\ref{theorem: spectrum of multipartite graphs} by setting $n_1 = \ldots = n_k = q$.
\end{proof}

\begin{lemma}
\label{lemma: Join-A-Spec} \cite{Butler2008} Let $\Gr{G}_{i}$
be $r_{i}$-regular graphs on $n_{i}$ vertices for $i\in \{1,2\}$, with the adjacency spectrum
$\sigma_{\A}(\Gr{G}_1)=(r_{1}=\mu_{1}\ge\mu_{2}\ge...\ge\mu_{n})$
and $\sigma_{A}(\Gr{G}_2) = (r_{2}=\nu_{1}\ge\nu_{2}\ge...\ge\nu_{n})$.
The $\A$-spectrum of $\Gr{G}_1\vee \Gr{G}_2$ is given by
\begin{align}
\sigma_{\A}(\Gr{G}_{1}\vee \Gr{G}_{2})=\{ \mu_{i} \}_{i=2}^{n_{1}} \cup \{ \nu_{i}\}_{i=2}^{n_{2}} \cup
\left\{ \frac{r_1+r_2 \pm\sqrt{(r_1-r_2)^{2}+4 n_1 n_2}}{2} \right\}.
\end{align}
\end{lemma}

\begin{theorem}
\label{theorem: A-spectrum of Turan graph}
Let $q,s\in \naturals$ such that $n=kq+s$ and $0 \le s \leq k-1.$ The following
holds with respect to the $\A$-spectrum of $T(n,k)$:
\begin{enumerate}
\item \label{item: irregular Turan graph}
If $1 \leq s \leq k-1$, then the $\A$-spectrum of the irregular Tur\'{a}n graph $T(n,k)$ is given by
\begin{align}
\sigma_{\A}(T(n,k))=& \biggl\{ [-q-1]^{s-1}, [-q]^{k-s-1}, [0]^{n-k} \biggr\} \nonumber \\
\label{eq: A-spectrum of irregular Turan graph}
& \cup \Biggl\{\tfrac12 \biggl[n-2q-1\pm \sqrt{\Bigl(n-2(q+1)s+1\Bigr)^2+4q(q+1)s(k-s)} \biggr] \Biggr\}.
\end{align}
\item \label{item: regular Turan graph}
If $s=0$, then $q = \frac{n}{k}$, and the $\A$-spectrum of the regular  Tur\'{a}n graph $T(n,k)$ is given by
\begin{align}
\label{eq: A-spectrum of regular Turan graph}
\sigma_{\A}(T(n,k))=& \Bigl\{ [-q]^{k-1}, [0]^{n-k}, (k-1)q \Bigr\}.
\end{align}
\end{enumerate}
\end{theorem}

\begin{proof}
Let $1 \leq s \leq k-1$, and we next derive the $\A$-spectrum of an irregular Tur\'{a}n graph $T(n,k)$ in
Item~\ref{item: irregular Turan graph} of this theorem (i.e., its spectrum if $n$ is not divisible by $k$ since $s \neq 0$).
By Corollary~\ref{corollary:Kqk-spectrum}, the spectra of the regular
graphs $\CoG{q}^{k-s}$ and $\CoG{q+1}^{s}$ is
\begin{align}
& \sigma_{\A}(\CoG{q}^{k-s})=\left\{ [-q]^{k-s-1}, [0]^{(q-1)(k-s)}, q(k-s-1) \right\},  \\
& \sigma_{\A}(\CoG{q+1}^{s})=\left\{ [-q-1]^{s-1}, [0]^{qs}, (q+1)(s-1) \right\}.
\end{align}
The $(k-s)$-partite graph $\CoG{q}^{k-s}$ is $r_1$-regular with $r_1=q(k-s-1)$, the
$s$-partite graph $\CoG{q+1}^{s}$ is $r_2$-regular with $r_2 = (q+1)(s-1)$, and
by Definition~\ref{definition: Turan graph}, we have $T(n,k) = \CoG{q}^{k-s} \vee \CoG{q+1}^{s}$.
Hence, by Lemma~\ref{lemma: Join-A-Spec}, the adjacency spectrum of $T(n,k)$ is given by
\begin{align}
\sigma_{\A}(T(n,k)) &= \sigma_{\A}(\CoG{q}^{k-s} \vee \CoG{q+1}^{s}) \nonumber \\
\label{eq0: 23.12.2024}
&=\set{S}_1 \cup \set{S}_2 \cup \set{S}_3,
\end{align}
where
\begin{align}
\label{eq1: 23.12.2024}
\set{S}_1 &= \Bigl\{ [-q]^{k-s-1}, [0]^{(q-1)(k-s)} \Bigr\}, \\
\label{eq2: 23.12.2024}
\set{S}_2 &= \Bigl\{ [-q-1]^{s-1}, [0]^{qs} \Bigr\}, \\
\set{S}_3 &= \biggl\{ \frac{r_1+r_2 \pm \sqrt{(r_1-r_2)^2 + 4 n_1 n_2}}{2} \biggr\} \nonumber \\
\label{eq3: 23.12.2024}
&= \Biggl\{\tfrac12 \biggl[n-2q-1\pm \sqrt{\Bigl(n-2(q+1)s+1\Bigr)^2+4q(q+1)s(k-s)} \biggr] \Biggr\},
\end{align}
where the last equality holds since, by the equality $n=kq+s$ and the above expressions of $r_1$
and $r_2$, it can be readily verified that $r_1+r_2 = n-2q-1$ and $r_1-r_2 = n-2(q+1)s+1$.
Finally, combining \eqref{eq0: 23.12.2024}--\eqref{eq3: 23.12.2024} gives the $\A$-spectrum
in \eqref{eq: A-spectrum of irregular Turan graph} of an irregular Tur\'{a}n graph $T(n,k)$.

We next prove Item~\ref{item: regular Turan graph} of this theorem, referring to a regular Tur\'{a}n graph $T(n,k)$
(i.e., $k|n$ or equivalently, $s=0$). In that case, we have $T(n,k)=\CoG{q}^{k}$ where $q = \frac{n}{k}$,
so the $\A$-spectrum in \eqref{eq: A-spectrum of regular Turan graph} holds by Corollary~\ref{corollary:Kqk-spectrum}.
\end{proof}

\begin{remark}
\label{remark: Turan - multiplicity of negative eigenvalues}
In light of Theorem~\ref{theorem: A-spectrum of Turan graph}, if $k \geq 2$, then the number of negative
eigenvalues (including multiplicities) of the adjacency matrix of the Tur\'{a}n graph $T(n,k)$ is $k-1$
if the graph is regular (i.e., if $k|n$), and it is $k-2$ otherwise (i.e., if the graph is irregular).
If $k=1$, which corresponds to an empty graph (having no edges), then all eigenvalues are zeros
(having no negative eigenvalues). Furthermore, the adjacency matrix of $T(n,k)$ always has a single positive
eigenvalue, which is of multiplicity~1 irrespectively of the values of $n$ and $k$. We rely on these properties
later in this paper (see Section~\ref{subsubsection: Turan graph is DS}).
\end{remark}

\begin{example}
\label{example: A-spectrum of Turan graph}
By Theorem~\ref{theorem: A-spectrum of Turan graph}, let us calculate the $\A$-spectrum of the
Tur\'{a}n graph $T(17,7)$, and verify it numerically with the SageMath software \cite{SageMath}. Having $n=17$
and $k=7$ gives $q=2$ and $s=3$, which by Theorem~\ref{theorem: A-spectrum of Turan graph} implies that
\begin{align}
\label{eq4: 23.12.2024}
\sigma_{\A}\bigl(T(17,7)\bigr)= \Bigl\{ [-3]^2, [-2]^3, [0]^{10}, \, 6(1+\sqrt{2}), \, 6(1-\sqrt{2}) \Bigr\}.
\end{align}
That has been numerically verified by programming in the SageMath software.
\end{example}

\subsubsection{Tur\'{a}n graphs are DS}
\label{subsubsection: Turan graph is DS}
The main result of this subsection establishes that all Tur\'{a}n graphs are determined by their $\A$-spectrum.
This result is equivalent to Theorem~3.3 in \cite{MaRen2010}, while also presenting an alternative proof that
offers additional insights.
\begin{theorem}
\label{theorem: Turan's graph is DS}
The Tur\'{a}n graph $T(n,k)$ is $\A$-DS.
\end{theorem}

In order to prove Theorem~\ref{theorem: Turan's graph is DS}, we first introduce an auxiliary result from
\cite{Smith1970}, followed by several other lemmata.
\begin{theorem} \cite[Theorem~1]{Smith1970}
\label{theorem:Smith_theorem_1}
Let $\Gr{G}$ be a graph. Then, the following statements are equivalent:
\begin{itemize}
\item $\Gr{G}$ has exactly one positive eigenvalue.
\item $\Gr{G}=\Gr{H} \cup \CGr{\CoG{m}}$ for some $m$, where $\Gr{H}$ is a nonempty
complete multipartite graph. In other words, the non-isolated vertices of $\Gr{G}$
form a complete multipartite graph.
\end{itemize}
\end{theorem}

\begin{proof}[Proof of Theorem~\ref{theorem: Turan's graph is DS}]
Let $\Gr{G}$ be a graph that is $\A$-cospectral with $T(n,k)$. Denote $n=qk+s$
for $s,q\in \naturals$ such that $0\le s<k$.
\begin{lemma}
\label{lemma:Tnk_DS_lemma_1}
The graph $\Gr{G}$ doesn't have a clique of order $k+1$.
\end{lemma}

\begin{proof}
Suppose to the contrary that the graph $\Gr{G}$ has a clique of order $k+1$, which means that
$\CoG{k+1}$ is an induced subgraph of $\Gr{G}$. The complete graph $\CoG{k+1}$ has
$k$ negative eigenvalues ($-1$ with a multiplicity of $k$). On the other hand, $\Gr{G}$ has
at most $k-1$ negative eigenvalues, $n-k$ zero eigenvalues, and exactly one positive eigenvalue;
indeed, this follows from Theorem~\ref{theorem: A-spectrum of Turan graph} (see
Remark~\ref{remark: Turan - multiplicity of negative eigenvalues}), and since $\Gr{G}$
and $T(n,k)$ are $\A$-cospectral graphs. Hence, by Cauchy's Interlacing Theorem,
every induced subgraph of $\Gr{G}$ on $k+1$ vertices has at most $k-1$ negative
eigenvalues (i.e., those eigenvalues interlaced between the negative and zero eigenvalues of $\Gr{G}$
that are placed at distance $k+1$ apart in a sorted list of the eigenvalues of $\Gr{G}$ in decreasing order).
This contradicts our assumption on the existence of a clique of $k+1$ vertices because of the $k$ negative
eigenvalues of $\CoG{k+1}$.
\end{proof}

\begin{lemma}
\label{lemma:Tnk_DS_lemma_2}
The graph $\Gr{G}$ is a complete multipartite graph.
\end{lemma}

\begin{proof}
Since $\Gr{G}$ has exactly one positive eigenvalue, which is of multiplicity one, we get
from Theorem~\ref{theorem:Smith_theorem_1} that $\Gr{G} = \Gr{H} \cup \CGr{\CoG{\ell}}$
for some $\ell \in \naturals$, where $\Gr{H}$ is a nonempty multipartite graph. We next
show that $\ell=0$. Suppose to the contrary that $\ell \geq 1$, and let $v$ be an isolated
vertex of $\CGr{\CoG{\ell}}$. Since $\Gr{H}$ is a nonempty graph, there exists a vertex
$u\in \V{\Gr{H}}$. Let $\Gr{G}_1$ be the graph obtained from $\Gr{G}$ by adding the
single edge $\{v,u\}$. By Lemma~\ref{lemma:Tnk_DS_lemma_1}, $\Gr{G}$ does not have a clique
of order $k+1$. Hence, $\Gr{G}_1$ does not have a clique of order $k+1$ either, contradicting
Corollary~\ref{corollary:turan}. Hence, $\Gr{G}=\Gr{H}$.
\end{proof}

\begin{lemma}
\label{lemma:Tnk_DS_lemma_3}
The graph $\Gr{G}$ is a complete $k$-partite graph.
\end{lemma}

\begin{proof}
By Lemma \ref{lemma:Tnk_DS_lemma_2}, $\Gr{G}$ is a complete multipartite graph. Let
$r$ be the number of partite subsets in the vertex set $\V{\Gr{G}}$. We show that $r=k$,
which then gives that $\Gr{G}$ is a complete $k$-partite graph. By
Lemma~\ref{lemma:Tnk_DS_lemma_1}, $\Gr{G}$ doesn't have a clique of order $k+1$.
Hence, $r\le k$. Suppose to the contrary that $r<k$. Since $\Gr{G}$ is a complete
$r$-partite graph, the largest order of a clique in $\Gr{G}$ is $r$.
Let $\Gr{G}_{1}$ be a graph obtained from $\Gr{G}$ by adding an edge between
two vertices within the same partite subset. The graph $\Gr{G}_1$ becomes an $(r+1)$-partite
graph. Consequently, the maximum order of a clique in $\Gr{G}_1$ is at most $r+1 \leq k$.
The graph $\Gr{G}_1$ has exactly one more edge than $\Gr{G}$. Since $\Gr{G}$ is
$\A$-cospectral to $T(n,k)$, it has the same number of edges as in $T(n,k)$.
Hence, $\Gr{G}_1$ contains more edges than $T(n,k)$, while also lacking a clique of order $k+1$.
This contradicts Corollary~\ref{corollary:turan}, so we conclude that $r=k$.
\end{proof}

Let $n_1, n_2, \dots , n_k \in \naturals$ be the number of vertices in each partite subset
of the complete $k$-partite graph $\Gr{G}$, i.e., $\Gr{G} = \CoG{n_1, n_2, \dots , n_k}$.
Then, the next two lemmata subsequently hold.
\begin{lemma}
\label{lemma:Tnk_DS_lemma_4}
For all $i \in \OneTo{k}$, $n_i \le q+1$ .
\end{lemma}

\begin{proof}
Suppose to the contrary that there exists a partite subset in the complete $k$-partite graph
$\Gr{G}$ with more than $q+1$ vertices. Let $\set{P}_1$ be such a partite subset, and suppose
without loss of generality that $n_1= \card{\set{P}_1} \geq q+1$. By the pigeonhole principle, there
exists a partite subset of $\Gr{G}$ with at most $q$ vertices (since $\sum_{i \in \OneTo{k}} n_i = n = kq+s$,
where $0\leq s \leq k-1$). Let $\set{P}_2$ be such a partite subset of $\Gr{G}$, and suppose
without loss of generality that $n_2 = \card{\set{P}_2} \leq q$. Let $\Gr{G}_{1}$
be the graph obtained from $\Gr{G}$ by removing a vertex $v \in \set{P}_1$, adding
a new vertex $u$ to $\set{P}_2$, and connecting $u$ to all the vertices outside its partite subset.
The new graph $\Gr{G}_1$ is also $k$-partite, so it does not contain a clique of order greater than $k$.
Furthermore, by construction, $\Gr{G}_{1}$ has more edges than $\Gr{G}$, so
\begin{align}
\label{eq1: 24.12.2024}
\bigcard{\E{\Gr{G}_1}} > \bigcard{\E{\Gr{G}}} = \bigcard{\E{T(n,k)}}.
\end{align}
Hence, $\Gr{G}_{1}$ is a graph with no clique of order greater than $k$,
and it has more edges than $T(n,k)$. That contradicts Theorem~\ref{theorem: Turan's theorem},
so $\{n_i\}_{i=1}^k$ cannot include any element that is larger than $q+1$.
\end{proof}

\begin{lemma}
\label{lemma:Tnk_DS_lemma_5}
For all $i \in \OneTo{k}$, $n_i \geq q$.
\end{lemma}

\begin{proof}
The proof of this lemma is analogous to the proof of Lemma~\ref{lemma:Tnk_DS_lemma_4}.
Suppose to the contrary that there exists a partite subset of $\Gr{G}$ with less than $q$
vertices. Let $\set{P}_1$ be such a partite subset, so $p_1 \triangleq \bigcard{\set{P}_1}<q$.
By the pigeonhole principle, there exists a partition with at least $q+1$ vertices.
Let $\set{P}_2$ be such a partite subset set, whose number of vertices is denoted by
$p_{2} \triangleq \card{\set{P}_2} \geq q+1$.
Let $\Gr{G}_{1}$ be the graph obtained by removing a vertex $v \in \set{P}_2$,
adding a new vertex $u$ to $\set{P}_1$, and connecting the vertex $u$ to all
the vertices outside its partite subset. $\Gr{G}_{1}$ is $k$-partite, so it does
not contain a clique of order greater than $k$, and $\Gr{G}_{1}$ has more edges
than $\Gr{G}$ so \eqref{eq1: 24.12.2024} holds.
Hence, $\Gr{G}_1$ is a graph with no clique of order greater than $k$,
and it has more edges than $T(n,k)$. That contradicts Theorem~\ref{theorem: Turan's theorem},
so $\{n_i\}_{i=1}^k$ cannot include any element that is smaller than $q$.
\end{proof}

By Lemmata~\ref{lemma:Tnk_DS_lemma_4} and~\ref{lemma:Tnk_DS_lemma_5}, we conclude that
$n_i \in \{q,q+1\} $ for every $1 \le i \le k-1$. Let $\alpha$ be the number of partite
subsets of $q$ vertices and $\beta$ be the number of partite subsets of $q+1$ vertices.
Since $\Gr{G}$ has $n$ vertices, where $\sum n_{i}=n$, it follows that $q\alpha + (q+1)\beta = n$.
Moreover, $\Gr{G}$ is $k$-partite, so it follows that $\alpha+\beta=k$. This gives the linear
system of equations
\begin{align}
\begin{pmatrix}q & q+1\\
1 & 1
\end{pmatrix}\begin{pmatrix}\alpha\\
\beta
\end{pmatrix}=\begin{pmatrix}n\\
k
\end{pmatrix},
\end{align}
which has the single solution
\begin{align}
\alpha = k-s, \quad \beta=n-qk=s.
\end{align}
Hence, $\Gr{G} = T(n,k)$, which completes the proof of Theorem~\ref{theorem: Turan's graph is DS}.
\end{proof}

\begin{remark}
\label{remark: on the alternative proof by Ma and Ren}
The proof of Theorem~\ref{theorem: Turan's graph is DS} is an alternative proof of Theorem~3.3 in \cite{MaRen2010}.
While both proofs rely on Theorem~\ref{theorem:Smith_theorem_1}, which is Theorem~1 of \cite{Smith1970}, our proof
relies on the adjacency spectral characterization in Theorem~\ref{theorem: A-spectrum of Turan graph}, noteworthy
in its own right, and further builds upon a sequence of results presented in Lemmata~\ref{lemma:Tnk_DS_lemma_1}--\ref{lemma:Tnk_DS_lemma_5}.
On the other hand, the proof of Theorem~3.3 in \cite{MaRen2010} relies on Theorem~\ref{theorem:Smith_theorem_1},
but then deviates substantially from our proof (see Lemmata~2.4 and~2.5 in \cite{MaRen2010} and Theorem~3.1 in \cite{MaRen2010},
serving to prove Theorem~\ref{theorem: Turan's graph is DS}).
\end{remark}

\subsection{Line graphs}
\label{subsection: Line graphs}

Among the various studied transformations on graphs, the line graphs of graphs are one of the most studied transformations
\cite{BeinekeB21}. We first introduce their definition, and then address the spectral graph determination properties.
\begin{definition}
\label{definition: line graph}
The {\em line graph} of a graph $\Gr{G}$, denoted by $\ell(\Gr{G})$, is a graph whose vertices
are the edges in $\Gr{G}$, and two vertices are adjacent in $\ell(\Gr{G})$ if the corresponding
edges are incident in $\Gr{G}$.
\end{definition}
A notable spectral property of line graphs is that all the eigenvalues of their adjacency matrix
are greater than or equal to~$-2$ (see, e.g., \cite[Theorem~4.6]{BeinekeB21}). For the determination
of all graphs whose spectrum is bounded from below by $-2$, the interested reader is referred to
\cite[Section~4.5]{BeinekeB21}.

The following theorem characterizes some families of line graphs that are DS.
\begin{theorem}
\label{theorem: DS line graphs}
The following line graphs are DS:
\begin{enumerate}[1]
\item \label{Item 1: DS line graphs}
The line graph of the complete graph $\CoG{k}$, where $k \geq 4$ and $k \neq 8$ (see
\cite[Theorem~4.1.7]{CvetkovicRS2010}),
\item \label{Item 2: DS line graphs}
The line graph of the complete bipartite graph $\CoBG{k}{k}$, where $k \geq 2$ and
$k \neq 4$ (see \cite[Theorem~4.1.8]{CvetkovicRS2010}),
\item \label{Item 3: DS line graphs}
The line graph $\ell(\CGr{\CG{6}})$ (see \cite[Proposition~4.1.5]{CvetkovicRS2010}),
\item \label{Item 4: DS line graphs}
The line graph of the complete bipartite graph $\CoBG{m}{n}$, where $m+n \geq 19$ and
$\{m,n\} \neq \{2s^2+s, 2s^2-s\}$ with $s \in \naturals$ (see
\cite[Proposition~4.1.18]{CvetkovicRS2010}).
\end{enumerate}
\end{theorem}

\begin{remark}
\label{remark: triangular graphs}
In regard to Item~\ref{Item 1: DS line graphs} of Theorem~\ref{theorem: DS line graphs},
the line graphs of complete graphs are referred to as triangular graphs. These
are strongly regular graphs with the parameters $\srg{\tfrac12 k(k-1)}{2(k-2)}{k-2}{4}$,
where $k \geq 4$. For $k=8$, the corresponding triangular graph is cospectral and nonisomorphic
to the three Chang graphs (see Remark~\ref{remark: NICS SRGs}), which are strongly regular graphs
$\srg{28}{12}{6}{4}$.
\end{remark}

We next prove the following result in regard to the Petersen graph, which appears in \cite[Problem~4.3]{CvetkovicRS2010}
and \cite[Section~10.3]{BrouwerM22} (without a proof).
\begin{corollary}
\label{corollary: Petersen graph is DS}
The Petersen graph is DS.
\end{corollary}
\begin{proof}
The Petersen graph is known to be isomorphic to the complement of the line graph of the complete graph $\CoG{5}$ (i.e., it
is isomorphic to $\CGr{\ell(\CoG{5})}$. By Item~\ref{Item 1: DS line graphs} of Theorem~\ref{theorem: DS line graphs}, the line
graph $\ell(\CoG{5})$ is DS. It is also a 6-regular graph (as the line graph of a $d$-regular graph is $(2d-2)$-regular, and
$\CoG{5}$ is a 4-regular graph). Consequently, by Item~\ref{item 5: DS graphs} of Theorem~\ref{theorem: special classes of DS graphs},
the complement of $\ell(\CoG{5})$ is also DS.
\end{proof}

The following definition and theorem provide further elaboration on Item~\ref{Item 2: DS line graphs}
of Theorem~\ref{theorem: DS line graphs}.
\begin{definition} \cite[Section~1.1.8]{BrouwerM22}
\label{definition: Lattice graphs}
The {\em Hamming graph} $\mathrm{H}(2,q)$, where $q \geq 2$, has the vertex set $\OneTo{q} \times \OneTo{q}$, and
any two vertices are adjacent if and only if they differ in one coordinate (i.e., their Hamming distance
is equal to~1). These are also referred to {\em lattice graphs}, and denoted by $\mathrm{L}_2(q)$. The
Lattice graph $\mathrm{L}_2(q)$, where $q \geq 2$, is also the line graph of the complete bipartite
graph $\CoBG{q}{q}$, and it is a strongly regular graph with parameters $\srg{q^2}{2(q-1)}{q-2}{2}$.
\end{definition}

\begin{theorem} \cite{Shrikhande59}
\label{theorem: DS Lattice graphs}
The lattice graph $\mathrm{L}_2(q)$ is a strongly regular DS graph for all $q \neq 4$. For $q=4$, the graph $\mathrm{L}_2(4)$
is not DS since it is cospectral and nonisomorphic to the Shrikhande graph, which are nonisomorphic strongly regular graphs
with the common parameters $\srg{16}{6}{4}{2}$.
\end{theorem}

The following result provides an interesting connection between $\A$ and $\Q$-cospectralities of graphs.
\begin{theorem} \cite[Proposition~7.8.5]{CvetkovicRS2010}
If two graphs are $\Q$-cospectral, then their line graphs are $\A$-cospectral.
\end{theorem}

\subsection{Nice graphs}
\label{subsection: Nice graphs}

The family referred to as "nice graphs" has been recently introduced in \cite{KovalK2024}.
\begin{definition}
\label{definition: nice graphs}
\noindent
\begin{itemize}
\item A graph $\Gr{G}$ is \emph{sunlike} if it is connected, and can be obtained from a cycle
$\mathrm{C}$ by adding some vertices and connecting each of them to some vertex in $\mathrm{C}$.
\item Let $l,k\in \naturals$. A sunlike graph $\Gr{G}$ is \emph{$(l,k)$-nice} if it can be obtained by a cycle $\CG{\ell}$ and
\begin{itemize}
\item There is a single vertex $v_1 \in \V{\CG{\ell}}$ of degree $3$.
\item There are $k$ vertices $u_1, \ldots, u_k \in \V{\CG{\ell}}$ of degree $4$. Let $\set{U} = \{u_1, \ldots, u_k\}$.
\item By starting a walk on $\CG{\ell}$ from $v_1$ at some orientation, then after 4 or 6 steps we get to a vertex $u_1 \in \set{U}$.
Then, after another $4$ or $6$ steps from $u_1$ we get to $u_2 \in \set{U}$, and so on until we get to the vertex $u_k \in \set{U}$.
\end{itemize}
\end{itemize}
\end{definition}

\begin{theorem} \cite{KovalK2024}
\label{theorem: Koval and Kwan, 2024}
Let $l,k \in \naturals$ such that $l \equiv 2 \, (\text{mod } 4)$. Let $\Gr{G}$ be an $(l,k)$-nice graph.
If the order of $\Gr{G}$ is a prime number greater than some $n_0 \in \naturals$, then the line graph
$\ell(\Gr{G})$ is DS.
\end{theorem}

A more general class of graphs is introduced in \cite{KovalK2024}, where it is shown that for every sufficiently
large $n \in \naturals$, the number of nonisomorphic $n$-vertex DS graphs is at least $e^{cn}$ for some positive
constant $c$ (see \cite[Theorem~1.4]{KovalK2024}). This recent result represents a significant advancement in the
study of Haemers' conjecture because the earlier lower bounds on the number of nonisomorphic $n$-vertex DS graphs
were all of the form of $e^{c \sqrt{n}}$, for some positive constant $c$. As noted in \cite{KovalK2024}, the first
form of such a lower bound was derived by van Dam and Haemers \cite[Proposition~6]{vanDamH09}, who proved that a
graph $\Gr{G}$ is DS if every connected component of $\Gr{G}$ is a complete subgraph, leading to a lower bound that
is approximately of the form $e^{c \sqrt{n}}$ with $c = \sqrt{\tfrac{2}{3}} \, \pi$.
Therefore, the transition to a lower bound in \cite{KovalK2024} that scales exponentially with $n$, rather than with
$\sqrt{n}$, is both remarkable and noteworthy.

\subsection{Friendship graphs and their generalization}
\label{subsection: The friendship graphs}

The next theorem considers whether friendship graphs (see Definition~\ref{definition: friendship graph})
can be uniquely determined by the spectra of four of their associated matrices.

\begin{theorem}
\label{theorem: friendship graphs are X-DS}
The friendship graph $\FG{p}$ satisfies the following properties:
It is DS if and only if $p \ne 16$ (i.e., the friendship graph is DS unless
it has 16~triangles) \cite{CioabaHVW2015}, $\LM$-DS \cite{LinSM2010},
$\Q$-DS \cite{WangBHB2010}, and ${\bf{\mathcal{L}}}$-DS \cite{BermanCCLZ2018}.
\end{theorem}

The friendship graph $\FG{p}$, where $p \in \naturals$, can be expressed in the
form $\FG{p}=\CoG{1} \vee (p\CoG{2})$ (see Figure~\ref{fig:friendship graph F4}).
The last observation follows from a property of a generalized friendship graph, which is defined as follows.
\begin{definition}
Let $p,q\in \naturals$. The \emph{generalized friendship graph} is given by
$\GFG{p}{q} = \CoG{1} \vee (p \CoG{q})$. Note that $\GFG{p}{2} = \FG{p}$.
\end{definition}

The following theorem addresses the conditions under which generalized friendship graphs can be uniquely
determined by the spectra of their normalized Laplacian matrix.
\begin{theorem}
\label{theorem: Berman's generalized friendship graph}
The generalized friendship graph $\GFG{p}{q}$ is ${\bf{\mathcal{L}}}$-DS if and only if $q \ge 2$, or $q=1$
and $p=2$ \cite{BermanCCLZ2018}.
\end{theorem}

\begin{corollary}
\label{corollary: friendship graph is DS w.r.t. normalized Laplacian}
The friendship graph $\FG{p}$ is ${\bf{\mathcal{L}}}$-DS \cite{BermanCCLZ2018}.
\end{corollary}

\subsection{Strongly regular graphs}
\label{subsection: strongly regular graphs}

Strongly regular graphs with an identical vector of parameters $(n,d,\lambda,\mu)$ are cospectral, but may not be isomorphic (see,
Corollary~\ref{corollary: cospectral SRGs}, Remark~\ref{remark: NICS SRGs}, and Theorem~\ref{theorem: DS Lattice graphs}).
For that reason, strongly regular graphs are not necessarily determined by their spectrum. There are, however, infinite
families of strongly regular DS graphs:
\begin{theorem} \cite[Proposition~14.5.1]{BrouwerH2011}
\label{theorem: DS SRG}
If $n \neq 8$, $m \neq 4$, $a \geq 2$, and $\ell \geq 2$, then the disjoint union of identical complete graphs $a \CoG{\ell}$,
the line graph of a complete graph $\ell(\CoG{n})$, and the line graph of a complete bipartite graph with partite sets of equal
size $\ell(\CoBG{m}{m})$, as well as their complements, are strongly regular DS graphs.
\end{theorem}

We next show that, although connected strongly regular graphs are not generally DS, the property of strong regularity,
as well as the four parameters that characterize strongly regular graphs can be determined by the spectrum of their adjacency
matrix.

\begin{theorem}
\label{theorem: on A-spectrum and SRGs}
Let $\Gr{G}$ be a connected strongly regular graph. Then, its strong regularity, the vector of parameters $(n,d,\lambda,\mu)$,
Lov\'{a}sz $\vartheta$-function $\vartheta(\Gr{G})$, number of edges and triangles, girth, and diameter can be all determined
by its $\A$-spectrum.
\end{theorem}
\begin{proof}
The order of a graph $n$ is determined by the $\A$-spectrum, being the number of eigenvalues (including multiplicities).
By Theorem~\ref{theorem: graph regularity from A-spectrum}, the regularity of a graph is determined by its $\A$-spectrum.
By Item~\ref{Item 5: eigenvalues of srg} of Theorem~\ref{theorem: eigenvalues of srg}, a connected regular graph is
strongly regular if and only if it has three distinct eigenvalues. Hence, the strong regularity property of $\Gr{G}$
is determined by its $\A$-spectrum. For such a connected regular, the largest eigenvalue is simple, $\lambda_1=d$, and
the other two distinct eigenvalues of the adjacency matrix of $\Gr{G}$ are given by $\lambda_2$ and $\lambda_n$ with
$\lambda_n<\lambda_2$.
We next show that the number of common neighbors of any pair of adjacent vertices ($\lambda$), and the number of common
neighbors of any pair of nonadjacent vertices ($\mu$) in $\Gr{G}$ are, respectively, given by
\begin{align}
\label{eq1:23.09.23}
& \lambda = \lambda_1 + (1+\lambda_2)(1+\lambda_n) - 1, \\
\label{eq2:23.09.23}
& \mu = \lambda_1 + \lambda_2 \lambda_n,
\end{align}
so, these parameters are explicitly expressed in terms of the adjacency spectrum of the strongly regular graph. Indeed, by
Theorem~\ref{theorem: eigenvalues of srg}, the second-largest and least eigenvalues of the adjacency matrix of $\Gr{G}$ are given by
\begin{align}
\label{eq1:21.01.25}
\begin{dcases}
\lambda_2 = \tfrac12 \Bigl( \lambda -  \mu + \sqrt{(\lambda-\mu)^2 + 4(d-\mu)} \Bigr), \\[0.1cm]
\lambda_n = \tfrac12 \Bigl( \lambda -  \mu - \sqrt{(\lambda-\mu)^2 + 4(d-\mu)} \Bigr),
\end{dcases}
\end{align}
from which it follows that (noting that $d = \lambda_1$)
\begin{align}
\label{eq2:21.01.25}
\begin{dcases}
\lambda_2 + \lambda_n = \lambda - \mu, \\
\lambda_2 \lambda_n = \mu - \lambda_1.
\end{dcases}
\end{align}
This gives \eqref{eq1:23.09.23} and \eqref{eq2:23.09.23} from, respectively, the second equality in \eqref{eq2:21.01.25}
and by adding the two equalities in \eqref{eq2:21.01.25}.

The Lov\'{a}sz $\vartheta$-function of a strongly regular graph is given by (see \cite[Proposition~1]{Sason23})
\begin{align}
\label{eq: Lovasz SRG}
\vartheta(\Gr{G}) = -\frac{n \lambda_n}{d - \lambda_n},
\end{align}
so $\vartheta(\Gr{G})$ is determined by its $\A$-spectrum since the strong regularity property of $\Gr{G}$ was first determined.
The number of edges of $\Gr{G}$ of a $d$-regular graph is given by $\tfrac12 nd$, and the number of triangles of the strongly regular
graph $\Gr{G}$ is given by $\tfrac16 nd \lambda$, so they are both determined once the four parameters of the strongly regular graphs are
revealed. The diameter of a connected strongly regular graph is equal to~2 (note that complete graphs are excluded from the family of
strongly regular graphs). Finally, the girth of the strongly regular graph $\Gr{G}$ is determined as follows \cite{CameronL2010}:
\begin{enumerate}
\item If $\lambda > 0$, then the girth of $\Gr{G}$ is equal to~3;
\item If $\lambda=0$ and $\mu \geq 2$, then the girth of $\Gr{G}$ is equal to~4;
\item If $\lambda=0$ and $\mu=1$, then the girth of $\Gr{G}$ is equal to~5.
\end{enumerate}
\end{proof}

\begin{remark}
\label{remark: connected SRG}
A strongly regular graph is connected if and only if $\mu > 0$.
\end{remark}

By \cite[Proposition~2]{vanDamH03}, no pair of $\A$-cospectral graphs exists where one graph is regular, and
the other is not. The following result extends this observation to strong regularity.
\begin{corollary}
\label{corollary: cospectral graphs - one is SRG}
There are no two $\A$-cospectral connected graphs where one is strongly regular and the other is not.
\end{corollary}
\begin{proof}
For a connected strongly regular graph, the strong regularity is determined by the $\A$-spectrum.
\end{proof}

Another corollary that follows from Theorem~\ref{theorem: on A-spectrum and SRGs} applies to strongly regular DS graphs.
\begin{corollary}
\label{corollary: SRG DS graphs}
Let $\Gr{G}$ be a connected strongly regular graph such that there is no other nonisomorphic strongly regular graph
with an identical vector of parameters $(n, d, \lambda, \mu)$. Then, $\Gr{G}$ is a DS graph.
\end{corollary}
Corollary~\ref{corollary: SRG DS graphs} naturally raises the following question.
\begin{question}
\label{question: DS SRG}
Which connected strongly regular graphs are determined by their vector of parameters $(n, d, \lambda, \mu)$?
\end{question}
A partial answer to Question~\ref{question: DS SRG} is provided below.

By Corollary~\ref{corollary: SRG DS graphs}, for connected strongly regular graphs, there exists an equivalence between
their spectral determination (due to their regularity and in light of Theorem~\ref{theorem: regular DS graphs}, based on
the spectrum of their adjacency, Laplacian, or signless Laplacian matrices) and the uniqueness of these graphs for the given
parameter vector $(n, d, \lambda, \mu)$.

The study of the number of nonisomorphic strongly regular graphs corresponding to a given set of parameters has been extensively
explored. For example, by Theorem~\ref{theorem: DS Lattice graphs}, there is a unique (up to isomorphism) strongly regular graph
of the form $\srg{q^2}{2(q-1)}{q-2}{2}$ for any given $q \geq 2$ with $q \neq 4$. Specifically, this implies the uniqueness of
$\srg{36}{10}{4}{2}$ (setting $q=6$). On the other hand, a computer search by McKay and Spence established that there are
$32,548$ strongly regular graphs of the form $\srg{36}{15}{6}{6}$, so none of them is DS (by Corollary~\ref{corollary: SRG DS graphs}).

Further results on the uniqueness or non-uniqueness of strongly regular graphs with a given parameter vector $(n, d, \lambda, \mu)$
can be found in \cite{Cameron2003} and the references therein. Infinite families of strongly regular DS graphs are presented in
Theorem~\ref{theorem: DS SRG}. Some known sporadic strongly regular DS graphs are listed in \cite[Table~2]{vanDamH03}, with an
update in \cite[Table~1]{vanDamH09}). The uniqueness of further strongly regular graphs with given parameter vectors, making them
therefore DS graphs, was established, e.g., in \cite{Brouwer83,BrouwerH92,StevanovicM2009,Coolsaet2006,Mesner56}.

A combination of \cite[Table~1]{vanDamH09} and Corollary~\ref{corollary: Petersen graph is DS} implies that, apart from
complete graphs on fewer than three vertices and all complete bipartite regular graphs (which are known to be DS, as stated in
Item~\ref{item 2: DS graphs} of Theorem~\ref{theorem: special classes of DS graphs}), also all the seven currently known triangle-free
strongly regular graphs (see \cite{StruikB2010}) are DS. These include:
\begin{itemize}
\item The Pentagon graph $\CG{5}$ that is $\srg{5}{2}{0}{1}$ (by Item~\ref{item 2: DS graphs} of Theorem~\ref{theorem: special classes of DS graphs},
and see \cite[Section~10.1]{BrouwerM22}),
\item The Petersen graph $\srg{10}{3}{0}{1}$ (by Corollary~\ref{corollary: Petersen graph is DS}, and see \cite[Section~10.3]{BrouwerM22}),
\item Clebsch graph $\srg{16}{5}{0}{2}$ (see \cite[Section~10.7]{BrouwerM22}),
\item Hoffman-Singleton srg $\srg{50}{7}{0}{1}$  (see \cite[Section~10.19]{BrouwerM22}),
\item Gewirtz graph $\srg{56}{10}{0}{2}$ (see \cite[Section~10.20]{BrouwerM22}),
\item Mesner ($\mathrm{M}_{22}$) graph $\srg{77}{16}{0}{4}$ (see \cite[Section~10.27]{BrouwerM22} and \cite{Brouwer83}),
\item Higman-Sims graph $\srg{100}{22}{0}{6}$ (see \cite[Section~10.31]{BrouwerM22} and \cite{Mesner56}).
\end{itemize}

An up-to-date list of strongly regular DS graphs --- strongly regular graphs that are uniquely determined by their parameter
vectors --- as well as the number of strongly regular NICS graphs for given parameter vectors, is available on Brouwer's website
\cite{Brouwer}. An exclamation mark placed to the left of a parameter vector $(n,d,\lambda,\mu)$, without a preceding number, indicates
a strongly regular DS graph. In contrast, an exclamation mark preceded by a natural number greater than~1 specifies the number
of strongly regular NICS graphs with the corresponding parameter vector. For example, as shown in \cite{Brouwer}, strongly regular
graphs with the parameter vectors $(13,6,2,3)$, $(15,6,1,3)$, $(17,8,3,4)$, and $(21,10,3,6)$, among others, are DS graphs. On
the other hand, according to \cite{Brouwer}, there are 15 strongly regular NICS graphs with the parameter vector $(25,12,5,6)$,
10 strongly regular NICS graphs with the parameter vector $(26,10,3,4)$, and so forth.

To conclude, as strongly regular NICS graphs are not DS, $\LM$-DS, or $\Q$-DS, we were recently informed of ongoing research by
Cioaba \textit{et al.} \cite{CioabaGJM25}, which investigates the spectral properties of higher-order Laplacian matrices associated
with these graphs. This research demonstrates that the spectra of these new matrices can distinguish some of the strongly regular
NICS graphs. However, in other cases, strongly regular NICS graphs remain indistinguishable even with the spectra of these
higher-order Laplacian matrices.

\section{Graph operations for the construction of cospectral graphs}
\label{section: graph operations}

This section presents such graph operations, focusing on unitary and binary transformations that enable the systematic
construction of cospectral graphs. These operations are designed to preserve the spectral properties of the original
graph while potentially altering its structure, thereby producing non-isomorphic graphs with identical eigenvalues.
By employing these techniques, one can generate diverse examples of cospectral graphs, offering valuable tools for
investigating the limitations of spectral characterization and exploring the boundaries between graphs that are or
are not determined by their spectrum, which then relates the scope of the present section to Section~\ref{section: special families of graphs}
that deals with graphs or graph families that are determined by their spectrum.

\subsection{Coalescence}
\label{subsection: Coalescence}

\noindent

A construction of cospectral trees has been offered in \cite{Schwenk1973}, implying that almost all trees are not DS.

\begin{definition}
\label{definition: Coalescence}
Let $\Gr{G}_1 = (V_1,E_1), \Gr{G}_2=(V_2,E_2)$ be two graphs with $n_1$, $n_2$ vertices, respectively.
Let $v_1 \in V_1$ and $v_2 \in V_2$ be an arbitrary choice of vertices in both graphs.
The \emph{coalescence of $\Gr{G}_1$ and $\Gr{G}_2$ with respect to $v_1$ and $v_2$} is the graph with
$n_1 + n_2 -1$ vertices, obtained by the union of $\Gr{G}_1$ and $\Gr{G}_2$ where $v_1$ and $v_2$ are
identified as the same vertex in the united graph.
\end{definition}

\begin{theorem}
\label{theorem: Coalescence}
Let $\Gr{G}_1 = (V_1,E_1), \Gr{G}_2=(V_2,E_2)$ be two cospectral graphs, and let $v_1 \in V_1$ and $v_2 \in V_2$
be an arbitrary choice of vertices in both graphs.
Let $\Gr{H}_1$ and $\Gr{H}_2$ be the subgraphs of $\Gr{G}_1$ and $\Gr{G}_2$ that are induced by $V_1 \setminus \{v_1\}$
and $V_2 \setminus \{v_2\}$, respectively.
Let $\Gamma$ be a graph and $u\in V(\Gamma)$. If $\Gr{H}_1$ and $\Gr{H}_2$ are cospectral, then the coalescence of
$\Gr{G}_1$ and $\Gamma$ with respect to $v_1$ and $u$ is cospectral to the coalescence of $\Gr{G}_2$ and $\Gamma$
with respect to $v_2$ and $u$.
\end{theorem}

Combinatorial arguments that rely on the coalescence operation on graphs lead to a striking asymptotic result in
\cite{Schwenk1973}, stating that the fraction of $n$-vertex trees with cospectral and nonisomorphic mates, which
are also trees, approaches one as $n$ tends to infinity. Consequently, the fraction of the $n$-vertex nonisomorphic
trees that are determined by their spectrum (DS) approaches zero as $n$ tends to infinity. In other words, this means
that almost all trees are not DS (with respect to their adjacency matrix) \cite{Schwenk1973}.

\subsection{Seidel switching}
\label{subsection: Seidel switching}

\noindent

Seidel switching is one of the well-known methods for the construction of cospectral graphs.

\begin{definition}
\label{definition: Seidel switching}
Let $\Gr{G}$ be a graph, and let $\set{U} \subseteq \V{\Gr{G}}$. Constructing a graph $\Gr{G}_{\set{U}}$ by preserving
all the edges in $\Gr{G}$ between vertices within $\set{U}$, as well as all edges in $\Gr{G}$ between vertices within
the complement set $\cset{U} = \V{\Gr{G}} \setminus \set{U}$, while modifying adjacency and nonadjacency between any
two vertices where one is in $\set{U}$ and the other is in $\cset{U}$, is referred to (up to isomorphism) as
\emph{Seidel switching} of $\Gr{G}$ with respect to $\set{U}$.
\end{definition}
By Definition~\ref{definition: Seidel switching}, the Seidel switching of $\Gr{G}$ with respect to $\set{U}$ is equivalent
to its Seidel switching with respect to $\cset{U}$.
Let $\A(\Gr{G})$ and $\A(\Gr{G}_{\set{U}})$ be the adjacency matrices of a graph $\Gr{G}$ and its Seidel switching
$\Gr{G}_{\set{U}}$, and let $\A_{\set{U}}$ and $\A_{\cset{U}}$ be the matrices of $\A(\Gr{G})$ that, respectively,
refer to the adjacency matrices of the subgraphs of $\Gr{G}$ induced by $\set{U}$ and $\cset{U}$. Then, for some
$\mathbf{B} \in \{0,1\}^{\card{\cset{U}} \times \card{\set{U}}}$, we get
\begin{align}
\A(\Gr{G})=
\begin{pmatrix}
\A_{\set{U}} & \mathbf{B}^{\mathrm{T}}\\
\mathbf{B} & \A_{\cset{U}}
\end{pmatrix},
\end{align}
and by Definition~\ref{definition: Seidel switching},
\begin{align}
\A(\Gr{G}_{\set{U}})=
\begin{pmatrix}
\A_{\set{U}} & \overline{\mathbf{B}}^{\mathrm{T}}\\
\overline{\mathbf{B}} & \A_{\cset{U}}
\end{pmatrix},
\end{align}
where $\overline{\mathbf{B}}$ is obtained from $\mathbf{B}$ by interchanging zeros and ones. If $\Gr{G}$
is a regular graph, the following necessary and sufficient condition for $\Gr{G}_{\set{U}}$ to be a regular
graph of the same degree of its vertices.
\begin{theorem} \cite[Proposition~1.1.7]{CvetkovicRS2010}
\label{theorem: d-regular graphs in Seidel switching}
Let $\Gr{G}$ be a $d$-regular graph on $n$ vertices. Then, $\Gr{G}_{\set{U}}$ is also $d$-regular if and
only if $\set{U}$ induces a regular subgraph of degree $k$, where $\card{\set{U}} = n - 2(d-k)$.
\end{theorem}
The next result shows the relevance of Seidel switching for the construction of regular and cospectral graphs.
\begin{theorem} \cite[Proposition~1.1.8]{CvetkovicRS2010}
\label{theorem: Seidel switching}
Let $\Gr{G}$ be a $d$-regular graph, $\set{U} \subseteq \V{\Gr{G}}$, and let $\Gr{G}_{\set{U}}$ be obtained
from $\Gr{G}$ by Seidel switching. If $\Gr{G}_{\set{U}}$ is also a $d$-regular graph, then $\Gr{G}$ and
$\Gr{G}_{\set{U}}$ are cospectral (and due to their regularity, they are $\mathcal{X}$-cospectral for
every $\mathcal{X} \in \{\A, \LM, \Q, \bf{\mathcal{L}}\}$).
\end{theorem}

\begin{remark}
\label{remark: Seidel switching}
Theorem~\ref{theorem: Seidel switching} provides a method for finding cospectral regular graphs. These graphs
may be, however, also isomorphic. If the graphs are nonisomorphic, then it gives a pair of NICS graphs.
\end{remark}

\begin{remark}
\label{remark: limitation of Seidel switching}
A regular graph $\Gr{G}$ on $n$ vertices cannot be switched into another regular graph if $n$ is odd (see
\cite[Corollary~4.1.10]{CvetkovicRS2010}), which means that the conditions in Theorem~\ref{theorem: Seidel switching}
cannot be satisfied for any regular graph of an odd order.
\end{remark}

\begin{remark}
\label{remark: equivalence relation of Seidel switching}
Seidel switching determines an equivalence relation on graphs. This follows from the fact that switching
with respect to a subset $\set{U} \in \V{\Gr{G}}$, and then with respect to a subset $\set{V} \in \V{\Gr{G}}$,
is the same as switching with respect to $(\set{V} \setminus \set{U}) \cup (\set{U} \setminus \set{V})$ (see
\cite[p.~18]{CvetkovicRS2010}).
\end{remark}

\begin{example}
\label{example: NICS graphs by Seidel switching}
The Shrikhande graph can be obtained through Seidel switching applied to the line graph $\ell(\CoBG{4}{4})$ with respect to
four independent vertices of the latter (see \cite[Example~1.2.4]{CvetkovicRS2010}). Both are 6-regular graphs (hence, they
are cospectral graphs by Theorem~\ref{theorem: Seidel switching}). Moreover, the former graph is a strongly regular graph
$\srg{16}{6}{2}{2}$, whereas the line graph $\ell(\CoBG{4}{4})$ is not. Consequently, these are nonisomorphic and cospectral
(NICS) 6-regular graphs on 16 vertices.
\end{example}

\subsection{The Godsil and McKay method}
\label{subsection: Godsil and McKay method}

\noindent

Another construction of cospectral pairs of graphs was offered by Godsil and McKay in \cite{GodsilM1982}.

\begin{theorem}
\label{theorem: Godsil and McKay method}
Let $\Gr{G}$ be a graph with an adjacency matrix of the form
\begin{align}
\A(\Gr{G})= \begin{pmatrix}
{\mathbf{B}}  & {\mathbf{N}} \\
{\mathbf{N}}^{\mathrm{T}} & {\mathbf{C}}
\end{pmatrix}
\end{align}
where the sum of each column in ${\mathbf{N}} \in \{0,1\}^{b \times c}$ is either $0, b$ or $\frac{b}{2}$.
Let $\widehat{{\mathbf{N}}}$ be the matrix obtained by replacing each column $\underline{c}$ in ${\mathbf{N}}$
whose sum of elements is $\frac{b}{2}$ with its complement $\mathbf{1}_n - \underline{c}$. Then, the modified
graph $\widehat{\Gr{G}}$ whose adjacency matrix is given by
\begin{align}
\A(\widehat{\Gr{G}})
= \begin{pmatrix}
{\mathbf{B}}  & \widehat{\mathbf{N}}  \\
{\widehat{\mathbf{N}}}^{\mathrm{T}} & {\mathbf{C}}
\end{pmatrix}
\end{align}
is cospectral with $\Gr{G}$.
\end{theorem}
Two examples of pairs of NICS graphs are presented in Section~1.8.3 of \cite{BrouwerH2011}.

\subsection{Graphs resulting from the duplication and corona graphs}
\label{subsection: Graphs resulting from the duplication and corona graphs}

\begin{definition} \cite{SampathkumarW1980}
\label{definition: duplication graph}
Let $\Gr{G}$ be a graph with a vertex set $\V{\Gr{G}}=\{v_1, \ldots, v_n\}$, and
consider a copy of $\Gr{G}$ with a vertex set $\mathrm{\Gr{G}} = \{u_1, \ldots, u_n\}$,
where $u_i$ is a duplicate of the vertex $v_i$. For each $i \in \OneTo{n}$, connect the
vertex $u_i$ to all the neighbors of $v_i$ in $\Gr{G}$, and then delete all edges in
$\Gr{G}$. Similarly, for each $i \in \OneTo{n}$, connect the vertex $v_i$ to all the
neighbors of $u_i$ in the copied graph, and then delete all edges in the copied graph.
The resulting graph, which has $2n$ vertices is called the \emph{duplication graph}
of $\Gr{G}$, and is denoted by $\DuplicationGraph{\Gr{G}}$ (see Figure~\ref{fig:DG(C5)}).
\end{definition}
\begin{figure}[hbt]
\centering
\includegraphics[width=0.30\textwidth]{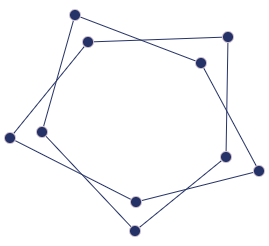}
\caption{\label{fig:DG(C5)} The duplication graph $\DuplicationGraph{\CG{5}}$ (see
Definition~\ref{definition: duplication graph}).}
\end{figure}

\begin{definition} \cite{FruchtH1970}
\label{definition: corona graph}
Let $\Gr{G}_1$ and $\Gr{G}_2$ be graphs on disjoint vertex sets of $n_1$
and $n_2$ vertices, and with $m_{1}$ and $m_{2}$ edges, respectively. The \emph{corona}
of $\Gr{G}_1$ and $\Gr{G}_2$, denoted by $\Corona{\Gr{G}_1}{\Gr{G}_2}$, is a graph
on $n_1+n_1 n_2$ vertices obtained by taking one copy of $\Gr{G}_1$
and $n_1$ copies of $\Gr{G}_2$, and then connecting, for each $i \in \OneTo{n_1}$,
the $i$-th vertex of $\Gr{G}_1$ to each vertex in the $i$-th copy of $\Gr{G}_2$
(see Figure~\ref{fig:C4 CORONA 2K1}).
\begin{figure}[hbt]
\begin{centering}
\includegraphics[width=0.30\textwidth]{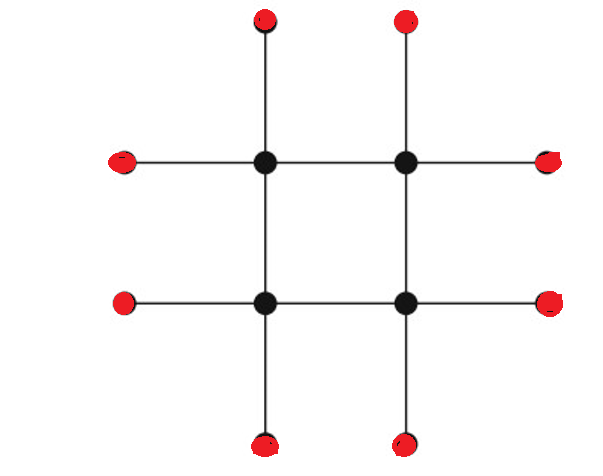}
\par\end{centering}
\caption{\label{fig:C4 CORONA 2K1} The corona graph $\Corona{\CG{4}}{(2\CoG{1})}$
(see Definition~\ref{definition: corona graph}) consists of a single copy of $\CG{4}$
(represented by the black vertices) and four copies of $2\CoG{1}$ (represented by the red vertices).}
\end{figure}
\end{definition}

\begin{definition} \cite{HouS2010}
\label{definition: edge corona graph}
The \emph{edge corona} of $\Gr{G}_1$ and $\Gr{G}_2$, denoted by
$\EdgeCorona{\Gr{G}_1}{\Gr{G}_2}$, is defined as the graph obtained by taking
one copy of $\Gr{G}_1$ and $m_1 = \bigcard{\E{\Gr{G}_1}}$ copies of $\Gr{G}_2$,
and then connecting, for each $j \in \OneTo{m_1}$, the two end-vertices of the
$j$-th edge of $\Gr{G}_1$ to every vertex in the $j$-th copy of $\Gr{G}_2$ (see
Figure~\ref{fig:C4 edge corona P3}).
\begin{figure}[hbt]
%\vspace*{-0.2cm}
\begin{centering}
\includegraphics[width=0.30\textwidth]{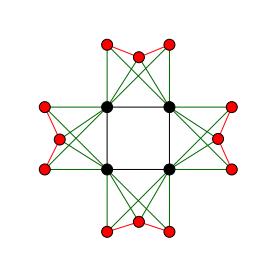}
\par\end{centering}
\caption{\label{fig:C4 edge corona P3}
The edge-corona graph $\EdgeCorona{\CG{4}}{\PathG{3}}$ (see
Definition~\ref{definition: edge corona graph}.}
\end{figure}
\end{definition}

\begin{definition}
\label{definition: duplication corona graph}
Let $\Gr{G}_1$ and $\Gr{G}_2$ be graphs with disjoint vertex sets of $n_1$
and $n_2$ vertices, respectively. Let $Du(\Gr{G}_1)$ be the duplication
graph of $\Gr{G}_1$ with vertex set $\V{\Gr{G}_1} \cup \mathrm{U}(\Gr{G}_1)$, where
$\V{\Gr{G}_1}=\{v_{1},\ldots,v_{n_1}\}$ and with the duplicated vertex set
$\mathrm{U}(\Gr{G}_1)=\{u_{1},\ldots,u_{n_{1}}\}$ (see Definition~\ref{definition: duplication graph}).
The \emph{duplication corona graph}, denoted by $\DuplicationCorona{\Gr{G}_1}{\Gr{G}_2}$,
is the graph obtained from $\DuplicationGraph{\Gr{G}_1}$ and $n_1$ copies of $\Gr{G}_2$
by connecting, for each $i \in \OneTo{n_1}$, the vertex $v_{i} \in \V{\Gr{G}_1}$ of
the graph $\DuplicationGraph{\Gr{G}_1}$ to every vertex in the $i$-th copy of $\Gr{G}_2$
(see Figure~\ref{fig:D.C}).
\begin{figure}[hbt]
\begin{centering}
\includegraphics[width=0.30\textwidth]{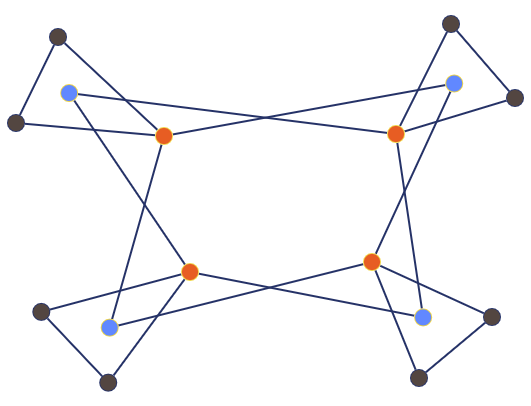}
\par\end{centering}
\caption{\label{fig:D.C} The duplication corona graph $\DuplicationCorona{\CG{4}}{\CoG{2}}$
(see Definition~\ref{definition: duplication corona graph}).}
\end{figure}
\end{definition}

\begin{definition}
\label{definition: duplication neighborhood corona}
Let $\Gr{G}_1$ and $\Gr{G}_2$ be graphs with disjoint vertex sets of $n_1$
and $n_2$ vertices, respectively. Let $\DuplicationGraph{\Gr{G}_1}$ be the duplication
graph of $\Gr{G}_1$ with the vertex set $\mathrm{V}(\Gr{G}_1)\cup \mathrm{U}(\Gr{G}_1)$, where
$\mathrm{V}(\Gr{G}_1)=\{v_1, \ldots, v_{n_1}\}$ and the duplicated vertex set
$\mathrm{U}(\Gr{G}_1)=\{u_1,\ldots,u_{n_1}\}$ (see Definition~\ref{definition: duplication graph}).
The \emph{duplication neighborhood corona}, denoted by $\Gr{G}_1 \boxbar \Gr{G}_2$,
is the graph obtained from $\DuplicationGraph{\Gr{G}_1}$ and $n_1$ copies of $\Gr{G}_2$
by connecting the neighbors of the vertex $v_i \in \V{\Gr{G}_1}$ of $\DuplicationGraph{\Gr{G}_1}$
to every vertex in the $i$-th copy of $\Gr{G}_2$ for $i \in \OneTo{n_1}$ (see Figure~\ref{fig:D.N.C}).
\begin{figure}[hbt]
\begin{centering}
\includegraphics[width=0.30\textwidth]{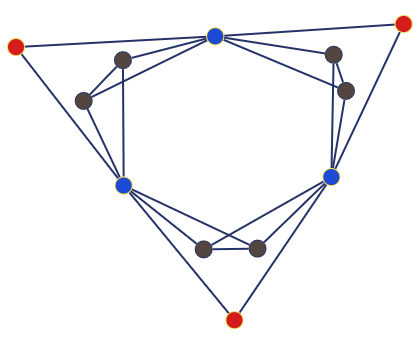}
\par\end{centering}
\caption{\label{fig:D.N.C} The duplication neighborhood corona $\CoG{3} \boxbar \CoG{2}$
(see Definition~\ref{definition: duplication neighborhood corona}).}
\end{figure}
\end{definition}

\begin{definition}
\label{definition: duplication edge corona}
Let $\Gr{G}_1$ and $\Gr{G}_2$ be graphs with disjoint vertex sets of $n_1$
and $n_2$ vertices, respectively. Let $\DuplicationGraph{\Gr{G}_1}$ be the
duplication graph of $\Gr{G}_1$ with vertex set $\V{\Gr{G_1}} \cup
\mathrm{U}(\Gr{G}_1)$, where $\V{\Gr{G}_1}=\{v_1, \ldots, v_{n_1}\}$ is the
vertex set of $\Gr{G}_1$ and $\mathrm{U}(\Gr{G}_1)=\{u_1, \ldots, u_{n_{1}}\}$
is the duplicated vertex set.
The \emph{duplication edge corona}, denoted by $\DuplicationEdgeCorona{\Gr{G}_1}{\Gr{G}_2}$,
is the graph obtained from $\DuplicationGraph{\Gr{G}_1}$ and $\bigcard{\E{\Gr{G}_1}}$
copies of $\Gr{G}_2$ by connecting each of the two vertices $v_i, v_j \in \V{\Gr{G}_1}$
of $\DuplicationGraph{\Gr{G}_1}$ to every vertex in the $k$-th copy of $\Gr{G}_2$
whenever $\{v_i, v_j\}=e_k \in \E{\Gr{G}_1}$ (see Figure~\ref{fig:D.E.C}).
\begin{figure}[hbt]
\begin{centering}
\includegraphics[width=0.30\textwidth]{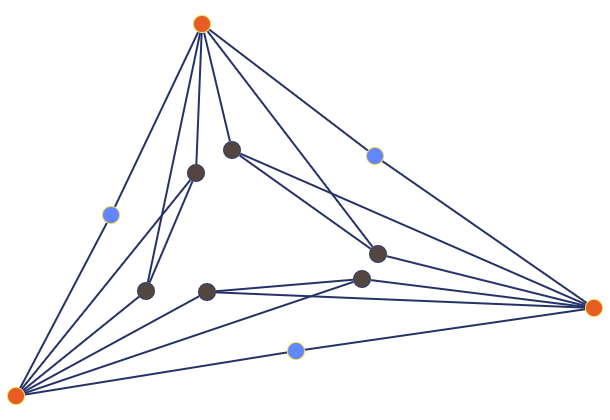}
\par\end{centering}
\centering{} \caption{\label{fig:D.E.C}
The duplication edge corona $\DuplicationEdgeCorona{\CoG{3}}{\CoG{2}}$ (see
Definition~\ref{definition: duplication edge corona}).}
\end{figure}
\end{definition}

\begin{definition}
\label{definition: closed neighborhood corona}
Consider two graphs $\Gr{G}_1$ and $\Gr{G}_2$ with $n_{1}$ and $n_{2}$
vertices and, respectively. The \emph{closed neighborhood corona} of
$\Gr{G}_1$ and $\Gr{G}_2$, denoted by $\ClosedNeighborhoodCorona{\Gr{G}_1}{\Gr{G}_2}$,
is a new graph obtained by creating $n_{1}$ copies of $\Gr{G}_2$. Each
vertex of the $i^{th}$ copy of $\Gr{G}_2$ is then connected to the
$i^{th}$ vertex and neighborhood of the $i^{th}$ vertex of $\Gr{G}_1$
(see Figure~\ref{fig:CNCP}).
\end{definition}

\begin{figure}[hbt]
\begin{centering}
\includegraphics[width=0.30\textwidth]{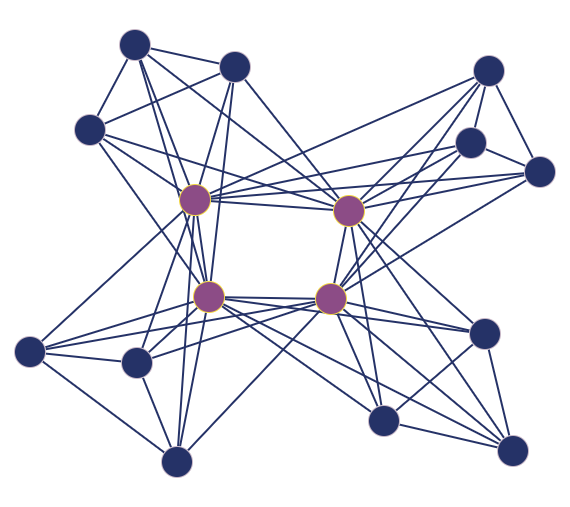}
\par\end{centering}
\caption{\label{fig:CNCP} \centering{The closed neighborhood corona
of the 4-length cycle $\CG{4}$ and the triangle $\CoG{3}$, denoted by
$\ClosedNeighborhoodCorona{\CG{4}}{\CoG{3}}$ (see
Definition~\ref{definition: closed neighborhood corona}).}}
\end{figure}

\begin{theorem}  \cite{AdigaRS2018}
\label{theorem: NICS graphs - 1}
Let $\Gr{G}_1,\Gr{H}_1$ be $r_1$-regular, cospectral graphs, and let
$\Gr{G}_2$ and $H_{2}$ be $r_2$-regular, cospectral, and nonisomorphic (NICS)
graphs. Then, the following holds:
\begin{itemize}
\item The duplication corona graphs $\DuplicationCorona{\Gr{G}_1}{\Gr{G}_2}$ and
$\DuplicationCorona{\Gr{H}_1}{\Gr{H}_2}$ are $\{ \A, \LM, \Q \}$-NICS irregular graphs.
\item The duplication neighborhood corona $\Gr{G}_1\boxbar \Gr{G}_2$ and
$\Gr{H}_1\boxbar H_{2}$ are $\{ \A, \LM, \Q \}$-NICS irregular graphs.
\item The duplication edge corona $\DuplicationEdgeCorona{\Gr{G}_1}{\Gr{G}_2}$ and
$\DuplicationEdgeCorona{\Gr{H}_1}{\Gr{H}_2}$ are $\{ \A, \LM, \Q \}$-NICS irregular graphs.
\end{itemize}
\end{theorem}

\begin{question}
\label{question: NICS graphs - 1}
Are the irregular graphs in Theorem~\ref{theorem: NICS graphs - 1} also cospectral with
respect to the normalized Laplacian matrix?
\end{question}

\begin{theorem}
\label{theorem: NICS graphs - 2}
\cite{SonarS2024} Let $\Gr{G}_1$ and $\Gr{G}_2$ be cospectral regular graphs, and let
$\Gr{H}$ be an arbitrary graph. Then, the following holds:
\begin{itemize}
\item The closed neighborhood corona $\ClosedNeighborhoodCorona{\Gr{G}_1}{\Gr{H}}$ and
$\ClosedNeighborhoodCorona{\Gr{G}_2}{\Gr{H}}$ are $\{ \A, \LM, \Q \}$-NICS irregular graphs.
\item The closed neighborhood corona $\ClosedNeighborhoodCorona{\Gr{H}}{\Gr{G}_{1}}$
and $\ClosedNeighborhoodCorona{\Gr{H}}{\Gr{G}_{2}}$ are $\{ \A, \LM, \Q \}$-NICS irregular graphs.
\end{itemize}
\end{theorem}

\begin{question}
\label{question: NICS graphs - 2}
Are the irregular graphs in Theorem~\ref{theorem: NICS graphs - 2} also cospectral with respect
to the normalized Laplacian matrix?
\end{question}

\subsection{Graphs constructions based on the subdivision and bipartite incidence graphs}
\label{subsection: Graphs constructions based on the subdivision and bipartite incidence graphs}
\begin{definition} \cite{CvetkovicRS2010}
\label{definition: subdivision graph}
Let $\Gr{G}$ be a graph. The \emph{subdivision graph} of $\Gr{G}$, denoted by $\SubdivisionGraph{\Gr{G}}$,
is obtained from $\Gr{G}$ by inserting a new vertex into every edge of $\Gr{G}$.
Subdivision is the process of adding a new vertex along an edge, effectively splitting the edge into two
edges connected in series through the new vertex (see Figure~\ref{fig:subdivition}).
\begin{figure}[hbt]
\centering{}\includegraphics[width=0.30\textwidth]{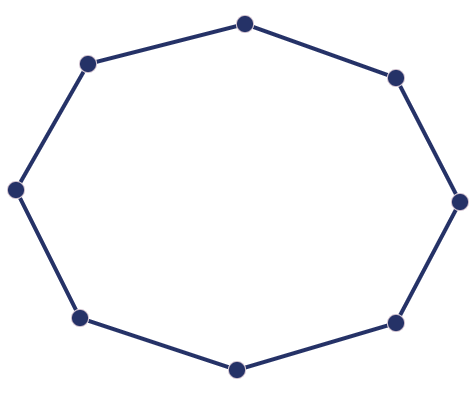}
\caption{\label{fig:subdivition} \centering{The subdivision graph of a $4$-length cycle, denoted by
$\SubdivisionGraph{\CG{4}}$, which is an $8$-length cycle $\CG{8}$ (see Definition~\ref{definition: subdivision graph}).}}
\end{figure}
\end{definition}

\begin{definition}  \cite{PisankiS13}
\label{definition: bipartite incidence graph}
Let $\Gr{G}$ be a graph. The \emph{bipartite incidence graph} of $\Gr{G}$, denoted by $\BipartiteIncidenceGraph{\Gr{G}}$,
is a bipartite graph constructed as follows: For each edge $e \in \Gr{G}$, a new vertex $u_{e}$ is added to
the vertex set of $\Gr{G}$. The vertex $u_e$ is then made adjacent to both endpoints of $e$ in $\Gr{G}$
(see Figure~\ref{fig:R(C4)-1}).
\begin{figure}[H]
\begin{centering}
\includegraphics[width=0.30\textwidth]{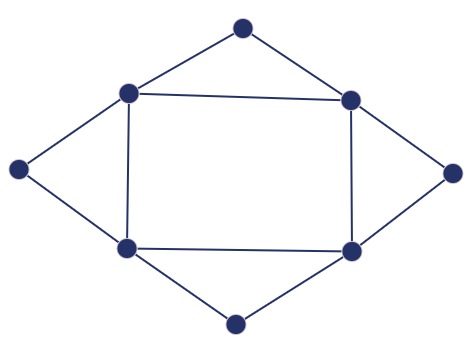}
\par\end{centering}
\centering{}\caption{\label{fig:R(C4)-1} \centering{The bipartite incidence graph of a length-4 cycle $\CG{4}$,
denoted by $\BipartiteIncidenceGraph{\CG{4}}$ (see Definition~\ref{definition: bipartite incidence graph}).}}
\end{figure}
\end{definition}

Let $\Gr{G}_1$ and $\Gr{G}_2$ be two vertex-disjoint graphs with $n_1$ and $n_2$ vertices, and $m_1$ and $m_2$ edges,
respectively. Four binary operations on these graphs are defined as follows.
\begin{definition}
\label{definition: subdivision-vertex-bipartite-vertex join}
The \emph{subdivision-vertex-bipartite-vertex join} of graphs $\Gr{G}_1$ and $\Gr{G}_2$, denoted $\SVBVJ{\Gr{G}_1}{\Gr{G}_2}$,
is the graph obtained from $\SubdivisionGraph{\Gr{G}_1}$ and $\BipartiteIncidenceGraph{\Gr{G}_2}$
by connecting each vertex of $\V{\Gr{G}_1}$ (that is the subset of the original vertices in the vertex set of $\SubdivisionGraph{\Gr{G}_1}$)
with every vertex of $\V{\Gr{G}_2}$ (that is the subset of the original vertices in the vertex set of $\BipartiteIncidenceGraph{\Gr{G}_2}$).
\end{definition}
By Definitions~\ref{definition: subdivision graph}, \ref{definition: bipartite incidence graph}, and~\ref{definition: subdivision-vertex-bipartite-vertex join},
it follows that the graph $\SVBVJ{\Gr{G}_1}{\Gr{G}_2}$ has $n_1+n_2+m_1+m_2$ vertices and $n_1 n_2+2m_1+3m_2$ edges.
Figure~\ref{fig:subdivision-vertex-bipartite-vertex join} displays the graph $\SVBVJ{\CG{4}}{\PathG{3}}$.
\begin{figure}[hbt]
\begin{centering}
\includegraphics[width=0.40\textwidth]{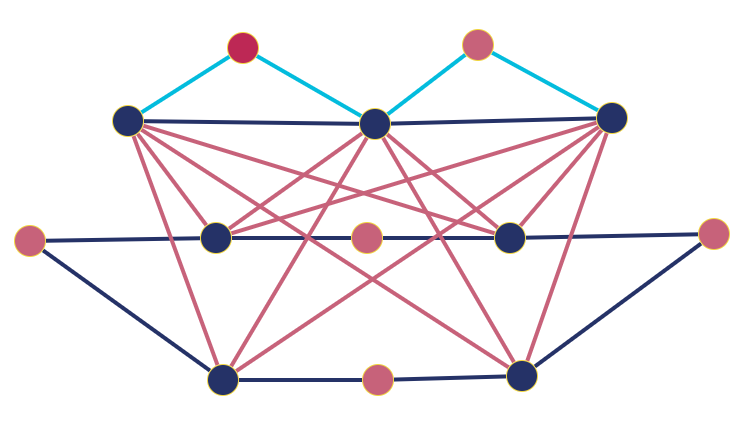}
\par\end{centering}
\caption{\label{fig:subdivision-vertex-bipartite-vertex join} The graph $\SVBVJ{\CG{4}}{\PathG{3}}$ (see Definition~\ref{definition: subdivision-vertex-bipartite-vertex join}).
The black vertices represent the vertices of the length-4 cycle $\CG{4}$ and the vertices of the path $\PathG{3}$. The additional vertices in the
subdivision graph $\SubdivisionGraph{\CG{4}}$ are the four red vertices located at the bottom of this figure (as shown in Figure~\ref{fig:subdivition}).
Similarly, the additional vertices in the bipartite incidence graph of the path $\PathG{3}$ are the two red vertices located at the top of this figure.
The edges of the graph include the following: edges connecting each black vertex of $\CG{4}$ to each black vertex of $\PathG{3}$, edges between the black
and red vertices at the bottom of the figure (that correspond to the subdivision of $\CG{4}$), and the four (light blue) edges connecting the two top red
vertices to the top black vertices of the figure (that correspond to the bipartite incidence graph of $\PathG{3}$).}
\end{figure}

\begin{definition}
\label{definition: subdivision-edge-bipartite-edge join}
The \emph{subdivision-edge-bipartite-edge join} of $\Gr{G}_1$ and $\Gr{G}_2$, denoted by $\SEBEJ{\Gr{G}_1}{\Gr{G}_2}$,
is the graph obtained from $\SubdivisionGraph{\Gr{G}_1}$ and $\BipartiteIncidenceGraph{\Gr{G}_2}$
by connecting each vertex of $\V{\SubdivisionGraph{\Gr{G}_1}} \setminus \V{\Gr{G}_1}$ (that is the subset of the added vertices in the
vertex set of $\SubdivisionGraph{\Gr{G}_1}$) with every vertex of $\V{\BipartiteIncidenceGraph{\Gr{G}_2}} \setminus \V{\Gr{G}_2}$
(that is the subset of the added vertices in the vertex set of $\BipartiteIncidenceGraph{\Gr{G}_2}$).
\end{definition}
By Definitions~\ref{definition: subdivision graph}, \ref{definition: bipartite incidence graph}, and~\ref{definition: subdivision-edge-bipartite-edge join},
it follows that the graph $\SEBEJ{\Gr{G}_1}{\Gr{G}_2}$ has $n_1+n_2+m_1+m_2$ vertices and $m_1 m_2 + 2m_1 + 3m_2$ edges.
Figure~\ref{fig:subdivision-edge-R-edge join} displays the graph $\SEBEJ{\CG{4}}{\PathG{3}}$.
\begin{figure}[hbt]
\begin{centering}
\includegraphics[width=0.40\textwidth]{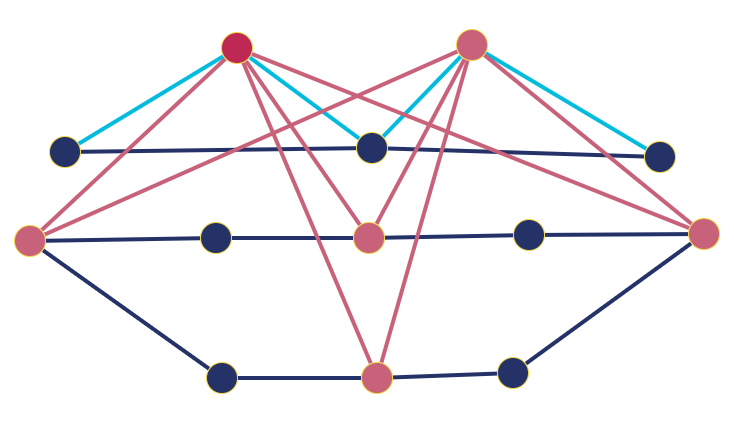}
\par\end{centering}
\caption{\label{fig:subdivision-edge-R-edge join} The graph $\SEBEJ{\CG{4}}{\PathG{3}}$ (see Definition~\ref{definition: subdivision-edge-bipartite-edge join}).
In comparison to Figure~\ref{fig:subdivision-vertex-bipartite-vertex join}, edges connecting each black vertex of $\CG{4}$ to each black vertex of $\PathG{3}$
are deleted and replaced in this figure by edges connecting each red vertex of $\CG{4}$ to each red vertex of $\PathG{3}$.}
\end{figure}

\begin{definition}
\label{definition: subdivision-edge-bipartite-vertex join}
The \emph{subdivision-edge-bipartite-vertex join} of $\Gr{G}_1$ and $\Gr{G}_2$, denoted by $\SEBVJ{\Gr{G}_1}{\Gr{G}_2}$,
is the graph obtained from $\SubdivisionGraph{\Gr{G}_1}$ and $\BipartiteIncidenceGraph{\Gr{G}_2}$ by connecting each vertex
of $\V{\SubdivisionGraph{\Gr{G}_1}} \setminus \V{\Gr{G}_1}$ (that is the subset of the added vertices in the vertex set of
$\SubdivisionGraph{\Gr{G}_1}$) with every vertex of $\V{\Gr{G}_2}$ (that is the subset of the original vertices in the vertex
set of $\BipartiteIncidenceGraph{\Gr{G}_2}$).
\end{definition}
By Definitions~\ref{definition: subdivision graph}, \ref{definition: bipartite incidence graph}, and~\ref{definition: subdivision-edge-bipartite-vertex join},
it follows that the graph $\SEBVJ{\Gr{G}_1}{\Gr{G}_2}$ has $n_1+n_2+m_1+m_2$ vertices and $m_1 n_2 + 2m_1 + 3m_2$ edges.
Figure~\ref{fig:subdivision-edge-R-vertex join} displays the graph $\SEBVJ{\CG{4}}{\PathG{3}}$.
\begin{figure}[hbt]
\begin{centering}
\includegraphics[width=0.40\textwidth]{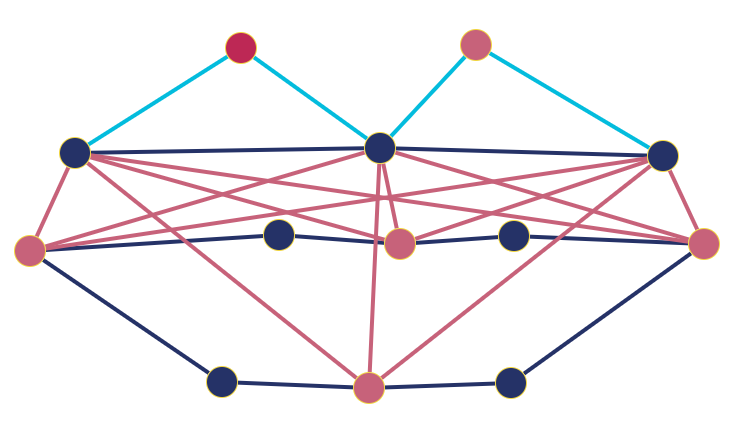}
\par\end{centering}
\caption{\label{fig:subdivision-edge-R-vertex join} The graph $\SEBVJ{\CG{4}}{\PathG{3}}$ (see Definition~\ref{definition: subdivision-edge-bipartite-vertex join}).
In comparison to Figure~\ref{fig:subdivision-vertex-bipartite-vertex join}, edges connecting each black vertex of $\CG{4}$ to each black vertex of $\PathG{3}$
are deleted and replaced in this figure by edges connecting each red vertex of $\CG{4}$ to each black vertex of $\PathG{3}$.}
\end{figure}

\begin{definition}
\label{definition: subdivision-vertex-bipartite-edge join}
The \emph{subdivision-vertex-bipartite-edge join} of $\Gr{G}_1$ and $\Gr{G}_2$, denoted by $\SVBEJ{\Gr{G}_1}{\Gr{G}_2}$,
is the graph obtained from $\SubdivisionGraph{\Gr{G}_1}$ and $\BipartiteIncidenceGraph{\Gr{G}_2}$ by connecting each vertex
of $\V{\Gr{G}_1}$ (that is the subset of the original vertices in the vertex set of $\SubdivisionGraph{\Gr{G}_1}$) with
every vertex of $\V{\BipartiteIncidenceGraph{\Gr{G}_2}} \setminus \V{\Gr{G}_2}$ (that is the subset of the added vertices in the vertex
set of $\BipartiteIncidenceGraph{\Gr{G}_2}$).
\end{definition}
By Definitions~\ref{definition: subdivision graph}, \ref{definition: bipartite incidence graph}, and~\ref{definition: subdivision-vertex-bipartite-edge join},
it follows that the graph $\SVBEJ{\Gr{G}_1}{\Gr{G}_2}$ has $n_1+n_2+m_1+m_2$ vertices and $n_1 m_2 + 2m_1 + 3m_2$ edges.
Figure~\ref{fig:subdivision-vertex-R-edge join} displays the graph $\SVBEJ{\CG{4}}{\PathG{3}}$.
\begin{figure}[hbt]
\begin{centering}
\includegraphics[width=0.40\textwidth]{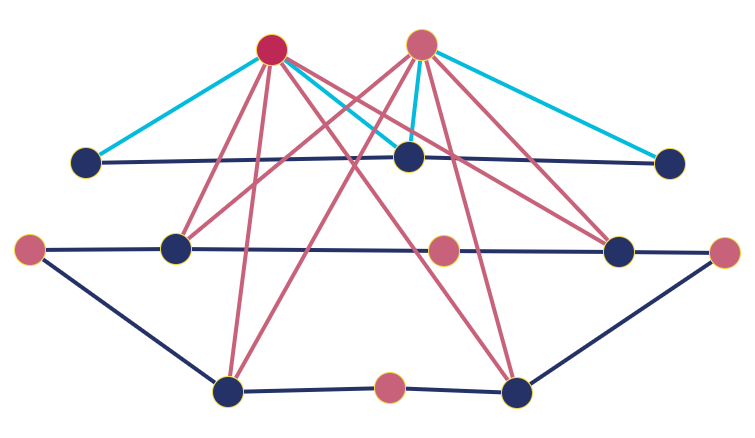}
\par\end{centering}
\caption{\label{fig:subdivision-vertex-R-edge join} The graph $\SVBEJ{\CG{4}}{\PathG{3}}$ (see Definition~\ref{definition: subdivision-vertex-bipartite-edge join}).
In comparison to Figure~\ref{fig:subdivision-vertex-bipartite-vertex join}, edges connecting each black vertex of $\CG{4}$ to each black vertex of $\PathG{3}$
are deleted and replaced in this figure by edges connecting each black vertex of $\CG{4}$ to each red vertex of $\PathG{3}$.}
\end{figure}

We next present the main result of this subsection, which motivates the four binary graph operations introduced in
Definitions~\ref{definition: subdivision-vertex-bipartite-vertex join}--\ref{definition: subdivision-vertex-bipartite-edge join}.
\begin{theorem} \cite{DasP2013}
\label{theorem: NICS - DasP2013} Let $\Gr{G}_i$ , $\Gr{H}_i$, where $i \in \{1,2\}$,
be regular graphs, where $\Gr{G}_1$ need not be different from $\Gr{H}_1$.
If $\Gr{G}_1$ and $\Gr{H}_1$ are $\A$-cospectral, and $\Gr{G}_2$ and $\Gr{H}_2$
are $\A$-cospectral, then
\begin{itemize}
\item $\SVBVJ{\Gr{G}_1}{\Gr{G}_2}$ and $\SVBVJ{\Gr{H}_1}{\Gr{H}_2}$
are irregular $\{ \A, \LM, {\bf{\mathcal{L}}} \} $-NICS graphs.
\item $\SEBEJ{\Gr{G}_1}{\Gr{G}_2}$ and $\SEBEJ{\Gr{H}_1}{\Gr{H}_2}$
are irregular $\{ \A, \LM, {\bf{\mathcal{L}}} \} $-NICS graphs.
\item $\SEBVJ{\Gr{G}_1}{\Gr{G}_2}$ and $\SEBVJ{\Gr{H}_1}{\Gr{H}_2}$
are irregular $\{ \A, \LM, {\bf{\mathcal{L}}} \} $-NICS graphs.
\item $\SVBEJ{\Gr{G}_1}{\Gr{G}_2}$ and $\SVBEJ{\Gr{H}_1}{\Gr{H}_2}$
are irregular $\{ \A, \LM, {\bf{\mathcal{L}}} \} $-NICS graphs.
\end{itemize}
\end{theorem}

In light of Theorem~\ref{theorem: NICS - DasP2013}, the following questions naturally arises.
\begin{question}
\label{question: NICS - DasP2013}
Are the graphs in Theorem~\ref{theorem: NICS - DasP2013} also cospectral with respect to the signless Laplacian matrix (i.e., $\Q$-cospectral)?
\end{question}

\subsection{Connected irregular NICS graphs}
\label{subsection: Connected irregular NICS graphs}

\noindent

The (joint) cospectrality of regular graphs with respect to their adjacency, Laplacian, signless Laplacian, and
normalized Laplacian matrices can be asserted by verifying their cospectrality with respect to only one of these matrices.
In other words, regular graphs are $X$-cospectral for some $X \in \{\A, \LM, \Q, {\bf{\mathcal{L}}} \}$
if and only if they are cospectral with respect to all these matrices (see Proposition~\ref{proposition: regular graphs cospectrality}).
For irregular graphs, this does not hold in general.
Following \cite{Butler2010}, it is natural to ask the following question:
\begin{question}
\label{question: Butler2010}
Are there pairs of irregular graphs that are $\{\A,\LM,\Q,{\bf{\mathcal{L}}}\}$-NICS,
i.e., $X$-NICS with respect to every $X \in \{\A, \LM, \Q, {\bf{\mathcal{L}}}\}$?
\end{question}
This question remained open until two coauthors of this paper recently proposed a method for constructing pairs
of irregular graphs that are $X$-cospectral with respect to every $X \in \{\A, \LM, \Q, {\bf{\mathcal{L}}}\}$,
providing explicit constructions \cite{HamudB24}. Building on that work, another coauthor of this paper demonstrated
in \cite{Sason2024} that for every even integer $n \geq 14$, there exist two connected, irregular
$\{\A,\LM,\Q,{\bf{\mathcal{L}}}\}$-NICS graphs on $n$ vertices with identical independence, clique, and chromatic numbers,
yet distinct Lov{\'a}sz $\vartheta$-functions. We now present the preliminary definitions required to outline
the relevant results in \cite{HamudB24} and \cite{Sason2024}, and the construction of such cospectral irregular
$X$-NICS graphs for all $X \in \{\A, \LM, \Q, {\bf{\mathcal{L}}}\}$.

\begin{definition}[Neighbors splitting join of graphs]
\label{definition: neighbors splitting join of graphs}
\cite{LuMZ2023} Let $\Gr{G}$ and $\Gr{H}$ be graphs with disjoint vertex sets, and let
$\V{\Gr{G}} = \{v_1, \ldots, v_n\}$. The {\em neighbors splitting (NS) join}
of $\Gr{G}$ and $\Gr{H}$ is obtained by adding vertices $v'_1, \ldots, v'_n$
to the vertex set of $\Gr{G} \vee \Gr{H}$ and connecting $v'_i$ to $v_j$ if
and only if $\{v_i, v_j\} \in \E{\Gr{G}}$.
The NS join of $\Gr{G}$ and $\Gr{H}$ is denoted by $\Gr{G} \NS \Gr{H}$.
\end{definition}

\begin{definition}[Nonneighbors splitting join of graphs \cite{Hamud23,HamudB24}]
\label{definition: nonneighbors splitting join of graphs}
Let $\Gr{G}$ and $\Gr{H}$ be graphs with disjoint vertex sets, and let
$\V{\Gr{G}} = \{v_1, \ldots, v_n\}$. The {\em nonneighbors splitting (NNS) join}
of $\Gr{G}$ and $\Gr{H}$ is obtained by adding vertices $v'_1, \ldots, v'_n$
to the vertex set of $\Gr{G} \vee \Gr{H}$ and connecting $v'_i$ to $v_j$, with
$i \neq j$, if and only if $\{v_i, v_j\} \not\in \E{\Gr{G}}$.
The NNS join of $\Gr{G}$ and $\Gr{H}$ is denoted by $\Gr{G} \NNS \Gr{H}$.
\end{definition}

\begin{remark}
In general, $\Gr{G} \NS \Gr{H} \not\cong \Gr{H} \NS \Gr{G}$ and
$\Gr{G} \NNS \Gr{H} \not\cong \Gr{H} \NNS \Gr{G}$ (unless $\Gr{G} \cong \Gr{H}$).
\end{remark}

\begin{example}[NS and NNS join of graphs \cite{HamudB24}]
\label{example: NS and NNS Join of Graphs}
Figure~\ref{fig:BermanH23-Fig. 2.2} shows the NS and NNS
joins of the path graphs $\PathG{4}$ and $\PathG{2}$, denoted by $\PathG{4} \NS \PathG{2}$
and $\PathG{4} \NNS \PathG{2}$, respectively.
\vspace*{-0.1cm}
\begin{figure}[H]
\centering
\includegraphics[width=11cm]{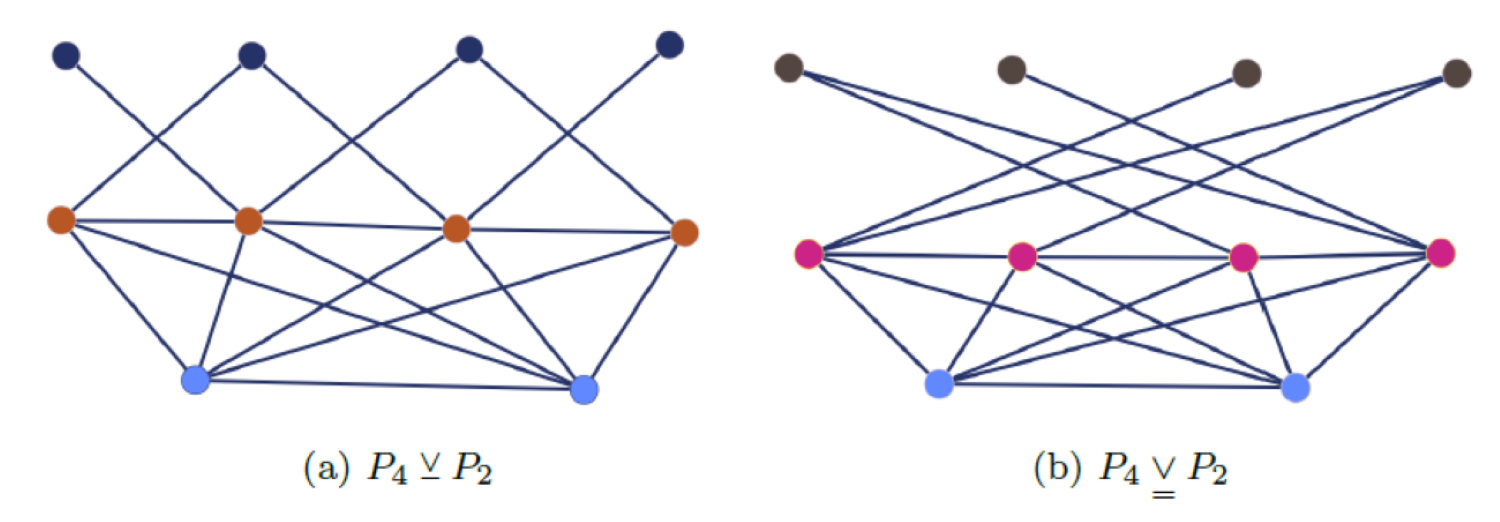}
\caption{The neighbors splitting (NS)
and nonneighbors splitting (NNS) joins of the path graphs $\PathG{4}$ and $\PathG{2}$
are depicted, respectively, on the left and right-hand sides of this figure.
The NS and NNS joins of graphs are, respectively, denoted by $\PathG{4} \NS \PathG{2}$
and $\PathG{4} \NNS \PathG{2}$ \cite{HamudB24}.}\label{fig:BermanH23-Fig. 2.2}
\end{figure}
\end{example}

\begin{theorem}[Irregular $\{\A, \LM, \Q, \bf{\mathcal{L}}\}$-NICS graphs]
\label{theorem: Berman and Hamud, 2023}
Let $\Gr{G}_1$ and $\Gr{H}_1$ be regular and cospectral graphs, and let $\Gr{G}_2$ and $\Gr{H}_2$
be regular, nonisomorphic, and cospectral (NICS) graphs. Then, the following statements hold:
\begin{enumerate}[(1)]
\item  \label{Item 1: Berman and Hamud, 2023}
The NS-join graphs $\Gr{G}_1 \NS \Gr{G}_2$ and $\Gr{H}_1 \NS \Gr{H}_2$
are irregular $\{\A, \LM, \Q, \bf{\mathcal{L}}\}$-NICS graphs.
\item  \label{Item 2: Berman and Hamud, 2023}
The NNS join graphs $\Gr{G}_1 \NNS \Gr{G}_2$ and $\Gr{H}_1 \NNS \Gr{H}_2$
are irregular $\{\A, \LM, \Q, \bf{\mathcal{L}}\}$-NICS graphs.
\end{enumerate}
\end{theorem}
The proof of Theorem~\ref{theorem: Berman and Hamud, 2023} is provided in \cite{Hamud23,HamudB24}, and it relies
heavily on the notion of the Schur complement (see Theorem~\ref{theorem: Schur complement}).
The interested reader is referred to the recently published paper \cite{HamudB24} for further details.

Relying on Theorem~\ref{theorem: Berman and Hamud, 2023}, the following result is stated and proved in \cite{Sason2024}.
\begin{theorem}[On irregular NICS graphs]
\label{theorem: Sason - existence of NICS graphs with distinct theta functions}
For every even integer $n \geq 14$, there exist two connected, irregular $\{\A, \LM, \Q, \bf{\mathcal{L}}\}$-NICS graphs
on $n$ vertices with identical independence, clique, and chromatic numbers, yet their Lov\'{a}sz $\vartheta$-functions
are distinct.
\end{theorem}
That result is in fact strengthened in Section~4.2 of \cite{Sason2024}, and the interested reader is referred to that paper
for further details. The proof of Theorem~\ref{theorem: Sason - existence of NICS graphs with distinct theta functions} is
also constructive, providing explicit such graphs.

\section{Open questions and outlook}
\label{section: summary and outlook}

We conclude this paper by highlighting some of the most significant open questions in the fascinating area of research related to
cospectral nonisomorphic graphs and graphs determined by their spectrum, leaving them as topics for further future study.

\subsection{Haemers' conjecture}
\label{subsection: Haemers' conjecture}

Haemers' conjecture \cite{Haemers2016,Haemers2024}, a prominent topic in spectral graph theory, posits that almost all graphs
are uniquely determined by their adjacency spectrum. This conjecture suggests that for large graphs, the probability of having two
non-isomorphic graphs, on a large number of vertices, sharing the same adjacency spectrum is negligible. The conjecture has inspired
extensive research, including studies on specific graph families, cospectrality, and algebraic graph invariants, contributing to deeper
insights into the relationship between graph structure and eigenvalues. Haemers' conjecture is stated formally as follows.

\begin{definition}
\label{definition: Haemers' conjecture}
For $n \in \mathbb{N}$, let $I(n)$ to be the numbers of distinct graphs on $n$ vertices, up to isomorphism.
For $\mathcal{X} \subseteq \Gmats$, let $\alpha_{\mathcal{X}}(n)$ be the number of $\mathcal{X}$-DS graphs
on $n$ vertices, up to isomorphism, and (for the sake of simplicity of notation) let $\alpha(n) \triangleq \alpha_{\{ A \}}(n)$
\end{definition}

\begin{conjecture}
\label{conjecture: Haemers' conjecture} \cite{Haemers2016}
For sufficiently large $n$, almost all of the graphs are DS, i.e.,
\begin{align}
\label{eq: Heamers' conjecture}
\lim_{n \rightarrow \infty} \frac{\alpha(n)}{I(n)} = 1.
\end{align}
\end{conjecture}

Several results lend support to this conjecture \cite{Haemers2016, KovalK2024, WangW2024}, but a complete proof in its full generality remains elusive.
By \cite{GodsilR2001}, the number of graphs with $n$ vertices up to isomorphism is
\begin{align}
I(n) = \bigl( 1 - \epsilon(n) \bigr) \, \frac{2^{n(n-1)/2}}{n!},
\end{align}
where $\underset{n \rightarrow \infty}{\lim} \epsilon(n) = 0$.
By Stirling's approximation for $n!$ and straightforward algebra, $I(n)$ can be verified to be given by
\begin{align}
I(n) = 2^{\frac{n(n-1)}{2} \; \bigl(1 - \epsilon(n)\bigr)}.
\end{align}

It is shown in \cite{KovalK2024} that the number of DS graphs on $n$ vertices is at least $e^{cn}$ for some constant $c>0$ (see also the
discussion in Section~\ref{subsection: Nice graphs}).
In light of Remark~\ref{remark: X,Y cospectrality}, Conjecture~\ref{conjecture: Haemers' conjecture} leads to the following question.
\begin{question}
\label{question: generalized Haemers' conjecture}
For what minimal subset $\mathcal{X} \subset \Gmats$ (if any) does the limit
\begin{align}
\label{eq: generalized Haemers' conjecture}
\lim_{n\rightarrow \infty} \frac{\alpha_{\mathcal{X}}(n)}{I(n)} = 1
\end{align}
hold?
\end{question}

Some computer results somewhat support Haemers' Conjecture. Approximately $80\%$ of the graphs with $12$ vertices are DS \cite{BrouwerS2009}.
A new class of graphs is defined in \cite{WangW2024}, offering an algorithmic method for finding all the $\{\A,\overline{\A}\}$-cospectral mates
of a graph $\Gr{G}$ in this class  (i.e., graphs that are $\{\A,\overline{\A}\}$-cospectral with $\Gr{G}$).
Using this algorithm, they found that out of $10,000$ graphs with $50$ vertices, chosen uniformly at random from this class, at least $99.5\%$
of them were $\{\A, \overline{\A}\}$-DS.

\subsection{DS properties of structured graphs}
\label{subsection: DS properties of structured graphs}

This paper explores various structures of graph families determined by their spectrum with respect to one or more of their associated matrices,
as well as the related problem of constructing pairs of nonisomorphic graphs that are cospectral with respect to some or all of these matrices.
Several questions are interspersed throughout the paper (see Questions~\ref{question: DS SRG}, \ref{question: NICS graphs - 1}, \ref{question: NICS graphs - 2},
\ref{question: NICS - DasP2013}, and~\ref{question: Butler2010}), which remain open for further study.

In addition to serving as a survey on spectral graph determination, this paper suggests a new alternative proof to Theorem~3.3 in \cite{MaRen2010},
asserting that all Tur\'{a}n graphs are determined by their adjacency spectrum (see Theorem~\ref{theorem: Turan's graph is DS}).
This proof is based on a derivation of the adjacency spectrum of the family of Tur\'{a}n graphs (see Theorem~\ref{theorem: A-spectrum of Turan graph}).
Since these graphs are generally bi-regular multipartite graphs (i.e., their vertices have only two possible degrees, which in this case are consecutive
integers), it does not necessarily imply that Tur\'{a}n graphs are determined by the spectrum of some other associated matrices, such as the Laplacian,
signless Laplacian, or normalized Laplacian matrices. Determining whether this holds is an open question.

The distance matrix of a connected graph is the symmetric matrix whose columns and rows are indexed by the graph vertices, and its entries are
equal to the pairwise distances between the corresponding vertices. The distance spectrum is the multiset of eigenvalues of the distance matrix,
and its characterization has been a subject of fruitful research (see \cite{AouchicheH14,HuiqiuJJY21} for comprehensive surveys on the distance
spectra of graphs, and \cite{Banerjee23,HogbenR22} on the spectra of graphs with respect to variants of the distance matrix). Nonisomorphic graphs
may share an identical adjacency spectrum, as well as an identical distance spectrum, and therefore be graphs that are not determined by either
their adjacency or distance spectrum.
This holds, e.g., for the Shrikhande graph and the line graph of the complete bipartite graph $\CoBG{4}{4}$, which are nonisomorphic strongly
regular graphs with the identical parameters $(16,6,2,2)$, sharing an identical $\A$-spectrum given by $\{6, [2]^6, [-2]^6\}$ and an identical
distance spectrum given by $\{24, [0]^9, [-4]^6\}$. There exist, however, graphs that are determined by their distance spectrum but are not determined
by their adjacency spectrum. In this context, \cite{JinZ2014} proves that complete multipartite graphs are uniquely determined by their distance
spectra (this result confirms a conjecture proposed in \cite{LinHW2013} and extends the special case established there for complete bipartite graphs).
A Tur\'{a}n graph is, in particular, determined by its distance spectrum \cite{JinZ2014}, and by its adjacency spectra (see
Theorem~\ref{theorem: Turan's graph is DS} here). On the other hand, while complete multipartite graphs are not generally determined by their
adjacency spectra (e.g., for complete bipartite graphs, see Theorem~\ref{thm:when CoBG is DS?}), they are necessarily determined by their distance
spectra \cite{JinZ2014}. Another family of graphs that are determined by their distance spectrum but are not determined by their adjacency spectrum,
named $d$-cube graphs, is provided in \cite{KoolenHI2016,KoolenHI2016b}. These graphs have their vertices represented by binary $n$-length tuples,
where any two vertices are adjacent if and only if their corresponding binary $n$-tuples differ in one coordinate; it is also shown that these graphs
have exactly three distinct distance eigenvalues. The question of which multipartite graphs, or graphs in general, are determined by their distance spectra remains open.

Another newly obtained proof presented in this paper refers to a necessary and sufficient condition in \cite{MaRen2010} for complete bipartite graphs
to be $\A$-DS (see Theorem~\ref{thm:when CoBG is DS?} and Remark~\ref{remark: complete bipartite graphs} here). These graphs are also bi-regular, and
the two possible vertex degrees can differ by more than one.
Both of these newly obtained proofs, discussed in Section~\ref{section: special families of graphs}, provide insights into the broader question of which
(structured) multipartite graphs are determined by their adjacency spectrum or, more generally, by the spectra of some of their associated matrices.

Even if Haemers' conjecture is eventually proved in its full generality, it remains surprising when a new family of structured graphs is shown to be DS
(or $X$-DS, more generally). This is because for certain structured graphs, such as strongly regular graphs and trees, their spectra often fail to uniquely
determine them \cite{vanDamH03,vanDamH09}. This stark contrast between the fact that almost all random graphs of high order are likely to be DS and the
existence of interesting structured graphs that are not DS has significant implications.

In addition to the questions posed earlier in this paper, we raise the following additional concrete question:
\begin{question}
{\em By Theorem~\ref{theorem: Berman's generalized friendship graph}, the family of generalized friendship graphs $F_{p,q}$ is ${\bf{\mathcal{L}}}$-DS.
Are these graphs also DS with respect to their other associated matrices?}
\end{question}

We speculate that the DS property of a graph correlates, to some extent, with the number of symmetries that the graph possesses, and we hypothesize
that the size of the automorphism group of a graph can partially indicate whether it is DS.

A justification for this claim is that the automorphism group of a graph reflects its symmetries, which can influence the eigenvalues of its adjacency
matrix. Highly symmetric graphs (i.e., those with large automorphism groups) often exhibit eigenvalue multiplicities and patterns that are shared by other
nonisomorphic graphs, making such graphs less likely to be DS. Conversely, graphs with trivial automorphism groups are typically less symmetric and may
have eigenvalues that uniquely determine their structure, increasing the likelihood that they are DS. As noted in \cite{GodsilR2001}, almost all graphs
have trivial automorphism groups. This observation aligns with the conjecture that most graphs are DS, as the absence of symmetry reduces the likelihood
of two nonisomorphic graphs sharing the same spectrum.

It is noted, however, that the DS property of graphs is not solely dictated by their automorphism groups. Specifically, a graph with a large
automorphism group can still be DS if its eigenvalues uniquely encode its structure (see Section~\ref{subsection: strongly regular graphs}).
In contrast to these DS graphs, other graphs with trivial automorphism groups are not guaranteed to be DS; in such cases, the spectrum might
not capture enough structural information to uniquely determine the graph. Typically, graphs with a small number of distinct eigenvalues
seem to be, in general, the hardest graphs to distinguish by their spectrum. As noted in \cite{vanDamO11}, it seems that most graphs with
few eigenvalues (e.g., most of the strongly regular graphs) are not determined by their spectrum, which served as one of the motivations of
the work in \cite{vanDamO11} in studying graphs whose normalized Laplacian has three eigenvalues.

To conclude, the size of the automorphism group of a graph can provide some indication of whether it is DS, but it is not a definitive criterion. While
large automorphism groups often correlate with the graph not being DS due to shared eigenvalues among nonisomorphic graphs, this is not an absolute rule.
Therefore, the claim should be understood as a general observation that requires qualification to account for exceptions.

\end{document}